\documentclass{article}
 \usepackage[utf8]{inputenc}
 \usepackage{amsmath}
 \usepackage{amsfonts}
 \usepackage{amssymb}
 \usepackage{graphicx}
 \usepackage{amsmath}
 \usepackage{amsthm}
 \usepackage{enumerate}
 \usepackage{verbatim}
 \usepackage{hyperref}
 \usepackage{lmodern,eurosym,ae,tabularx,amssymb,url,epsf,epsfig,amsmath,amsfonts,graphicx,dsfont}
 \usepackage{amsmath, amsfonts,amssymb}

 \usepackage{bbm}
 \usepackage{color}

\numberwithin{equation}{section}
\font\dsrom=dsrom10 scaled 1200

\def \N{\mathbb{N}}
\def \R{\mathbb{R}}
\def \E{\mathbb{E}}
\def \P{\mathbb{P}}

\def \m{\mathfrak{m}}
\newcommand\ind[1]{\textrm{\dsrom{1}}_{#1}}
\def \O{\mathcal{O}}
\def \L{\mathcal{L }}
\def \into{\int_{\mathcal{O}}}

\newcommand{\T}{\mathcal{T}}
\theoremstyle{plain} 
\newtheorem{theorem}{Theorem}[section]

\newtheorem{definition}[theorem]{Definition}

\newtheorem{lemma}[theorem]{Lemma}

\newtheorem{proposition}[theorem]{Proposition}
\newtheorem{remark}[theorem]{Remark}

\date{}

\begin{document}



\title{\bf Variational formulation of American option prices in the Heston Model}

\author{{\sc Damien Lamberton}\thanks{%
 Université Paris-Est, Laboratoire d'Analyse et de Mathématiques Appliquées (UMR 8050), UPEM, UPEC, CNRS, Projet Mathrisk INRIA, F-77454, Marne-la-Vallée, France - {\tt damien.lamberton@u-pem.fr}}\\
{\sc Giulia Terenzi}\thanks{%
Université Paris-Est, Laboratoire d'Analyse et de Mathématiques Appliquées (UMR 8050), UPEM, UPEC, CNRS, Projet Mathrisk INRIA, F-77454, Marne-la-Vallée, France, and Universit\`a di Roma Tor Vergata, Dipartimento di Matematica, Italy - {\tt terenzi@mat.uniroma2.it}}\\}
%
\maketitle
\begin{abstract}\noindent{\parindent0pt}
	We give an analytical characterization of the price function of an American option in Heston-type models. Our approach is based on variational inequalities and extends recent results of Daskalopoulos and Feehan (2011, 2016) and Feehan and Pop (2015). We study the existence and uniqueness of a weak solution of the associated degenerate parabolic obstacle problem. Then, we use suitable estimates on the joint distribution of the log-price process and the volatility process in order to characterize the analytical weak solution as the solution to the optimal stopping problem. We also rely on semi-group techniques and on the affine property of the model.
\end{abstract}


\smallskip


\section{Introduction}
The model introduced by S. Heston in 1993 (\cite{H}) is 
one of the most widely used stochastic volatility models in the financial world and it was the starting point for several more complex models 
which extend it. The great success of the
Heston model is due to the fact that the dynamics of the underlying asset can take into account the non-lognormal distribution of the asset returns and the observed
mean-reverting property of the volatility. Moreover, it remains analytically tractable and provides a closed-form valuation formula for European options using Fourier transform. 

These features have called for  an extensive literature on numerical methods to price derivatives in Heston-type models. In this framework, besides purely probabilistic methods such as standard Monte Carlo and tree approximations,  there is a large class of algorithms which  exploit numerical analysis techniques in order to solve the standard PDE (resp. the obstacle problem) formally associated with the European (resp. American) option price  function. 
However, these algorithms  have, in general, little mathematical support and in particular, as far as we know, a rigorous and complete study of the  analytic characterization of the American price function is not present in the literature.

The main difficulties in this sense come from the degenerate nature of the model. In fact, the infinitesimal generator associated with the two dimensional diffusion given by the log-price process and the volatility process is not uniformly elliptic: it degenerates on the boundary of the domain, that is when the volatility  variable vanishes. Moreover, it has unbounded coefficients with linear growth. Therefore, the existence and the uniqueness of the solution to the pricing PDE and obstacle problem do not follow from the classical theory, at least in the case in which the boundary of the state space is reached with positive probability, as happens in many cases of practical importance (see \cite{An}). Moreover, the probabilistic representation of the solution, that is the identification with the price function, is far from trivial in the case of non regular payoffs. 

It should be emphasized that a clear analytic characterization of the price function allows not only  to formally justify  the   theoretical convergence of some classical pricing algorithms but also to investigate the regularity properties of the price function (see  \cite{JLL} for the case of the Black and Scholes models).

Concerning the existing literature, E. Ekstrom and J. Tysk in  \cite{ET} give a rigorous and complete analysis of these issues in the case of European options, proving that, under some regularity assumptions on the payoff functions, the price function is the unique classical solution of the associated PDE with a certain boundary behaviour for vanishing values of the volatility.  However, the payoff functions they consider do not include the case of standard put and call options. 

Recently, P. Daskalopoulos and P. Feehan studied the existence, the uniqueness, and some regularity properties of the solution of this kind of degenerate PDE and obstacle problems in the elliptic case, introducing suitable weighted Sobolev spaces which clarify the behaviour of the solution near the degenerate boundary (see \cite{DF,DF2}). In another paper (\cite{FP}) P. Feehan and C. Pop addressed the issue of the probabilistic representation of the solution,  but we do not know if their assumptions on the solution of the parabolic obstacle problem are satisfied in the case of standard American options. 
Note that Feehan and Pop did prove regularity results in the elliptic case, see \cite{FP2}. They also announce results for the parabolic case in \cite{FP}.

The aim of this paper is to give a precise analytical characterization of the American option price function for a large class of payoffs which includes the standard put and call options. In particular, we give a variational formulation of the American pricing problem using the weighted Sobolev spaces and the bilinear form introduced in \cite{DF}. 
The paper is organized as follows. In Section 2, we introduce our notations and we state our main results. Then, in Section 3, we study the existence and uniqueness of the solution of the associated variational inequality, extending the results obtained in \cite{DF} in the elliptic case. The proof  relies, as in \cite{DF}, on the classical penalization technique introduced by Bensoussan and Lions \cite{BL} with some technical devices due to the degenerate nature of the problem. We also establish a Comparison Theorem. Finally, in section 4, we prove that the solution of the variational inequality  with obstacle function $\psi$ is actually the American option price function with payoff $\psi$, with  conditions on $\psi$ which are satisfied, for example, by the standard call and put options. In order to do this, we use  the affine property  of the underlying diffusion given by the log price process $X$ and the volatility process $Y$. Thanks to this property, we first identify the analytic semigroup associated with the bilinear form with a correction term  and the transition semigroup of the pair $(X,Y)$ with a killing term. Then, we prove regularity results on the solution of the variational inequality and suitable estimates on the joint law  of the process $(X,Y)$ and we deduce from them the analytical characterization of the solution of the optimal stopping problem, that is the  American option price.
\tableofcontents
\section{Notations and main results}
\subsection{The Heston model}
	We recall that in the Heston model the dynamics under the pricing measure of the asset price $S$ and the volatility process $Y$ are governed by the stochastic differential equation system
	\begin{equation*}
	\begin{cases}
	\frac{dS_t}{S_t}=(r-\delta)dt + \sqrt{Y_t}dB_t,\qquad&S_0=s>0,\\ dY_t=\kappa(\theta-Y_t)dt+\sigma\sqrt{Y_t}dW_t, &Y_0=y\geq 0,
	\end{cases}
	\end{equation*}
where $B$ and $W$ denote two correlated Brownian motions with 
	$$
	d\langle B,W \rangle_t=\rho dt, \qquad \rho \in (-1,1).
	$$
	We exclude  the  degenerate case $\rho=\pm 1$, that is the case in which the same Brownian motion drives the dynamics of $X$ and $Y$. Actually, it can be easily seen that, in this case, $S_t$ reduces to  a function of the pair $(Y_t,\int_0^tY_s ds)$ and the resulting degenerate model  cannot be treated with the techniques we develop in this paper. Moreover, this particular situation is not very interesting from a financial point of view.
	
	Here $r\geq0$ and $\delta\geq0$ are respectively  the risk free rate of interest and the continuous dividend rate. The dynamics of $Y$ follows a CIR process with mean reversion rate $\kappa>0$ and long run state $\theta> 0$. The parameter $\sigma>0$ is called the volatility of the volatility. Note that we do not require the Feller condition $2\kappa\theta \geq \sigma^2$: the volatility process $Y$  can hit $0$ (see, for example, \cite[Section 1.2.4]{Abook}).
	
	We are interested in studying the price of an American option with payoff function $\psi$. For technical reasons which will be clarified later on, hereafter we consider the process
	\begin{equation}\label{traslation}
	X_t=\log S_t-\bar{c}t,\qquad 	\mbox{with }	\bar{c}=r-\delta- \frac{\rho\kappa\theta}{\sigma},
	\end{equation}
which satisfies
	\begin{equation}\\
	\begin{cases}\label{hest}
dX_t=\big(\frac{\rho\kappa\theta}{\sigma}-\frac {Y_t} 2 \big)dt + \sqrt{Y_t}dB_t,\\dY_t=\kappa(\theta-Y_t)dt+\sigma\sqrt{Y_t}dW_t .
	\end{cases}
	\end{equation}

	 Note that, in this framework,  we have to consider payoff functions  $\psi$ which depend on both the time and the space variables. For example, in the case of a standard put option (resp. a call option) with strike price $K$ we have $ \psi(t,x)=(K-e^{x+\bar c t})_+ $ (resp. $ \psi(t,x)=(e^{x+\bar c t}-K)_+ $). So, the natural price at time $t$ of an American option with a nice enough payoff $(\psi(t,X_t,Y_t))_{0\leq t\leq T}$ is given by $P(t,X_t,Y_t)$, with
	\begin{equation*}
	P(t,x,y)=\sup_{\theta \in\mathcal{T}_{t,T}}\E[e^{-r(\theta-t )}\psi(\theta,{X}^{t,x,y}_\theta, Y^{t,y}_\theta)],
	\end{equation*}  
	where $\mathcal{T}_{t,T}$ is the set of all stopping times with values in $[t,T]$ and $(X^{t,x,y}_s, Y^{t,y}_s)_{t\leq s\leq T}$ denotes the solution to \eqref{hest} with the starting condition $(X_t,Y_t) = (x,y)$.
	
Our aim is to give an analytical characterization of the price function $P$.
	We recall that the infinitesimal generator of the two dimensional diffusion $(X,Y)$ is given by
	\begin{equation*}
	\mathcal{L}= \frac y 2 \left( \frac{\partial^2}{\partial x^2 } +2 \rho \sigma  \frac{\partial^2}{\partial y \partial x }+  \sigma^2  \frac{\partial^2}{\partial y^2 }  \right) + \left(\frac{\rho\kappa\theta}{\sigma}-\frac y 2\right)\frac{\partial}{\partial x}+\kappa(\theta - y)\frac{\partial}{\partial y} ,
	\end{equation*}
which is	defined on the open set $\mathcal{O}:= \R \times (0,\infty) $. Note that  $\mathcal{L}$ has unbounded coefficients and  is not uniformly elliptic: it degenerates on the boundary $\partial\mathcal{O}=\R\times\{0\}$.
	\subsection{American options and variational inequalities}
	\subsubsection{Heuristics}
		From the optimal stopping theory, we know that the discounted price process $\tilde{P}(t,X_t,Y_t)=e^{-rt}P(t,X_t,Y_t)$ is a supermartingale  and that its 
		finite variation part only decreases on the set $P=\psi$ with respect to the time variable $t$. We want to have an analytical interpretation of these features on the function $P(t,x,y)$.
	So, assume that $P \in\mathcal{C}^{1,2} ((0,T)\times \mathcal{O})$. Then, by applying It\^{o}'s formula, the finite variation part of $\tilde{P}(t,X_t,Y_t)$ is 
		$$
	 \left( \frac{\partial \tilde P }{\partial t } + \mathcal{L}\tilde P\right)(t,X_t,Y_t) .
		$$
		Since $\tilde{P}$ is a supermartingale, we can deduce the inequality 
		$$
		\frac{\partial \tilde P }{\partial t } + \mathcal{L}\tilde P\leq 0
		$$
		and, since its  finite variation part decreases only on the set $P(t,X_t,Y_t)=\psi(t,X_t,Y_t)$, we can write
		$$
		\left(  \frac{\partial \tilde P }{\partial t } + \mathcal{L}\tilde P\  \right) (\psi- P)=0.
		$$
		This relation has to be satisfied $dt-a.e.$ along the trajectories of $(t,X_t,Y_t)$. Moreover, we have the two trivial conditions $P(T,x,y)=	\psi(T,x,y)$ and $P\geq \psi$.
		
		The previous discussion is only heuristic, since the price function $P$ is not regular enough to apply  It\^{o}'s formula. However, it suggests the following strategy:
		\begin{enumerate}
			\item Study the obstacle problem
			\begin{equation} \label{system}
			\begin{cases}
			\frac{\partial u }{\partial t } + \mathcal{L}u\leq 0, \qquad u\geq \psi, \qquad &in \  [0,T]\times \mathcal{O},\\
			\left(  \frac{\partial u }{\partial t } + \mathcal{L}u  \right) (\psi-u)=0, \qquad &in \  [0,T]\times \mathcal{O},\\
			u(T,x,y)=\psi(T,x,y).
			\end{cases}
			\end{equation}
			\item Show that the discounted price function $\tilde P$ is equal to the solution of \eqref{system} where $\psi$ is replaced by $\tilde \psi(t,x,y)=e^{-rt}\psi(t,x,y)$.
		\end{enumerate}
We will follow this program providing a variational formulation of system \eqref{system}.
\subsubsection{Weighted Sobolev spaces and  bilinear form associated with the Heston operator} 
We consider the measure first introduced in \cite{DF}:
\begin{equation*}\label{m}
\mathfrak{m}_{\gamma, \mu}(dx,dy)=y^{\beta - 1}e^{-\gamma |x|- \mu y } dx dy, 
\end{equation*}
with $\gamma>0, \ \mu >0 \mbox{ and } \beta:= \frac{2 \kappa \theta}{\sigma^2}$.

It is worth noting  that in \cite{DF} the authors fix  $\mu=\frac{2\kappa}{\sigma^2}$ in the definition of the measure $\m_{\gamma,\mu}$. This specification will not be necessary in this paper, but it is useful to mention it in order to better understand how this measure arises. In fact, recall that the density of the speed measure of the CIR process is given by $y^{\beta-1}e^{-\frac{2\kappa}{\sigma^2}y}$. Then,  the term $y^{\beta-1}e^{ -\frac{2\kappa}{\sigma^2}y}$ in the definition of $\m_{\gamma,	\mu}$ has a clear probabilistic interpretation, while the exponential term $e^{-\gamma |x|}$ is classically introduced  just to deal with the unbounded domain in the $x-$component.


For $u\in \R^n$ we denote by $|u|$ the standard Euclidean norm of $u$ in $\R^n$. Then, we recall the weighted  Sobolev spaces introduced in \cite{DF}. The choice of these particular Sobolev spaces will  allow us to formulate the obstacle problem \eqref{system} in a variational framework with respect to the measure $\m_{\gamma,\mu}$. 
\begin{definition}\label{def-sob}
	For every $p\geq 1$, let $L^p(\mathcal{O},\mathfrak{m}_{\gamma, \mu})$ be the space of all Borel measurable functions $u: \mathcal{O} \rightarrow \R $ for which 
	$$
	\Vert u \Vert^p_{L^p(\mathcal{O},\mathfrak{m}_{\gamma, \mu})} := \int_{\mathcal{O}} |u|^p d\mathfrak{m}_{\gamma, \mu}  < \infty, 
	$$
	and denote $H^0(\mathcal{O},\mathfrak{m}_{\gamma, \mu}):=L^2(\mathcal{O},\mathfrak{m}_{\gamma, \mu}).$
	
	\begin{enumerate}
		\item If $\nabla u:=(u_x,u_y)$ and $u_x$, $u_y$ are defined in the sense of distributions, we set 
		$$
		H^1(\mathcal{O},\mathfrak{m}_{\gamma, \mu}):=\{  u \in  L^2(\mathcal{O},\mathfrak{m}_{\gamma, \mu}): \sqrt{1+y}  u \mbox{ and }  \sqrt{y}|\nabla  u| \in L^2(\mathcal{O},\mathfrak{m}_{\gamma, \mu}) \},
		$$
	 and 
		$$
		\Vert u \Vert^2_{H^1(\mathcal{O},\mathfrak{m}_{\gamma, \mu})} := \int_{\mathcal{O}} \left(  y|\nabla  u|^2 + (1+y)u^2 \right) d\mathfrak{m}_{\gamma, \mu}.
		$$
		\item If $D^2u:=(u_{xx},u_{xy},u_{yx}, u_{yy})$ and all derivatives of $u$ are defined in the sense of distributions, we set 
		$$
		H^2(\mathcal{O},\mathfrak{m}_{\gamma, \mu}):=\{  u \in  L^2(\mathcal{O},\mathfrak{m}_{\gamma, \mu}): \sqrt{1+y}  u, \ (1+y)|\nabla  u|, \   y|D^2u| \in L^2(\mathcal{O},\mathfrak{m}_{\gamma, \mu}) \}
		$$
		and 
		$$
		\Vert u \Vert^2_{H^2(\mathcal{O},\mathfrak{m}_{\gamma, \mu})} := \int_{\mathcal{O}} \left(  y^2|D^2u|^2 + (1+y)^2|\nabla  u|^2+(1+y)u^2 \right) d\mathfrak{m}_{\gamma, \mu} .
		$$
	\end{enumerate}
\end{definition}
%
%

	For brevity and when the context is clear, we shall often denote
	$$
	H:= H^0(\mathcal{O}, \mathfrak{m}_{\gamma, \mu}), \qquad V:= H^1(\mathcal{O}, \mathfrak{m}_{\gamma, \mu}) 
	$$
	and
	$$\Vert
	u\Vert_H :=\Vert u \Vert_{L^2(\mathcal{O}, \mathfrak{m}_{\gamma, \mu}) }, \qquad \Vert u \Vert_V :=\Vert u \Vert_{H^1(\mathcal{O}, \mathfrak{m}_{\gamma, \mu}) }.
	$$
	Note that we have the inclusion 
	$$
	H^2(\mathcal{O}, \mathfrak{m}_{\gamma, \mu}) \subset H^1(\mathcal{O}, \mathfrak{m}_{\gamma, \mu})
	$$
	and that the spaces $H^k(\mathcal{O}, \mathfrak{m}_{\gamma, \mu})$, for $k=0,1,2$ are Hilbert spaces with the inner products
	$$
	(u,v)_H=	(u,v)_{L^2(\mathcal{O}, \mathfrak{m}_{\gamma, \mu}) }=\int_{\mathcal{O}}^{}uv d \m_{\gamma, \mu} ,
	$$ 
	$$
	(u,v)_V=(u,v)_{H^1(\mathcal{O}, \mathfrak{m}_{\gamma, \mu})} =\int_{\mathcal{O}}^{}\left(y \left( \nabla  u,\nabla  v\right) + (1+y)uv \right)d\m_{\gamma, \mu} 
	$$
	and
	$$
	(u,v)_{H^2(\mathcal{O}, \mathfrak{m}_{\gamma, \mu})} :=\int_{\mathcal{O}}^{}\left(y^2\left( D^2u,D^2v\right) + (1+y)^2 \left( \nabla  u,\nabla  v\right)+ (1+y)uv \right) d \m_{\gamma, \mu},
	$$
	where $(\cdot,\cdot)$ denotes the standard scalar product in $\R^n$.
	
Moreover, for every  $T>0, \,p\in [1,+\infty)$ and $i=0,1,2$, we set
\begin{align*}
	L^p([0,T];H^i(\O,\m_{\gamma,\mu}))=\bigg\{& u:[0,T]\times \O\rightarrow \R \mbox{ Borel measurable}: u(t,\cdot,\cdot) \in H^i(\O,\m_{\gamma,\mu})\mbox{ for a.e. } t\in [0,T] \\&\mbox{ and }	\int_0^T \Vert u(t,\cdot.\cdot)\Vert^p_{H^i(\O,\m_{\gamma,\mu})}dt<\infty\bigg\}
\end{align*}
	and
	$$
	\Vert u\Vert^p_{L^p([0,T];H^i(\O,\m_{\gamma,\mu}))}= 	\int_0^T \Vert u(t,\cdot.\cdot)\Vert^p_{H^i(\O,\m_{\gamma,\mu})}dt.
	$$
We also define $L^\infty([0,T];H^i)$ with the usual essential sup norm.

	We can now introduce the following bilinear form.
	\begin{definition}
		For any $u,v \in H^1(\mathcal{O}, \mathfrak{m}_{\gamma, \mu})$ we define the bilinear form
		\begin{align*}
		a_{\gamma, \mu}(u,v) 
		=&  \frac 1 2 \int_{\mathcal{O}} y\left(  u_xv_x(x,y)  +\rho\sigma u_xv_y(x,y) + \rho\sigma u_yv_x(x,y) +\sigma^2 u_y v_y(x,y ) \right)d\m_{\gamma, \mu} \\
		&+\into y \left(  j_{\gamma,\mu}(x)u_x(x,y) + k_{\gamma,\mu}  (x)u_y(x,y) \right) v(x,y) d\m_{\gamma, \mu},
		\end{align*}
		where
		 \begin{equation}\label{JandGamma}
		j_{\gamma,\mu}(x)=\frac 1 2 \left( 1  -\gamma \mbox{sgn}(x)-\mu\rho\sigma \right), \qquad k_{\gamma,\mu}(x)=\kappa-\frac{\gamma\rho\sigma}{2} \mbox{sgn}(x) -\frac{\mu\sigma^2}{2}  . 
		\end{equation}
		\end{definition}
		We will prove that $a_{\gamma,\mu}$ is the bilinear form associated with the operator $\L$, in the sense that for every $ u\in H^2(\mathcal{O},\m_{\gamma,\mu})$ and for every $ v\in H^1(\mathcal{O},\m_{\gamma,\mu})$, we have
		\begin{equation*}
		(\mathcal{L}u,v)_H=-a_{\gamma, \mu}(u,v).
		\end{equation*}
		In order to simplify the notation,  for the rest of this paper we will write   $	\m$ and $a(\cdot,\cdot)$ instead of $\m_{\gamma, \mu}$ and $a_{\gamma, \mu}(\cdot,\cdot)$ every time the dependence on $\gamma$ and $\mu$ does not play a role in the analysis and computations.
	\subsection{Variational formulation of the American price}
	Fix $T>0$. We consider an assumption on the payoff function $\psi$ which will be crucial in the discussion of the penalized problem.

\medskip
\noindent	
\textbf{Assumption $\mathcal{H}^1$.}
We say that a function $\psi$ satisfies Assumption $\mathcal{H}^1$ if	$\psi\in \mathcal{C}([0,T];H) $, $\sqrt{1+y}\psi \in L^2([0,T];V)$, $\psi(T)\in V$ and there exists  $\Psi\in L^2([0,T];V)$ such that $\left| \frac{\partial \psi}{\partial t} \right|\leq \Psi$. 

\medskip

We will also need a domination condition on $\psi$ by a function $\Phi$ which satisfies the following assumption. 

\medskip
\noindent	
\textbf{Assumption $\mathcal{H}^2$.}
We say that a function $\Phi\in L^2([0,T];H^2(\mathcal{O},\m))$ satisfies Assumption $\mathcal{H}^2$   if  
 $(1+y)^{\frac 32}\Phi\in L^2([0,T];H)$, $\frac{\partial \Phi}{\partial t}+ \mathcal{L}\Phi \leq 0$ and $\sqrt{1+y}\Phi \in L^\infty([0,T]; L^2(\mathcal{O}, 	\m_{\gamma,\mu'}))$ for some $0<\mu'<\mu$.

\medskip
The domination condition is needed to deal with the lack of coercivity of the bilinear form associated with our problem. Similar conditions are also used in \cite{DF}.

The first step in the variational formulation of the problem is to introduce the associated variational inequality and 
to prove the following existence and uniqueness result.
\begin{theorem}\label{variationalinequality}
Assume that $	\psi$ satisfies Assumption $\mathcal{H}^1$ together with $0\leq \psi\leq \Phi$, where $\Phi$ satisfies Assumption $\mathcal{H}^2$.
	Then, there exists a unique function $u$ such that $	u\in\mathcal{C}([0,T];H)\cap L^2([0,T];V),\,\frac{\partial u}{\partial t } \in L^2([0,T];  H)$ and 
	\begin{equation}\label{VI}
	\begin{cases}
-\left( \frac{\partial u}{\partial t },v -u  \right)_H + a(u,v-u)\geq 0, \quad \mbox{a.e. in } [0,T] \quad v\in L^2([0,T];V), \ v\geq \psi,\\
u\geq \psi \mbox{ a.e. in } [0,T]\times \R \times (0,\infty),\\
	u(T)=\psi(T),\\
	0\leq u \leq \Phi.
	\end{cases}
	\end{equation}
	\end{theorem}
The proof is presented in Section 3 and essentially relies on the penalization technique introduced by Bensoussan and Lions  (see also \cite{F})  with some technical devices due to the degenerate nature of the problem. We extend in the parabolic framework the  results obtained in \cite{DF} for the elliptic case.  

The second step is to identify the unique solution of the variational inequality \eqref{VI} as the solution of the optimal stopping problem, that is the (discounted) American option price. In order to do this, we consider the following  assumption on the payoff function.

\medskip
\noindent	
\textbf{Assumption $\mathcal{H}^*$.}
We say that a function $\psi:[0,T]\times \R\times [0,\infty)\rightarrow \R$ satisfies Assumption $\mathcal{H}^*$   if  $\psi$ is continuous and there exist constants $C>0$ and $L\in \left[0,\frac{2\kappa}{\sigma^2}\right)$ such that, for all $(t,x,y)\in[0,T]\times \R\times [0,\infty)$,
\begin{equation}\label{boundonpsi}
0\leq \psi(t,x,y)\leq C(e^{x}+e^{Ly}),
\end{equation}
 and
\begin{equation}\label{boundonderivatives}
\left|\frac{\partial \psi}{\partial t}(t,x,y)\right|+\left|\frac{\partial \psi}{\partial x}(t,x,y)\right|+ \left|\frac{\partial \psi}{\partial y}(t,x,y)\right|\leq C(e^{a|x|+by}),
\end{equation}
for some $a,b\in \R$.
\medskip

Note that  the payoff functions of a standard  call  and put option with strike price $K$ (that is, respectively, $\psi=\psi(t,x)=(K-e^{x+\bar c t})_+ $ and $\psi=\psi(t,x)=(e^{x+\bar c t}-K)_+$) satisfy Assumption $\mathcal H^*$.
Moreover, it is easy to see that, if $\psi$ satisfies Assumption $\mathcal{H}^*$, then it is possible to choose $\gamma$ and $\mu$ in the definition of the measure $\m_{\gamma,\mu}$  (see \eqref{m})  such that $\psi$ satisfies  the assumptions of Theorem \ref{variationalinequality}. Then, for such $\gamma$ and $\mu$, we get the following identification result.
	\begin{theorem}	\label{theorem2}
	Assume that $\psi$ satisfies Assumption $\mathcal{H}^*$
Then, the solution $u$ of the variational inequality \eqref{VI} associated with $\psi$ is continuous and coincides with  the function $u^*$ defined by 
\begin{equation*}\label{americanprice}
		u^*(t,x,y)= \sup_{\tau \in \mathcal{T}_{t,T}}\E \left[   \psi(\tau,X_\tau^{t,x,y}, Y_\tau^{t,y})     \right].
\end{equation*}

	\end{theorem}

\section{Existence and uniqueness of solutions to the variational inequality}

\subsection{Integration by parts and energy estimates}
The following result justifies the definition of the bilinear form $a$.
	\begin{proposition}\label{Prop_integrationbyparts}
		If $u\in H^2(\mathcal{O},\m)$ and $v\in H^1(\mathcal{O},\m)$, we have
		\begin{equation}\label{A=a}
		(\mathcal{L}u,v)_H=-a(u,v).
		\end{equation}
	\end{proposition}
This result is proved with the same arguments of  \cite[Lemma 2.23]{DF} or \cite[Lemma A.3]{DF2} but  we prefer to repeat here  the proof  since it clarifies why we have considered the process $X_t=\log S_t- \bar c t$ instead of the standard log-price process $\log S_t$. 

We first need the  following result which justifies the integration by parts formulas with respect to the measure $\m$. 	The proof follows standard approximation techniques, so we omit it (see  the proof of  \cite[Lemma 2.23]{DF} or \cite{T} for the details).
	\begin{lemma}\label{IPP1}
		Let us consider $u,v: \mathcal{O} \rightarrow \R$ locally square-integrable on $\mathcal{O}$. Then, if the derivatives $u_x$ and $v_x$ are locally square-integrable on $\mathcal{O}$  and
		$$
		\into \big(| u_x(x,y)v(x,y)|+ | u(x,y)v_x(x,y)|+ |u(x,y)v(x,y)| \big) d\m <\infty,
		$$ we have
		\begin{equation*}\label{IPPx}
		\into u_x(x,y) v(x,y)d\m=-\int_\mathcal{O} u(x,y) \left(v_x(x,y)-\gamma sgn(x)v\right)d\m.
		\end{equation*}
Similarly, if the derivatives $u_y$ and $v_y$ are locally square-integrable on $\mathcal{O}$ and 
		$$
		\into y\big(| u_y(x,y)v(x,y)|+ | u(x,y)v_xy(x,y)|\big)+ |u(x,y)v(x,y)|  d\m <\infty,
		$$  we have
		\begin{equation*}\label{IPPy}
		\int_\mathcal{O} y u_y(x,y) v(x,y)d\m=-\int_\mathcal{O}y u(x,y) v_y(x,y)d\m- \int_{\mathcal{O}} (\beta-\mu y)u(x,y) v(x,y) d\m.
		\end{equation*}
	\end{lemma}
	We can now prove Proposition  \ref{Prop_integrationbyparts}.	
	\begin{proof}[Proof of Proposition \ref{Prop_integrationbyparts}]
	By using Lemma \ref{IPP1}  we have
	\[
	\into y \frac{\partial^2 u}{\partial x^2}vd\m=-\into  y\frac{\partial u}{\partial x}\left(\frac{\partial v}{\partial x}
	-\gamma sgn(x) v\right)d\m,
	\]
	\[
	\into  y \frac{\partial^2 u}{\partial y^2}vd\m=-\into y \frac{\partial u}{\partial y}\frac{\partial v}{\partial y}d\m
	+\into  (\mu y -\beta)\frac{\partial u}{\partial y} vd\m,
	\]
	\[
	\into y \frac{\partial^2 u}{\partial x\partial y}vd\m=-\into y\frac{\partial u}{\partial y}\left(\frac{\partial v}{\partial x}
	-\gamma sgn(x) v\right)d\m
	\]
and
	\[
	\into  y \frac{\partial^2 u}{\partial x\partial y}vd\m=-\into y \frac{\partial u}{\partial x}\frac{\partial v}{\partial y}d\m
	+\into  (\mu y -\beta)\frac{\partial u}{\partial x} vd\m.
	\]
	Recalling  that 
	\[
\mathcal{L} =\frac{y}{2}\left(\frac{\partial^2 }{\partial x^2}+2\rho\sigma\frac{\partial^2 }{\partial x\partial y}+\sigma^2\frac{\partial^2 }{\partial y^2}\right)
	+\left(\frac{\rho\kappa\theta}{\sigma}-\frac{y}{2}\right)\frac{\partial}{\partial x}
	+\kappa(\theta-y) \frac{\partial}{\partial y}
	\]
and using the equality $\beta=2\kappa\theta/\sigma^2$, we get
	\begin{align*}
		(\mathcal{L} u,v)_H&=-\into \frac{y}{2}\left(\frac{\partial u}{\partial x}\frac{\partial v}{\partial x}+
		\sigma^2 \frac{\partial u}{\partial y}\frac{\partial v}{\partial y}+
		\rho\sigma \frac{\partial u}{\partial x}\frac{\partial v}{\partial y}+
		\rho\sigma \frac{\partial u}{\partial y}\frac{\partial v}{\partial x}\right) d\m+\into \frac{1}{2}\frac{\partial u}{\partial x}\left(y\gamma sgn(x) +{\rho\sigma}(\mu y -\beta)\right)vd\m\\
		&\qquad
		+\into\frac{1}{2}\frac{\partial u}{\partial y}\left(\mu \sigma^2y -\beta\sigma^2+ {\rho\sigma}y\gamma sgn(x)\right) vd\m+\into \left[\left(\frac{\rho\kappa\theta}{\sigma}-\frac{y}{2}\right)\frac{\partial u}{\partial x}
		+\kappa(\theta-y) \frac{\partial u}{\partial y}\right]
		vd\m\\&=-a(u,v).
	\end{align*}
	\end{proof}

	\begin{remark}\label{rem_trasl}
	By a closer look at the proof of Proposition \ref{Prop_integrationbyparts}	it is clear that the choice of $\bar c$  in \eqref{traslation} allows  to avoid terms of the type $\int  (u_x+u_y) v d\m$ in the associated bilinear form $a$. This trick will be crucial in order to obtain suitable energy estimates.
	\end{remark}
		 Recall the well-known inequality 
		 \begin{equation}\label{inequality}
		 bc=(\sqrt{\zeta}b) \left(\frac c {\sqrt{\zeta}}\right)  \leq \frac \zeta 2  b^2+\frac 1 {2\zeta} c^2, \qquad \ b,c \in \R, \ \zeta >0.
		 \end{equation}
		 Hereafter we will often apply \eqref{inequality} in the proofs even if it is not explicitly recalled each time.
		 
We have the following energy estimates.
		 \begin{proposition}\label{energyestimates}
		 For every $u,v\in V$, the bilinear form $a(\cdot,\cdot)$ satisfies 
		 	\begin{equation}	\label{lb}
		 	|a(u,v)|\leq C_1 \Vert u \Vert_V\Vert v \Vert_V,
		 	\end{equation}
		 	\begin{equation}\label{ub}
		 	a(u,u) \geq  C_2 \Vert u \Vert^2_V -C_3\Vert (1 + y)^{\frac 1 2 }u\Vert^2_H,
		 	\end{equation}
		 	where $$C_1=\delta_0+K_1,\quad C_2=\frac{\delta_1} 2, \quad C_3= \frac{\delta_1}{2}+\frac {K_1^2} {2\delta_1}, $$
		with
		 	 	\begin{equation}\label{delta0}
		 	 	\delta_0=\sup_{s_1^2+t_1^2>0, \;s_2^2+t_2^2>0}\frac{|s_1s_2+\rho\sigma s_1t_2+\rho\sigma s_2 t_1+\sigma^2 t_1t_2|}{2\sqrt{(s_1^2+t_1^2)(s_2^2+t_2^2)}},
		 	 	\end{equation}
		 	 	\begin{equation}\label{delta1}
		 	 	\delta_1=\inf_{s^2+t^2>0}\frac{s^2+2\rho\sigma st+\sigma^2 t^2}{2(s^2+t^2)},
		 	 	\end{equation}
		 	 	and
		 	 	\begin{equation}\label{K1}
		 	 	K_1=\sup_{x\in \R}\sqrt{j^2_{\gamma,\mu}(x)+k^2_{\gamma,\mu}(x)}.
		 	 	\end{equation}
		 	 \end{proposition}
		 	 It is easy to see  that the constants $\delta_0,\delta_1$ and $K_1$ defined in \eqref{delta0} and \eqref{K1} are positive and finite (recall that
	 the functions 	 $j_{\gamma,\mu}=j_{\gamma,\mu}(x)$  and $k_{\gamma,\mu}=\kappa_{\gamma,\mu}(x)$ defined in \eqref{JandGamma} are bounded). 
	 
	 These energy estimates  were already proved in  \cite[Lemma 2.40]{DF} with a very similar statement. Here we repeat the proof for the sake of completeness, since we will refer to it later on.
		 \begin{proof}[Proof of Proposition \ref{energyestimates}]
		 	In order to prove \eqref{ub}, we note that
		 	\begin{align*}
		 	&	 \frac 1 2 \int_{\mathcal{O}} y\left(  u_xv_x+\rho\sigma u_xv_y + \rho\sigma u_yv_x +\sigma^2 u_y v_y \right)d\m  \geq \delta_1 \into y|\nabla u|^2 d\m.
		 	\end{align*}
	Therefore
		 	\begin{align*}
		 	a(u,u)&\geq\delta_1 \into y|\nabla u|^2 d\m
		 	-		  K_1  \into y|\nabla u| |u| d\m  \\&\geq\delta_1 \into y|\nabla u|^2 d\m
		 	- \frac{K_1\zeta} 2 \int_{\mathcal{O}}y|\nabla u|^2   
	d\m  -  \frac {K_1} {2\zeta} \int_{\mathcal{O}} (1+y) u^2d\m  \\
		 	&= \left(\delta_1-\frac{K_1\zeta} 2 		\right)\int_{\mathcal{O}} \left(y |\nabla   u |^2 + (1+y)u^2\right) d\m -\left( \delta_1-\frac{K_1\zeta} 2+\frac {K_1} {2\zeta}  \right)\into (1+y)u^2 d\m.
		 	\end{align*}	
		 	The assertion then follows by choosing $\zeta=\delta_1/K_1$.  \eqref{lb} can be proved in a similar way.
		 \end{proof}
	 \subsection{Proof of Theorem \ref{variationalinequality} }
		 	Among the standard assumptions required in \cite{BL}  for the penalization procedure, there are the  coercivity and  the boundedness of the coefficients.  In the Heston-type models these assumptions are no longer satisfied and this leads to some technical difficulties. 
		 	In order to overcome them, we introduce some auxiliary operators.

		 	From now on, we set 
		 	$$
		 	a(u,v)=\bar{a}(u,v)+	\tilde{a}(u,v),
		 	$$
		 	where
		 	\begin{eqnarray*}
		 		\bar{a}(u,v)&=&\into\frac{y}{2}\left(\frac{\partial u}{\partial x}\frac{\partial v}{\partial x}+
		 		\rho\sigma \frac{\partial u}{\partial x}\frac{\partial v}{\partial y}+
		 		\rho\sigma \frac{\partial u}{\partial y}\frac{\partial v}{\partial x}+
		 		\sigma^2 \frac{\partial u}{\partial y}\frac{\partial v}{\partial y}\right) d\m,\\
		 		\tilde{a}(u,v)&=&\into y\frac{\partial u}{\partial x}j_{\gamma,\mu}vd\m
		 		+  \into y\frac{\partial u}{\partial y}k_{\gamma,\mu} vd\m.
		 	\end{eqnarray*}
		 	Note that $\bar a$ is symmetric.
As in the proof of Proposition \eqref{energyestimates} we have, for every $u, v\in V$,
		 	\begin{eqnarray*}
		 		|\bar{a}(u,v)|&\leq& \delta_0 \into y|\nabla u||\nabla v| d\m ,
		 	\end{eqnarray*}
		 	\begin{eqnarray*}
		 		\bar{a}(u,u)&\geq& \delta_1 \into y|\nabla u|^2 d\m,
		 	\end{eqnarray*}
and
		 	\[
		 	|\tilde{a}(u,v)|\leq K_1  \into y|\nabla u||v| d\m,
		 	\]
with $\delta_0,\,\delta_1$ and $K_1$ defined in Proposition \ref{energyestimates}.
		 	Moreover,  for $\lambda\geq 0$ and  $M>0$ we consider the bilinear forms
		 	\begin{eqnarray*}
		 		a_\lambda(u,v)&=&{a}(u,v)+\lambda \into (1+y)uv d\m,\\
		 		\bar{a}_\lambda(u,v)&=&\bar{a}(u,v)+\lambda \into (1+y)uv d\m,\\
		 		\tilde{a}^{(M)}(u,v)&=&\into (y\wedge M)\left(\frac{\partial u}{\partial x}j_{\gamma,\mu}
		 		+ \frac{\partial u}{\partial y}k_{\gamma,\mu}\right) vd\m
		 	\end{eqnarray*}
		 and
		 	\begin{eqnarray*}
		 		a^{(M)}_\lambda(u,v)&=& \bar{a}_\lambda(u,v)+\tilde{a}^{(M)}(u,v).
		 	\end{eqnarray*}
		 	The operator $a_\lambda$  was introduced in \cite{DF} to deal with the lack of coercivity of the bilinear form $a$, while the introduction of the truncated operator $a^{(M)}_\lambda$ with $M>0$ will be useful in order to overcome the technical difficulty related to the unboundedness of the coefficients.
		 	\begin{lemma}\label{coercivity}
		 		Let $\delta_0,\, \delta_1, \, K_1$ be defined as in \eqref{delta0}, \eqref{delta1} and \eqref{K1} respectively.
For any fixed $\lambda\geq \frac{\delta_1}{2}+\frac{K_1^2}{2\delta_1}$ the bilinear forms $a_{\lambda}$ and $a_{\lambda}^{(M)}$ are continuous and coercive. More precisely, we have
\begin{equation}
|a_{\lambda}(u,v)|\leq C\Vert u \Vert_V \Vert v \Vert_V, \qquad  u,v \in V,
\end{equation}
				 		\begin{equation}
				 		a_{\lambda}(u,u) \geq \frac{\delta_1} 2\Vert u \Vert_V^2, \qquad  u\in V,
				 		\end{equation}
				 		and
				 		\begin{equation}
				 		|a_{\lambda}^{(M)}(u,v)|\leq C\Vert u \Vert_V \Vert v \Vert_V, \qquad u,v \in V,
				 		\end{equation}
				 		\begin{equation}
				 		a_{\lambda}^{(M)}(u,u) \geq \frac{\delta_1} 2 \Vert u \Vert_V^2, \qquad  u\in V.
				 		\end{equation}
				 		where $C=\delta_0+K_1+\lambda$.
				 	\end{lemma}
			\begin{proof}
	The proof for the bilinear form  $a_\lambda$ follows as in \cite[Lemma 3.2]{DF}. We give the details for  $a_\lambda^{(M)}$ to check that the constants do not depend on $M$. Note that, for every $u,v\in V$,
		 	\begin{equation*}
		 		|\tilde{a}^{(M)}(u,v)|\leq K_1  \into y |\nabla u||v| d\m,
		 	\end{equation*}
		 	so that by straightforward computations we get
		 	\begin{align*}
		 	|a^{(M)}_\lambda(u,v)|
		\leq (\delta_0+\lambda+K_1)\|u\|_V\|v\|_V.
		 	\end{align*}
		 	On the other hand, for every $\zeta>0$,
		 	\begin{eqnarray*}
		 		a^{(M)}_\lambda(u,u)&\geq& \delta_1 \into y|\nabla u|^2 d\m+\lambda
		 		\into (1+y)u^2 d\m-K_1 \into y|\nabla u||u| d\m\\
		 		&\geq&\left(\delta_1-\frac{K_1\zeta}{2}\right)\into y|\nabla u|^2 d\m
		 		+\left(\lambda -\frac{K_1}{2\zeta}\right)
		 		\into (1+y)u^2 d\m.
		 	\end{eqnarray*}
		 	By choosing $\zeta=\delta_1/K_1$, we get
		 	\begin{eqnarray*}
		 		a^{(M)}_\lambda(u,u)&\geq& \frac{\delta_1}{2}\into y|\nabla u|^2 dm
		 		+\left(\lambda -\frac{K_1^2}{2\delta_1}\right)
		 		\into (1+y)u^2 d\m\geq  \frac{\delta_1}{2}\|u\|^2_V,
		 	\end{eqnarray*}
		 	for every $\lambda\geq \frac{\delta_1}{2}+\frac{K_1^2}{2\delta_1}$.
	 \end{proof}
	 	From now on in the rest of this paper we assume  $\lambda\geq \frac{\delta_1}{2}+\frac{K_1^2}{2\delta_1}$  as in Lemma \ref{coercivity}. Moreover, we will denote by $\|b\|=\sup_{u,v\in V,u,v\neq0}\frac{|b(u,v)|}{\|u\|_V\|v\|_V}$ the norm of a bilinear form $b:V\times V\rightarrow \R$.
	 	\begin{remark}
	 	We stress that  Lemma \ref{coercivity} gives us
	 		\begin{equation}\label{estimate_aM}
		\sup_{M>0}\|a_\lambda^{(M)}\|\leq C ,
\end{equation}
	 		where $C=\delta_0+K_1+\lambda$. This will be crucial in the penalization technique we are going to describe in  Section \ref{sect-penalization}.  Roughly speaking, 
	 		in order to prove the existence of a solution of the penalized coercive problem we will introduce in Theorem \ref{penalizedcoerciveproblem}, we proceed  as follows. First, we replace the bilinear form  $a_\lambda$ with the operator $a^{(M)}_\lambda$, which has bounded coefficients, and we solve the associated penalized truncated coercive problem (see Proposition \ref{penalizedcoercivetruncatedproblem}). 
Then,  thanks to \eqref{estimate_aM},  we can  deduce estimates on the solution which are uniform in $M$ (see Lemma \ref{lemma_estim})  and which will allow us to pass to the limit as $M$ goes to infinity and to find a solution of the original penalized coercive problem. 
	 	\end{remark}

%
%
		 	Finally, we define
		 	$$
		 	\mathcal{L}^\lambda  := 	\mathcal{L}-\lambda(1+y) 
		 	$$
		 	the differential operator associated with the bilinear form $a_\lambda$, that is 
		 	$$
		(	\mathcal{L}^\lambda u,v)_H=- 	a_\lambda(u,v), \qquad u\in H^2(\O,\m),\, v\in V.
		 	$$
		\subsubsection{Penalized problem }\label{sect-penalization}
		 	For any fixed $\varepsilon >0$  we define the penalizing operator 
		 	\begin{equation}\label{penalizedoperator}
		 	\zeta_\varepsilon(t,u)= -\frac 1 \varepsilon (\psi(t)-u)_+= \frac 1 \varepsilon \zeta(t,u),\qquad t\in[0,T],  u\in V.
		 	\end{equation}
		 	 Since for every fixed $t\in[0,T]$ the function $x\mapsto -(\psi(t)-x)_+$ is nondecreasing, we have the following well known monotonicity result (see \cite{BL}).
		 	\begin{lemma}	\label{monotonicity}
		 		For any fixed $t\in [0,T]$ the penalizing operator \eqref{penalizedoperator} is monotone, in the sense that
		 		$$
		 		(\zeta_\varepsilon(t,u)-\zeta_\varepsilon(t,v),u-v)_H \geq 0, \qquad   u,v \in V.
		 		$$
		 	\end{lemma}
		 
		 We now introduce the intermediate penalized coercive problem with  a source  term $g$. We consider the following assumption:
		 
		 \medskip
		 \noindent	
		 \textbf{Assumption $\mathcal{H}^0$.}
		 We say that a function $g$ satisfies Assumption $\mathcal{H}^0$ if	$ \sqrt{1+y}g\in L^2([0,T];H)$.
		 
		 \medskip

		 	\begin{theorem}\label{penalizedcoerciveproblem} 
		 Assume that  $\psi$ satisfies Assumption $\mathcal{H}^1$ and  $g$ satisfies Assumption $\mathcal{H}^0$. Then, for every fixed $\varepsilon>0$,  there exists a unique function $u_{\varepsilon,\lambda}$ such that 
		 		 $	u_{\varepsilon,\lambda}\in L^2([0,T];V)$,  $\frac{\partial u_{\varepsilon,\lambda}}{\partial t } \in L^2([0,T]; H)$ and, for all $v\in L^2([0,T];V)$,
		 		\begin{equation}\label{PCP}
		 		\begin{cases}
		 		-\left( \frac{\partial u_{\varepsilon,\lambda}}{\partial t }(t),v (t)  \right)_H + a_{\lambda}(u_{\varepsilon,\lambda}(t),v(t))+ (\zeta_\varepsilon(t,u_{\varepsilon,\lambda}(t)),v(t))_H= (g(t),v(t))_H,\qquad \mbox{ a.e. in } [0,T],\\
		 		u_{\varepsilon,\lambda}(T)=\psi(T).
		 		\end{cases}
		 		\end{equation}
		 		Moreover, the following estimates hold:
		 		\begin{equation}\label{sp1}
		 		\Vert u_{\varepsilon,\lambda} \Vert_{L^\infty([0,T],V)}\leq K,
		 		\end{equation}
		 		\begin{equation}\label{sp2}
		 		\left\Vert \frac{\partial u_{\varepsilon,\lambda}}{\partial t }\right \Vert_{L^2([0,T];H)}\leq K,
		 		\end{equation}
		 		\begin{equation}\label{sp3}
		 		\frac {1} {\sqrt{\varepsilon }}\left\Vert (\psi-u_{\varepsilon,\lambda})^+ \right \Vert_{L^\infty([0,T],H)}\leq K,
		 		\end{equation}
		 		where $K=C \left(    \Vert \Psi \Vert_{L^2([0,T];V)} + \Vert \sqrt{1+y}g\Vert_{L^2([0,T];H)}  + \Vert\sqrt{1+y}\psi \Vert_{L^2([0,T];V)}+\Vert\psi(T)\Vert_V^2   \right)$,  with $C>0$ independent of $\varepsilon$, and $\Psi$ is given in Assumption $\mathcal{H}^1$.
		 	\end{theorem}
		 	The proof of uniqueness of the solution of the penalized coercive problem follows a standard monotonicity argument as in \cite{BL}, so we omit the proof.
%

		 The  proof of existence in Theorem \ref{penalizedcoerciveproblem} is quite long and technical, so we split it into two propositions. We first consider the truncated penalized problem, which requires less stringent conditions on $\psi$ and $g$.
		 	\begin{proposition}\label{penalizedcoercivetruncatedproblem}
		 		Let $\psi\in \mathcal{C}([0,T];H)\cap     L^2([0,T];V)$ and $ g\in L^2([0,T];H)$. Moreover, assume that $ \psi(T) \in H^2(\O,\m)$, $ (1+y)\psi(T) \in H$, $\frac{\partial \psi}{\partial t}\in L^2([0,T];V) $  and $\frac{\partial g}{\partial t}\in L^2([0,T];H)$. Then, there exists a unique function $u_{\varepsilon,\lambda,M }$ such that $	u_{\varepsilon,\lambda,M }\in  L^2([0,T]; V),  \,
		 		\frac{\partial u_{\varepsilon,\lambda,M}}{\partial t } \in L^2([0,T]; V)$ and for all $v\in L^2([0,T]; V)$
		 		\begin{equation}\label{PTCP}
		 		\begin{cases}
		 		-\left( \frac{\partial u_{\varepsilon,\lambda,M}}{\partial t }(t),v(t)   \right)_H + a_{\lambda}^{(M)}(u_{\varepsilon,\lambda,M}(t),v(t))+ (\zeta_\varepsilon(t,u_{\varepsilon,\lambda,M}(t)),v(t))_H=  (g(t),v(t))_H,  \qquad \mbox{a.e. in     } [0,T) ,\\		u_{\varepsilon,\lambda,M}(T)=\psi(T). 
		 		\end{cases}
		 		\end{equation}
		 	\end{proposition}
		 	\begin{proof}
		 		\begin{enumerate}
		 			\item \textbf{ Finite dimensional problem}
		 			We use the classical Galerkin method of approximation, which consists in introducing a nondecreasing sequence $(V_j)_j$ of subspaces of $V$ such that $dimV_j<\infty$ and, for every $ v\in V ,$ there exists  a sequence $  \ (v_j)_{j\in \N}$  such that $v_j \in V_j $ for any $ j\in\N $ and $\Vert v-v_j\Vert_V\rightarrow0$ as $j\rightarrow \infty$. Moreover, we assume that $\psi(T)\in V_j,$ for all $ j\in\N$. Let $P_j$ be the projection of $V$ onto $V_j$ and $\psi_j(t)=P_j\psi(t)$. We have  $\psi_j (t) \rightarrow \psi(t)$ strongly in $V$ and $\psi_j(T)=\psi(T)$ for any $ j\in \N$. 
		 			The finite dimensional problem is, therefore, to find $u_j:[0,T]\rightarrow V_j$ such that 
		 			\begin{equation}\label{ATPCP}
		 			\begin{cases}
		 		
		 			-\left( \frac{\partial u_j}{\partial t }(t),v   \right)_H + a_{\lambda}^{(M)}(u_j(t),v)-\frac 1 \varepsilon ((\psi_j(t)-u_j(t))_+,v)_H= (g(t),v)_H, \qquad 	 v\in V_j,\\
		 				u_j(T)=\psi(T).
		 			\end{cases}
		 			\end{equation}
		
		 			This problem can be interpreted as an ordinary differential equation in $V_j$ (dim $V_j<\infty$),
		 			 that is 
		 			$$
		 			\begin{cases}
		 			-\frac{\partial u_j}{\partial t }(t) + A_{\lambda,j}^{(M)}u_j(t)-\frac 1 \varepsilon Q_j((\psi_j(t)-u_j(t))_+)= Q_jg(t),\\
		 			u_j(T)=\psi(T),
		 			\end{cases}
		 			$$
		 			where $  A_{\lambda,j}^{(M)}:V_j\rightarrow V_j$ is a finite dimensional linear operator and $Q_j$ is the projection of $H$ onto $V_j$. 
		 			It is not difficult to prove that the function $(t,u)\rightarrow Q_j((\psi_j(t)-u(t))_+)$ is continuous with values in $V_j$ and Lipschitz continuous with respect to the $u$ variable and that the term $Q_jg$ belongs to $L^2([0,T];V_j)$ (we refer to \cite{T} for the details). 
%
		 		 Therefore, by the Cauchy-Lipschitz Theorem, we can deduce the existence and the uniqueness of a solution $u_j$ of \eqref{ATPCP}, continuous from $[0,T]$ into $V_j$, a.e. differentiable and with integrable derivative.
		 			
		 			\item \textbf{ Estimates on the finite dimensional problem}
		 			First, we take $v=u_j(t)-\psi_j(t)$ in  \eqref{ATPCP}. We get
		 			$$
		 			-\left( \frac{\partial u_j}{\partial t }(t),u_j(t)-\psi_j (t)  \right)_H + a_{\lambda}^{(M)}(u_j(t),u_j(t)-\psi_j(t))-\frac 1 \varepsilon ((\psi_j(t)-u_j(t))_+,u_j(t)-\psi_j(t))_H$$
		 			$$= (g(t),u_j(t)-\psi_j(t))_H,
		 			$$
		 			which can be rewritten as
		 			\begin{align*}
		 			-\frac 1 2 &\frac{d }{d t }\Vert u_j(t)-\psi_j(t) \Vert_H^2 - \left( \frac{\partial \psi_j}{\partial t }(t),u_j(t)-\psi_j  (t) \right)_H + a_{\lambda}^{(M)}(u_j(t)-\psi_j(t),u_j(t)-\psi_j(t))_H \\&+ \frac 1 \varepsilon ((\psi_j(t)-u_j(t))_+,\psi_j(t)-u_j(t))_H+ a_{\lambda}^{(M)}(\psi_j(t),u_j(t)-\psi_j(t))  = (g(t),u_j(t)-\psi_j(t))_H .
		 			\end{align*}
		 			
		 			We integrate between $t$ and $T$ and we use coercivity and $u_j(T)=\psi_j(T)$ to obtain  
		 		\begin{align*}
 	&	\frac 1 2 \Vert u_j(t)-\psi_j(t) \Vert_H^2  + \frac{\delta_1}2\int_t^T \Vert u_j(s)-\psi_j(s)\Vert_V^2 ds +\frac 1 \varepsilon\int_t^T \Vert(\psi_j(s)-u_j(s))_+\Vert^2_Hds\\
		 		&
		 			\leq    \frac 1 {2\zeta}      \int_t^T \left \Vert \frac{\partial \psi_j(s)}{\partial t } \right \Vert_H^2 ds+\frac \zeta 2 \int_t^T  \Vert u_j(s)-\psi_j(s) \Vert_H^2 ds +    \frac 1 {2\zeta}   \int_t^T  \Vert g(s) \Vert^2_H ds+ \frac \zeta 2  \int_t^T \Vert u_j(s)-\psi_j(s) \Vert_H^2 ds
		 		\\& \qquad+ \frac {\|a^{(M)}_\lambda\|\zeta} 2  \int_t^T   \Vert u_j(s)-\psi_j(s) \Vert^2_V ds +  \frac {\|a^{(M)}_\lambda\|}{2\zeta}  \int_t^T  \Vert \psi_j(s) \Vert_V^2ds,
		 		\end{align*}
		 			for any $\zeta>0$.
		 			Recall  that $\psi_j=P_j\psi$,  and so $ \Vert\psi_j (t)\Vert^2_V\leq  \Vert\psi  (t)\Vert^2_V$. In the same way $\Vert \frac{\partial \psi_j(t)}{\partial t } \Vert^2_H\leq \Vert \frac{\partial \psi_j(t)}{\partial t } \Vert^2_V \leq \Vert \frac{\partial \psi(t)}{\partial t } \Vert^2_V    $ .
		 			Choosing $\zeta=\frac{\delta_1}{4+2\|a^{(M)}_\lambda\|}$ 
		 			after simple calculations we  deduce that there exists $C>0$ independent of $M$, $\varepsilon$ and $j$ such that 
		 			\begin{equation}\label{estim1}
		 			\begin{array}{c}
		 		\frac 1 4 	\Vert u_j(t)\Vert_H^2  +\frac{\delta_1}{8}\int_t^T \Vert u_j(s)
		 			 \Vert_V^2 ds +\frac 1 \varepsilon\int_t^T \Vert(\psi_j(s)-u_j(s))_+\Vert^2_Hds \\	\leq C \left(  \left  \Vert \frac{\partial \psi}{\partial t } \right\Vert^2_{L^2([t,T];V)} + \Vert g\Vert^2_{L^2([t,T];H)}  + \Vert\psi\Vert^2_{L^2([t,T];V)}  +\Vert\psi(T)\Vert^2_H \right).
		 			\end{array}
		 			\end{equation}
		 			
		 			\medskip
		 			
		 			We now go back to \eqref{ATPCP} and we take $v=\frac{\partial u_j}{\partial t}(t)$ so we get
		 			$$
		 			-\left \Vert \frac{\partial u_j}{\partial t} (t)\right \Vert^2_H+ \bar{a}_\lambda\left(u_j(t),\frac{\partial u_j}{\partial t}(t)\right) + \tilde{a}^{(M)}\left(u_j(t),\frac{\partial u_j}{\partial t}(t)\right)-\frac 1 \varepsilon \left(\left(\psi_j(t)-u_j(t)\right)_+,\frac{\partial u_j}{\partial t}(t)\right)_H$$$$= \left(g(t),\frac{\partial u_j}{\partial t}(t)\right)_H.
		 			$$
		 			Note that
		 			\begin{align*}
		 			-\frac 1 \varepsilon \left((\psi_j(t)-u_j(t))_+,\frac{\partial u_j}{\partial t}(t)\right)_H&= \frac 1 \varepsilon \left((\psi_j-u_j)_+,\frac{\partial (\psi_j-u_j)}{\partial t}(t)\right)_H - \frac 1 \varepsilon \left((\psi_j(t)-u_j(t))_+,\frac{\partial \psi_j}{\partial t}(t)\right)_H\\
		 			&= \frac 1 {2\varepsilon} \frac d  {d t}\Vert (\psi_j-u_j)_+(t)\Vert^2_H - \frac 1 \varepsilon \left((\psi_j(t)-u_j(t))_+,\frac{\partial \psi_j}{\partial t}(t)\right)_H.
		 			\end{align*}
		 			Therefore, using the symmetry of $\bar a_\lambda$, we have
		 			$$-\left \Vert \frac{\partial u_j}{\partial t} (t)\right \Vert^2_H+ \frac 1 2 \frac{d }{d t}     \bar{a}_\lambda(u_j(t), u_j(t)) + \tilde{a}^{(M)} \left(u_j(t),\frac{\partial u_j}{\partial t}(t)\right)+\frac 1 {2\varepsilon} \frac \partial  {\partial t}\Vert (\psi_j(t)-u_j(t))_+\Vert^2_H $$ 
		 			$$- \frac 1 \varepsilon \left((\psi_j(t)-u_j(t))_+,\frac{\partial \psi_j}{\partial t}(t)\right)_H= \left(g(t),\frac{\partial u_j}{\partial t}(t)\right)_H.
		 			$$
		 			Integrating between $t$ and $T$, we obtain 
		 			\begin{align*}
		 			&\int_t^T\left \Vert \frac{\partial u_j}{\partial t}(s) \right \Vert^2_Hds+ \frac 1 2 \bar{a}_\lambda(u_j(t), u_j(t)) +\frac 1 {2\varepsilon} \Vert (\psi_j(t)-u_j(t))_+\Vert^2_H\\
		 			&= \int_t^T \tilde{a}^{(M)}\left(u_j(s),\frac{\partial u_j}{\partial s}(s)\right)ds +\frac 1 2 \bar{a}_\lambda(\psi_j(T),\psi_j(T)) -\int_t^T \frac 1 \varepsilon \left((\psi_j(s)-u_j(s)_+,\frac{\partial \psi_j}{\partial s}(s)\right)_Hds\\&-\int_t^T \left(g(s),\frac{\partial u_j}{\partial s}(s)\right)_Hds.
		 			\end{align*}
		 			Recall that $
		 			\bar{a}_\lambda(u_j(t), u_j(t)) \geq \frac{\delta_1} 2 \Vert u_j(t)\Vert_V^2$, $\bar a _\lambda(\psi_j(T),\psi_j(T))=\bar a _\lambda(\psi(T),\psi(T))\leq \|\bar a_\lambda\|\|\psi(T)\|_V^2$ and $
		 			|\tilde{a}^{(M)}(u,v)|\leq K_1  \into y\wedge M |\nabla u||v| d\m
		 	$,
		 so that, for every  $\zeta>0$,
		 			\begin{align*}
		 		&	\int_t^T \left \Vert \frac{\partial u_j}{\partial s}(s) \right \Vert^2_Hds+ \frac {\delta_1} 4  \Vert u_j(t)\Vert_V^2 +\frac 1 {2\varepsilon} \Vert (\psi_j(t)-u_j(t))_+\Vert^2_H\\
		 		&	\leq K_1 \int_t^Tds \into y\wedge M |\nabla u_j(s,.)|\left|\frac{\partial u_j}{\partial t}(s, .)\right|d\m+\frac{\|\bar a_\lambda\|} 2 \Vert\psi
		 		(T)\Vert_V^2   +\frac 1 \varepsilon \int_t^T \Vert (\psi_j(s)-u_j(s))_+\Vert_H  \left \Vert \frac{\partial \psi_j}{\partial s}(s)\right \Vert_Hds\\&\quad
		 		+\int_t^T \Vert g(s)\Vert_H  \left \Vert \frac{\partial u_j}{\partial s}(s)\right \Vert_Hds\\
		&	\leq \frac{K_1}{2\zeta} \int_t^T \Vert u_j(s)\Vert^2_Vds + \frac{K_1M}{2}\zeta \int_t^T \left \Vert \frac{\partial u_j}{\partial s}(s) \right \Vert_H^2ds +\frac{\|\bar a_\lambda\|} 2 \Vert\psi(T)\Vert_V^2 \\& \quad+  \frac {\zeta} { 2\varepsilon}
		 \int_t^T \Vert (\psi_j(s)-u_j(s))_+\Vert^2_H ds+  \frac 1 {2\zeta\varepsilon}\int_t^T \left \Vert \frac{\partial \psi_j}{\partial t}(s)\right \Vert_H^2ds+\frac 1 {2 \zeta} \int_t^T \Vert g(s)\Vert_H^2 ds +\frac \zeta {2} \int_t^T  \left \Vert \frac{\partial u_j}{\partial s}(s)\right \Vert^2_Hds.
		 			\end{align*}
		 			
		 		From \eqref{estim1}, we already know that 
		 			$$	\int_t^T \!\!\Vert u_j(s)
		 			\Vert_V^2 ds+\frac 1 \varepsilon\int_t^T\!\!                    \Vert(\psi_j(s)-u_j(s))_+\Vert^2_Hds \leq C \left(  \left  \Vert \frac{\partial \psi}{\partial t } \right \Vert^2_{L^2([t,T];V)} \!\!\!\!\!\!+ \Vert g\Vert^2_{L^2([t,T];H)}  + \Vert\psi\Vert^2_{L^2([t,T];V)} +\Vert\psi(T)\Vert^2_H  \right),$$
		 			then we can finally deduce
		 			\begin{equation}\label{estim2}
		 			\begin{split}
		 			\int_t^T& \left \Vert \frac{\partial u_j}{\partial t} (s)\right \Vert^2_Hds+   \Vert u_j(t)\Vert_V^2 +\frac 1 {2\varepsilon} \Vert (\psi_j(t)-u_j(t))_+\Vert^2_H\\
		 		&	\leq C_{\varepsilon,M} \left(  \left  \Vert \frac{\partial \psi}{\partial t } \right\Vert^2_{L^2([t,T];V)} + \Vert g\Vert^2_{L^2([t,T];H)}  + \Vert\psi \Vert^2_{L^2([t,T];V)}  +\Vert\psi(T)\Vert^2_V \right),
		 			\end{split}
		 			\end{equation}   	
		 			where $C_{\varepsilon,M} $  is a constant which depends on $\varepsilon$ and $M$ but not on $j$.
		 			
		 			\medskip
		 			We will also need a further estimation.	If we denote $\bar{u}_j=\frac{\partial u_j}{\partial t}$ and we differentiate the equation \eqref{ATPCP} with respect to $t$ for a fixed $v$ independent of $t$, we obtain that $\bar{u}_j$ satisfies
		 			\begin{equation}\label{eqforubar}
		 			-\left(  \frac{\partial \bar{u}_j}{\partial t }(t), v\right)_H + a_{\lambda}^{(M)}(\bar{u}_j(t),v)- \frac 1 \varepsilon \left(  \bigg(\frac{\partial \psi_j}{\partial t }(t)-\bar{u}_j(t)\bigg)\mathbbm{1}_{\{ \psi_j(t) \geq u_j(t)\}},v   \right)_H=\left(\frac{\partial g}{\partial t}(t),v\right)_H, \quad v\in V_j.
		 			\end{equation}
		 		As regards the initial condition, from \eqref{ATPCP} computed in $t=T$,   for every $ v\in V_j$ we have
		 			\begin{align*}
		 			\left(  \frac{\partial u_j(T)}{\partial t},v     \right)_H &= a_{\lambda}^{(M)}(\psi(T),v)-(g(T),v)_H.
		 			\\	&=-\left( \mathcal{L}\psi(T),v\right)_H+\lambda\left( (1+y)\psi(T),v\right)_H+\left( (y\wedge M-y) \left( j_{\gamma,\mu} \frac{\psi(T)}{\partial x} +k_{\gamma,\mu}\frac{\psi(T)}{\partial y}\right) ,v\right)_H+\left(g(T),v\right)_H.
		 			\end{align*}
		 			Choosing $v= \frac{\partial u_j(T)}{\partial t}$, we deduce that 
		 			\begin{align*}
		 			\left \Vert  \frac{\partial u_j(T)}{\partial t}\right  \Vert_H& \leq C \left(     \Vert \mathcal L \psi(T) \Vert_H + \Vert(1+y)\psi(T)\Vert_H +  \Vert(y-M)_+\nabla   \psi(T)\Vert_H +\Vert g(T)\Vert_H \right)
		 		\\ &\leq C  \left(  \Vert \psi(T) \Vert_{H^2(\O,\m)}+\Vert(1+y)\psi(T)\Vert_H +\Vert g(T)\Vert_H \right),
		 			\end{align*}
		 		where  we have used that$\Vert \mathcal L \psi(T) \Vert_H+ \Vert(y-M)_+\nabla   \psi(T)\Vert_H \leq C  \Vert \psi(T) \Vert_{H^2(\O,\m)}$ for a certain constant $C>0$.

		 			We can take $v=\bar{u}_j(t)$ in \eqref{eqforubar} and we obtain
		 			$$
		 			-\left(  \frac{\partial \bar{u}_j}{\partial t }(t), \bar{u}_j(t)\right)_H + a_{\lambda}^{(M)}(\bar{u}_j(t),\bar{u}_j(t))- \frac 1 \varepsilon \left(  \bigg(\frac{\partial \psi_j}{\partial t }(t)-\bar{u}_j(t)\bigg)\mathbbm{1}_{\{ \psi_j (t)\geq u_j(t)\}},\bar{u}_j  (t) \right)_H
		 			=\left(\frac{\partial g}{\partial t}(t),\bar{u}_j(t)\right)_H,
		 			$$
		 		so that
		 		\begin{align*}
		 			-\frac 1 2\frac{d}{d t} \left\Vert   \bar{u}_j(t)\right\Vert_H^2 + \frac{\delta_1}{2}\Vert \bar{u}_j(t)\Vert_V^2&\leq \frac 1 \varepsilon \left(  \bigg(\frac{\partial \psi_j}{\partial t }(t)-\bar{u}_j(t)\bigg)\mathbbm{1}_{\{ \psi_j(t) \geq u_j\}},\bar{u}_j (t)  \right)_H+\left(\frac{\partial g}{\partial t}(t),\bar{u}_j(t)\right)_H\\
		 			&\leq \frac 1 \varepsilon \left(  \frac{\partial \psi_j}{\partial t }(t)\mathbbm{1}_{\{ \psi_j(t) \geq u_j\}},\bar{u}_j (t)  \right)_H+\left(\frac{\partial g}{\partial t}(t),\bar{u}_j(t)\right)_H.
		 		\end{align*}
		 			Integrating between $t$ and $T$, with the usual calculations, we obtain, in particular, that
		 			\begin{equation}\label{estim3}
		 			\begin{split}
		 &	\Vert \bar{u}_j(t)\Vert_H^2+\frac{\delta_1}{2}	 \int_t^T\Vert \bar{u}_j(s)\Vert_V^2 ds \\&\qquad  \leq C_\varepsilon \bigg( \Vert \psi(T) \Vert^2_{H^2(\O,\m)} +\Vert (1+y) \psi(T) \Vert^2_{H} +\Vert g(T)\Vert^2_H   + \left \Vert \frac{\partial \psi}{\partial t }\right \Vert^2_{L^2([t,T];H)}+ \left \Vert \frac{\partial g}{\partial t }\right \Vert^2_{L^2([t,T];H)} \bigg),
		 	\end{split}
		 			\end{equation}
		 			where $C_\varepsilon$ is a constant which depends on $\varepsilon$, but not on $j$.

		 			\item \textbf{Passage to the limit}
		 			
		 		Let $\varepsilon$ and $M$ be fixed. By passing to a subsequence, from \eqref{estim2} we can assume that $\frac{\partial u_j}{\partial t}$ weakly converges to a function $u_{\varepsilon,\lambda,M}'$ in $L^2([0,T];H)$. We deduce that, for any fixed $t\in [0,T]$, $u_j(t)$  weakly  converges in $H$ to 
		 			$$
		 			u_{\varepsilon,\lambda,M}(t)= \psi(T)-\int_t^T u_{\varepsilon,\lambda,M}'(s) ds. 
		 			$$ 
		 			Indeed, $u_{j}(t)$ is bounded in $V$, so the convergence is weakly in $V$. Passing to the limit in \eqref{estim3} we deduce that $\frac{\partial u_{\varepsilon,\lambda,M}}{\partial t } \in L^2([0,T];V)$.
		 			Moreover, from $\eqref{estim2}$, we have that $(\psi_j-u_j(t))^+ $ weakly converges in $H$ to a certain function $\chi(t) \in H$. 
		 			Now, for any $v\in V$ we know that there exists a sequence $(v_j)_{j\in \N}$  such that $v_j \in V_j $ for all $j \in\N $ and $\Vert v-v_j\Vert_V\rightarrow0$. 
		 			We have
		 			$$
		 			-\left( \frac{\partial u_j}{\partial t }(t),v_j   \right)_H + a_{\lambda}^{(M)}(u_j(t),v_j)_H -\frac 1 \varepsilon ((\psi_j(t)-u_j(t))_+,v_j)_H= (g(t),v_j)_H
		 			$$
		 			so, passing to the limit as $j\rightarrow \infty$, 
		 			$$
		 			-\left( \frac{\partial u_{\varepsilon,\lambda,M}}{\partial t }(t),v   \right)_H + a_\lambda(u_{\varepsilon,\lambda,M}(t),v)_H-\frac 1 \varepsilon (\chi(t),v)_H= (g(t),v)_H.
		 			$$
		 			We only have to note that $\chi(t)= (\psi(t)-u_{\varepsilon,\lambda,M}(t))_+$. In fact, $\psi_j(t)\rightarrow \psi(t)$ in $V$ and, up to a subsequence, $\mathbbm{1}_\mathcal{U}u_j(t)\rightarrow \mathbbm{1}_\mathcal{U}u_{\varepsilon,\lambda,M}(t)$ in $L^2(\mathcal{U},\m)$ for every open $\mathcal{U} $ relatively compact in $\mathcal{O}$. Therefore,  there exists a subsequence which converges a.e.  and this allows to conclude the proof. 
		 		\end{enumerate}
		 	\end{proof}
		 	We  now want to get rid of the truncated operator, that is to pass to the limit for $M\rightarrow  \infty$. In order to do this we need some  estimates on the function $u_{\varepsilon,\lambda,M}$ which are uniform in $M$.
		 	\begin{lemma}\label{lemma_estim}
		 		Assume that, in addition to the assumptions of Proposition \ref{penalizedcoercivetruncatedproblem},  
		 		$  \sqrt{1+y}\psi \in L^2([0,T];V),$    $\left| \frac{\partial \psi}{\partial t} \right|\leq \Psi$ with $\Psi\in L^2([0,T];V)$ and $g$ satisfies Assumption $\mathcal{H}^0$.
		 		Let $u_{\varepsilon,\lambda,M}$ be the solution of \eqref{PTCP}.
		 		Then,
		 		\begin{equation}\label{estim4}
		 		\begin{array}{c}
		 		\int_t^T\left \Vert \frac{\partial u_{\varepsilon,\lambda,M}}{\partial s}(s) \right \Vert^2_Hds+   \Vert u_{\varepsilon,\lambda,M}(t)\Vert_V^2 +\frac 1 {\varepsilon} \Vert (\psi(t)-u_{\varepsilon,\lambda,M}(t))_+\Vert^2_H\\
		 		\leq C  \left( \Vert 	\Psi \Vert_{L^2([0,T];V)}+ \|\sqrt{1+y}g\|_{L^2([0,T];H)} +\|\sqrt{1+y}\psi\|^2_{L^2([0,T];V)}   +\Vert\psi(T)\Vert_V^2\right),
		 		\end{array}
		 		\end{equation}   	
		 		where $C$ is a positive constant independent of $M$ and $\varepsilon$.
		 	\end{lemma}
		 	\begin{proof}
			To simplify the notation we denote $u_{\varepsilon,\lambda,M}$ by $u$ and $u_{\varepsilon,\lambda,M}-\psi=u-\psi$ by $w$.
		 	For $n\geq 0$, 
			 define  $\varphi_n(x,y)=1+y\wedge n$. Since  $\varphi_n$ and its derivatives are bounded, if $v\in V$,  we have $v\varphi_n\in V$. 
		Choosing  $v= (u-\psi)\varphi_n=w\varphi_n$ in \eqref{PTCP}, with simple passages we get
\begin{align*}
		 			 	&	-\left( \frac{\partial w}{\partial t }(t),w(t)\varphi_n   \right)_H  + a^{(M)}_{\lambda}(w(t),w(t)\varphi_n)
							    + (\zeta_\varepsilon(t,u(t)),w(t)\varphi_n)_H \\&\qquad    = \left(\frac{\partial \psi}{\partial t }(t)+g(t),w(t)\varphi_n\right)_H- a^{(M)}_{\lambda}(\psi(t),w(t)\varphi_n).
\end{align*}
		 	With the notation  $\varphi'_n=\frac{\partial \varphi_n}{\partial y}= \ind{\{y\leq n\}}$, we have
		 	\begin{align*}
  & a^{(M)}_{\lambda}(w(t),w(t)\varphi_n)= 
       \into\frac{y}{2}\left[\left(\frac{\partial w}{\partial x}(t)\right)^2+
		 		2\rho\sigma \frac{\partial w}{\partial x}(t)\frac{\partial w}{\partial y}(t)+\sigma^2 \left(\frac{\partial w}{\partial y}(t)\right)^2
		 		\right ]\varphi_n d\m	+      \lambda \into (1+y)w^2(t)\varphi_nd\m\\
				&\qquad+\into \frac{y}{2}\left(\rho\sigma \frac{\partial w}{\partial x}(t)+
		 		\sigma^2\frac{\partial w}{\partial y}(t)\right)w(t)\varphi'_nd\m
		 		+\into y\wedge M\left(\frac{\partial w}{\partial x}(t)j_{\gamma,\mu}+
		 		\frac{\partial w}{\partial y}(t)k_{\gamma,\mu}\right)w(t)\varphi_nd\m\\
		 		&\qquad \geq\delta_1 \into y\left|\nabla w(t)
		 			\right|^2 \varphi_n d\m+\lambda \into (1+y)w^2(t)\varphi_nd\m
					-K_1\into y\left|\nabla w(t)\right|| w(t)|\varphi_n d\m\\
					&\qquad
					-K_2\into y\left|\nabla w(t)\right|| w(t)|\ind{\{y\leq n\}}d\m,
					\end{align*}
where $K_2=\frac{\sqrt{\rho^2\sigma^2+\sigma^4}}{2}$. Note that, if $n=0$, the last term vanishes, and that, for all $n>0$,
\[
\into y\left|\nabla w(t)\right| |w(t)|\ind{\{y\leq n\}}d\m\leq \Vert w(t)\Vert _V^2.
\]
Therefore, for all $\zeta>0$,
\begin{align*}
 &  a^{(M)}_{\lambda}(w(t),w(t)\varphi_n)
			\geq\delta_1 \into y\left|\nabla w(t)
		 			\right|^2 \varphi_n d\m+\lambda \into (1+y)w^2(t)\varphi_nd\m\\
					&\qquad
					-K_1\into y\left(\frac{\zeta}{2}\left|\nabla w(t)\right|^2
					    +\frac{1}{2\zeta}|w(t)|^2\right)\varphi_n d\m				
					-K_2\Vert w(t)\Vert _V^2\\
					&\qquad \geq\left( \delta_1 -\frac{K_1\zeta}{2}\right)\into y\left|\nabla w(t)
		 			\right|^2 \varphi_n d\m+\left(\lambda -\frac{K_1}{2\zeta}\right)\into (1+y)w^2(t)\varphi_nd\m			
					-K_2\Vert w(t)\Vert _V^2\\
					&\qquad\geq 
					\frac{\delta_1}{2}\into\left( y\left|\nabla w(t)
		 			\right|^2+(1+y)w^2(t)\right)\varphi_nd\m-K_2\Vert w(t)\Vert _V^2,
					\end{align*}
where, for the last inequality, we have chosen $\zeta=\delta_1/K_1$ and used the inequality 
$\lambda\geq \frac{\delta_1}{2}+\frac{K_1^2}{2\delta_1}$. Again, in the case $n=0$ the last term 
on the righthand side can be omitted.

Hence, we have, with the notation $\Vert v\Vert ^2_{V,n}=\into\left( y\left|\nabla v
		 			\right|^2+(1+y)v^2\right)\varphi_nd\m$,		 		
		 	\begin{align*}
		 	&\-\frac{1}{2}\frac{d}{dt}\into w^2(t)\varphi_n d\m+
		 			\frac{\delta_1}{2}\Vert w(t)\Vert ^2_{V,n}
							 			+\frac{1}\varepsilon \into (-w(t))_+^2\varphi_nd\m\leq  	\\&\qquad	
				\left(g(t)+ \frac{\partial \psi}{\partial t }(t),w(t)\varphi_n   \right)_H
					        -a^{(M)}_{\lambda}(\psi(t),w(t)\varphi_n)
		 			+K_2\Vert w(t)\Vert _V^2.
							 	\end{align*}
In the case $n=0$, the inequality reduces to
$$
		 		-\frac{1}{2}\frac{d}{dt}\into w^2(t)d\m+
		 			\frac{\delta_1}{2}\Vert w(t)\Vert ^2_{V}
							 			+\frac{1}\varepsilon \into (\psi-u)_+^2d\m
		 			\leq  		
				\left(g(t)+ \frac{\partial \psi}{\partial t }(t),w(t)   \right)_H
					        -a^{(M)}_{\lambda}(\psi(t),w(t)).
$$
Now, integrate from  $t$ to $T$ and use $u(T)=\psi(T)$ to derive
		 	\begin{equation}\label{calcul}
		 	\begin{split}
		 	&	\frac{1}{2}\into w(t)^2\varphi_n d\m+
		 	\frac 	{\delta_1} 2\int_t^T ds \Vert w(s)\Vert ^2_{V,n}
		 	+\frac{1}\varepsilon\int_t^T ds \into (-w(s))_+^2\varphi_n d\m\\	&\leq\int_t^T\left(g(s)+ \frac{\partial \psi}{\partial t }(s),w(s)\varphi_n   \right)_Hds+\left|\int_t^Ta^{(M)}_{\lambda}(\psi(s),w(s)\varphi_n) ds\right|  +K_2\int_t^T\Vert  w(s) \Vert _V^2ds,
		 		 	\end{split}
		 \end{equation}
and, in the case $n=0$,
\begin{equation}\label{calcul0}
\begin{split}
&\frac{1}{2}\Vert w(t)\Vert _H^2+
		 	\frac 	{\delta_1} 2\int_t^T  \Vert w(s)\Vert ^2_{V}ds
		 	+\frac{1}\varepsilon\int_t^Tds \into  (-w(s))_+^2 d\m	\\&\quad \leq\int_t^T\!\left(g(s)+ \frac{\partial \psi}{\partial t }(s),w(s)
			  \right)_Hds+\int_t^T\! \left|a^{(M)}_{\lambda}(\psi(s),w(s))\right|  ds.
 \end{split}
\end{equation}

		 We have, for all $\zeta_1>0$,
		 \begin{align*}
		 &\int_t^T\bigg(g(s)+ \frac{\partial \psi}{\partial t }(s),w(s)\varphi_n   \bigg)_Hds\leq \frac {\zeta_1}2 \int_t^Tds\into|w(s)|^2\varphi_nd\m+\frac 1 {2\zeta_1}\int_t^Tds\into\left|g(s)+ \frac{\partial \psi}{\partial t }(s)\right|^2\varphi_n d\m\\
		& \qquad\leq \frac {\zeta_1}2 \int_t^Tds\into|w(s)|^2\varphi_nd\m+\frac 1 {\zeta_1}\|\sqrt{1+y}g\|_{L^2([t,T];H)}^2+\frac 1 {\zeta_1}
	\left	\|\sqrt{1+y}\frac{\partial\psi}{\partial t}\right\|_{L^2([t,T];H)}^2.
		 \end{align*}
		 Moreover,	 it is easy to check that, for all $v_1$, $v_2\in V$,
		 \[
		 |a^{(M)}_{\lambda}(v_1,v_2\varphi_n)|\leq K_3 \Vert v_1\Vert _{V,n}\Vert v_2\Vert _{V,n}, \mbox{\qquad with } K_3=\delta_0+K_1+K_2+\lambda,
		 \]
		 so that, for any $\zeta_2>0$,
		 \begin{align*}
		 \int_t^T |a^{(M)}_{\lambda}(\psi(s),w(s)\varphi_n)| ds&\leq K_3\int_t^Tds\Vert \psi(s)\Vert _{V,n}\Vert w(s)\Vert _{V,n} \leq  \frac {K_3\zeta_2}2\int_t^Tds\Vert w(s)\Vert ^2_{V,n}
                        +\frac {K_3}{2\zeta_2}\int_t^Tds\Vert \psi(s)\Vert ^2_{V,n}.
		 \end{align*}
		 Now, if we chose  $\zeta_1=K_3\zeta_2=\delta_1/4$
		 and we go back to \eqref{calcul} and \eqref{calcul0}, using $\left|\frac{\partial \psi}{\partial t}\right|\leq \Psi$ we get
		 \begin{equation}\label{calculbis}
		 	\begin{split}
		 	&	\frac{1}{2}\into w^2(t)\varphi_n d\m+
		 	\frac 	{\delta_1} 4\int_t^T  \Vert w(s)\Vert ^2_{V,n}ds
		 	+\frac{1}\varepsilon\int_t^T ds \into (-w(s))_+^2\varphi_n d\m\\	&\leq
			\frac 4 {\delta_1}\left(\|\sqrt{1+y}g\|_{L^2([t,T];H)}^2+\|\sqrt{1+y}\Psi\|_{L^2([t,T];H)}^2\right)
			+\frac {2K_3^2}{\delta_1}\int_t^T\Vert \psi(s)\Vert ^2_{V,n}ds+K_2\Vert  w \Vert _{L^2([t,T];H)}^2ds,\\
			&\leq
			\frac 4 {\delta_1}\left(\|\sqrt{1+y}g\|_{L^2([t,T];H)}^2+\|\sqrt{1+y}\Psi\|_{L^2([t,T];H)}^2\right)
			+\frac {4K_3^2}{\delta_1}\left\|\sqrt{1+y}\psi\right\|^2_{L^2([t,T];V)}+K_2\Vert  w \Vert _{L^2([t,T];H)}^2,
		 		 	\end{split}
		 \end{equation}
where the last inequality follows from the estimate $\Vert v\Vert ^2_{V,n}\leq 2\Vert \sqrt{1+y}v\Vert _V^2$, and, in the case $n=0$,
\begin{equation}\label{calcul0bis}
\frac{1}{2}\Vert w(t)\Vert _H^2+
		 	\frac 	{\delta_1} 4\int_t^T \!\! \Vert w(s)\Vert ^2_{V}
		 ds	+\frac{1}\varepsilon\int_t^T \!\!\!ds\! \into \!(-w(s))_+^2 d\m	\leq
			\frac 4 {\delta_1}\left(\|g\|_{L^2([t,T];H)}^2+  \! \|\Psi\|_{L^2([t,T];H)}^2\right)
			+\frac {2K_3^2}{\delta_1}\| \psi\|^2_{L^2([t,T];V)}.
\end{equation}
From \eqref{calcul0bis} recalling that $w=u-\psi$ we deduce 
\begin{equation}\label{calcul0tris}
\begin{split}
\int_t^T\!\!\|u(s)\|_V^2ds&\leq \int_t^T\!\!2( \|w(s)\|_V^2+\|\psi(s)\|_V^2)ds \leq \frac{32}{\delta_1^2}\left(\|g\|_{L^2([t,T];H)}^2+\|\Psi\|_{L^2([t,T];H)}^2\right)+\left(\frac {16K_3^2}{\delta_1^2}+2\right)\| \psi\|^2_{L^2([t,T];V)}.
\end{split}
\end{equation}
Moreover, combining \eqref{calculbis} and \eqref{calcul0bis}, we have
\begin{equation*}
		 	\begin{split}
		 	&	\frac{1}{2}\into w^2(t)\varphi_n d\m+
		 	\frac 	{\delta_1} 4\int_t^T  \Vert w(s)\Vert ^2_{V,n}ds
		 	+\frac{1}\varepsilon\int_t^T ds \into (-w(s))_+^2\varphi_n d\m\\	&\leq
			\left(\frac 4 {\delta_1}+ \frac{16K_2}{\delta_1^2}\right)\left(\|\sqrt{1+y}g\|_{L^2([t,T];H)}^2+\|\sqrt{1+y}\Psi\|_{L^2([t,T];H)}^2\right)
			+\frac {4K_3^2}{\delta_1}\left(1+\frac{2K_2}{\delta_1}\right)\|\sqrt{1+y}\psi\|^2_{L^2([t,T];V)}.
		 		 	\end{split}
		 \end{equation*}
	In particular,
		\begin{equation*}
		\begin{split}
&\int_t^Tds\into y|\nabla u(s)|^2\varphi_n d\m \leq \int_t^T\Vert u(s)\Vert ^2_{V,n }ds\leq 2 \int_t^T  \Vert w(s)\Vert ^2_{V,n}ds+ 2\int_t^T ds\Vert \psi(s)\Vert ^2_{V,n}ds\\
&\leq \frac 8 {\delta_1}  \left(\frac 4 {\delta_1}+ \frac{16K_2}{\delta_1^2}\right)\left(\|\sqrt{1+y}g\|_{L^2([t,T];H)}^2+\|\sqrt{1+y}\Psi\|_{L^2([t,T];H)}^2\right)
\\&\quad+\left(\frac {32K_3^2}{\delta_1^2}\left(1+\frac{2K_2}{\delta_1}\right)+4\right)\|\sqrt{1+y}\psi\|^2_{L^2([t,T];V)} 
\end{split}
		\end{equation*}
		and, 
	by	using the Monotone convergence theorem, we deduce 
		\begin{equation}\label{uchapeau}
	\int_t^T\|y|\nabla u(s)| \|_{H}^2 ds\leq K_4\left(\|\sqrt{1+y}g\|_{L^2([t,T];H)}^2+\|\sqrt{1+y}\Psi\|_{L^2([t,T];H)}^2+\|\sqrt{1+y}\psi\|^2_{L^2([t,T];V)} \right),
		\end{equation}
where $K_4=\frac 8 {\delta_1}  \left(\frac 4 {\delta_1}+ \frac{16K_2}{\delta_1^2}\right) \vee\left(\frac {32K_3^2}{\delta_1^2}\left(1+\frac{2K_2}{\delta_1}\right)+4  \right)		$.

We are now in a position to prove \eqref{estim4}.  Taking $v=\frac{\partial u}{\partial t }$ in \eqref{PTCP}, we have
$$
		 		-\left \Vert \frac{\partial u}{\partial t}(t) \right \Vert^2_H+ \bar{a}_\lambda\left(u(t),\frac{\partial u}{\partial t}(t)\right) +
				   \tilde{a}^{(M)}\left(u(t),\frac{\partial u}{\partial t}(t)\right)-\frac{1}{\varepsilon}\left((\psi(t)-u(t))_+,\frac{\partial u}{\partial t}(t)\right)_H
		 		= \left(g(t),\frac{\partial u}{\partial t}(t)\right)_H.
		 		$$
Note that, since $\bar a_\lambda$ is symmetric, $\frac{d}{dt}\bar{a}_\lambda\left(u(t),u(t)\right)=2\bar{a}_\lambda\left(u(t),\frac{\partial u}{\partial t}(t)\right)$.
On the other hand, 
\begin{align*}
\left((\psi(t)-u(t))_+,\frac{\partial u}{\partial t}\right)_H &=-\frac{1}{2}\frac{d}{dt}\Vert (\psi(t)-u(t))_+\Vert _H^2+\left((\psi(t)-u(t))_+,\frac{\partial \psi}{\partial t}(t)\right)_H,
\end{align*}
 so that
\begin{align*}
		 	&	\left\Vert \frac{\partial u}{\partial t} (t)\right \Vert^2_H-\frac{1}{2}\frac{d}{dt}\bar{a}_\lambda\left(u(t),u(t)\right) 
				   -\frac{1}{2\varepsilon}\frac{d}{dt}\Vert (\psi(t)-u(t))_+\Vert _H^2\\&\qquad
		 		 = \tilde{a}^{(M)}\left(u(t),\frac{\partial u}{\partial t}(t)\right) -\left(g(t),\frac{\partial u}{\partial t}(t)\right)_H
				- \frac 1{\varepsilon}\left((\psi(t)-u(t))_+,\frac{\partial \psi}{\partial t}(t)\right)_H\\
				& \qquad\leq \left|\tilde{a}^{(M)}\left(u(t),\frac{\partial u}{\partial t}(t)\right)\right|+
			   \Vert g(t)\Vert _H\left\Vert \frac{\partial u}{\partial t}(t) \right\Vert_H+\frac 1{\varepsilon}\left((\psi(t)-u(t)_+,\Psi (t)\right)_H\\
			   & \qquad\leq \left(K_1 \left\Vert y|\nabla u(t)|\right\Vert_H+\Vert g(t)\Vert _H \right) \left\Vert \frac{\partial u}{\partial t}(t)\right\Vert_H
			       +\frac 1{\varepsilon}\left((\psi(t)-u(t))_+,\Psi(t) \right)_H.
\end{align*}		
		Moreover, if we take $v=\Psi(t)$ in \eqref{PTCP}, we get
		 		$$
		 		-\left( \frac{\partial u}{\partial t }(t),\Psi(t) \right)_H + a_{\lambda}^{(M)}(u(t),\Psi(t))-
				\frac 1{\epsilon} \left((\psi(t)-u(t))_+,\Psi(t)\right)_H= \left(g(t),\Psi(t) \right)_H,
		 		$$
		 		so that
		 		\begin{equation}\label{estavecv1}
		 		\begin{split}
	 &	\frac{1}{\varepsilon}\left( (\psi(t)-u(t))_+,\Psi(t)\right)_H \leq  \left \Vert    \frac{\partial u}{\partial t } (t)\right \Vert_H \Vert \Psi(t)\Vert_H+ \Vert a^{(M)}_\lambda\Vert 
	           \Vert u(t) \Vert_V\Vert \Psi(t)\Vert_V+\Vert g(t)\Vert_H\Vert \Psi(t)\Vert_H.	\end{split}
		 		\end{equation}
		 		 		Therefore,
		 		\begin{align*}
		 	&	\left\Vert \frac{\partial u}{\partial t} (t)\right \Vert^2_H-\frac{1}{2}\frac{d}{dt}\bar{a}_\lambda\left(u(t),u(t)\right) 
				   -\frac{1}{2\varepsilon}\frac{d}{dt}\Vert (\psi(t)-u(t))_+\Vert _H^2
		 				\\&\qquad	 \leq	 \left(K_1 \left\Vert y|\nabla u(t)|\right\Vert_H+\Vert g(t)\Vert _H +\Vert \Psi(t)\Vert_H\right) 
								\left\Vert \frac{\partial u}{\partial t} (t)\right\Vert_H
			   + \Vert a^{(M)}_\lambda\Vert 
	           \Vert u (t)\Vert_V\Vert \Psi(t)\Vert_V+\Vert g(t)\Vert_H\Vert \Psi(t)\Vert_H,
\end{align*}
hence
\begin{align*}
		 	&	\frac{1}{2}\left\Vert \frac{\partial u}{\partial t}(t) \right \Vert^2_H-\frac{1}{2}\frac{d}{dt}\bar{a}_\lambda\left(u(t),u(t)\right) 
				   -\frac{1}{2\varepsilon}\frac{d}{dt}\Vert (\psi(t)-u(t))_+\Vert _H^2
		 				\\&	\qquad	 \leq \frac{1}{2}\left(K_1 \left\Vert y|\nabla u(t)|\right\Vert_H+\|g(t)\Vert _H +\Vert \Psi(t)\Vert_H\right)^2 		
			   +\Vert a^{(M)}_\lambda\Vert 
	           \Vert u(t) \Vert_V^2 \Vert \Psi(t)\Vert_V^2+\Vert g(t)\Vert_H\Vert \Psi(t)\Vert_H.
\end{align*}
		 		Integrating between $t$ and $T$, we get, 
		 		\begin{align*}
		 	&	\frac{1}{2}\left \Vert \frac{\partial u}{\partial s} \right \Vert^2_{L^2([t,T];H)}+ 
				\frac 1 2 \bar{a}_\lambda\left(u(t), u(t)\right) +\frac 1 {2\varepsilon} \Vert (\psi(t)-u(t))_+\Vert^2_H
		 		\leq \frac 1 2 \bar{a}_\lambda(\psi(T),\psi(T))+2\|g\Vert ^2_{L^2([t,T];H)}\\&\qquad  +
				     2\Vert \Psi\Vert_{L^2([t,T];H)}^2 	+\frac{3K_1^2}{2} \left\Vert y|\nabla u|\right\Vert^2_{L^2([t,T];H)}
				        			   +\frac{ \|a^{(M)}_\lambda\|}2\Vert u \Vert_{L^2([t,T];V)}+\frac{ \|a^{(M)}_\lambda\|}{2}\Vert \Psi\Vert_{L^2([t,T];V)},
		 		\end{align*}
		 so, recalling that $\bar{a}_\lambda(u(t),u(t)\geq\delta_1\into y|\nabla u(t)|^2d\m+\lambda\into(1+y)u^2d\m\geq (\delta_1\wedge\lambda) \|u(t)\|_V^2$,
		 \begin{align*}
		 		 		\frac{1}{2}&\left \Vert \frac{\partial u}{\partial s} \right \Vert^2_{L^2([t,T];H)}+ 
		 		 		\frac {\delta_1\wedge\lambda} 2 \|u(t)\|_V^2 +\frac 1 {2\varepsilon} \Vert (\psi(t)-u(t))_+\Vert^2_H\\&
		 		 		\leq \frac {\|\bar{a}_\lambda\|} 2 \|\psi(T)\|_V^2+2\|g\Vert ^2_{L^2([t,T];H)} +
		 		 		2\Vert \Psi\Vert_{L^2([t,T];H)}^2 \\&\quad 	+\frac{3K_1^2}{2} \left\Vert y|\nabla u|\right\Vert^2_{L^2([t,T];H)}
		 		 		+\frac{ \Vert a^{(M)}_\lambda\Vert }2\Vert u \Vert_{L^2([t,T];V)}+\frac{ \Vert a^{(M)}_\lambda\Vert }{2}\Vert \Psi\Vert^2_{L^2([t,T];V)}\\&\leq  \frac {\|\bar{a}_\lambda\|} 2 \|\psi(T)\|_V^2+2\|g\Vert ^2_{L^2([t,T];H)} +
		 		 		2\Vert \Psi\Vert_{L^2([t,T];H)}^2 \\&\quad +
		 		 	\frac{3K_1^2}{2}	K_4\left(\|\sqrt{1+y}g\|_{L^2([t,T];H)}^2+\|\sqrt{1+y}\Psi\|_{L^2([t,T];H)}^2+\|\sqrt{1+y}\psi\|^2_{L^2([t,T];V)}\right )\\
		 		 	&\quad +\frac{ \Vert a^{(M)}_\lambda\Vert }2\left(\frac{32}{\delta_1^2}\left(\|g\|_{L^2([t,T];H)}^2+\|\Psi\|_{L^2([t,T];H)}^2\right)+\left(\frac {16K_3^2}{\delta_1^2}+2\right)\| \psi\|^2_{L^2([t,T];V)}\right)+\frac{ \Vert a^{(M)}_\lambda\Vert }{2}\Vert \Psi\Vert^2_{L^2([t,T];V)},
		 \end{align*}
		 where the last inequality follows from \eqref{calcul0tris} and \eqref{uchapeau}.
		 Rearranging the terms, 
		  we deduce that there exists a constant $C>0$ independent of $M$ and $\varepsilon$ such that 
		   \begin{align*}
		   \frac{1}{2}\left \Vert \frac{\partial u}{\partial s} \right \Vert^2_{L^2([t,T];H)}&+ 
		   \frac {\delta_1\wedge\lambda} 4 \|u(t)\|_V^2 +\frac 1 {2\varepsilon} \Vert (\psi(t)-u(t))_+\Vert^2_H
		  \\&\leq C\left(\|\sqrt{1+y}g\|_{L^2([t,T];H)}^2+\|\Psi\|_{L^2([t,T];V)}^2+\left\|\sqrt{1+y}\psi\right\|^2_{L^2([t,T];V)}+ \|\psi(T)\|_V^2 \right),
		   \end{align*}
		   which concludes the proof.
		 	\end{proof}
		 	\begin{proof}[Proof of Theorem \ref{penalizedcoerciveproblem}: existence]
		 		Assume for a first moment that we have the further assumptions $ \psi(T) \in H^2(\O,\m)$, $ (1+y)\psi(T) \in H$, $\frac{\partial \psi}{\partial t}\in L^2([0,T];V) $  and $\frac{\partial g}{\partial t}\in L^2([0,T];H)$.
		 		Thanks to $\eqref{estim4}$ we can repeat the same arguments as in the proof of Proposition \ref{penalizedcoercivetruncatedproblem} in order to pass to the limit in $j$, but this time as $M\rightarrow \infty $.  
%
Therefore, we deduce the existence of a function $u_{\varepsilon,\lambda}\in L^2([0,T];V)$ with $\frac{\partial u_{\varepsilon,\lambda}}{\partial t}\in L^2([0,T];H)$ and such that
		 		$$
		 		-\left( \frac{\partial u_{\varepsilon,\lambda}}{\partial t }(t),v   \right)_H + a_\lambda(u_{\varepsilon,\lambda}(t),v)_H-\frac 1 \varepsilon ((\psi(t)-u_{\varepsilon,\lambda}(t))_+,v)_H= (g(t),v)_H.
		 		$$
 The estimates \eqref{sp1}, \eqref{sp2} and \eqref{sp3} directly follow from \eqref{estim4} as $M\rightarrow \infty$.

		 		 		We have now to weaken the assumptions on $g$ and $\psi$. We can do this by a regularization procedure. In fact, let us assume that $\psi$ satisfies Assumption $\mathcal H^1$ (so, in particular, $\left|\frac{\partial \psi}{\partial t}\right|\leq \Psi$ for a certain $\Psi\in L^2([0,T];V)$) and $g$ satisfies Assumption $\mathcal H^0$. Then, by  standard regularization techniques (see  for example \cite[Corollary A.12]{DF}), we can find  sequences of functions $(g_n)_n$, $(\psi_n)_n$ and $(\Psi_n)_n$ of class $C^\infty$ with compact support such that, for any $n\in\N$, $n\in\N$, $|\frac{\partial\psi_n}{\partial t}|\leq \Psi_n$ and all the regularity assumptions required in the first part of the proof are satisfied. Moreover, $\Vert \sqrt{1+y}g_n -\sqrt{1+y}g\Vert_{L^2([0,T];H)}\rightarrow0$, $ \Vert \sqrt{1+y}\psi_n -\sqrt{1+y}\psi\Vert_{L^2([0,T];V)} \rightarrow 0$, $ \Vert\Psi_n-\Psi\Vert_{L^2([0,T];V)} \rightarrow 0$, $\Vert \psi_n(T)-\psi(T)\Vert_V\rightarrow0$ as $n\rightarrow \infty$  (we refer to \cite{T} for the details).  Therefore, the solution $u_{\varepsilon,\lambda,M}^n$ of the equation \eqref{PCP} with source function $g_n$ and obstacle function $\psi_n$ satisfies  
		 			\begin{equation}\label{lessass}
		 			\begin{array}{c}
		 			\int_t^T \left \Vert \frac{\partial u^n_{\varepsilon,\lambda,M}}{\partial s} (s)\right \Vert^2_H \,ds+   \Vert u^n_{\varepsilon,\lambda,M}(t)\Vert_V^2 +\frac 1 {\varepsilon} \Vert (\psi_n(t)-u^n_{\varepsilon,\lambda,M}(t))_+\Vert^2_H\\
		 			\leq C  \left(   \|\sqrt{1+y}g_n\|_{L^2([0,T];H)} +\|\sqrt{1+y}\psi_n\|^2_{L^2([0,T];V)}   +\|\Psi_n\|_{L^2([0,T];V)}^2 +\Vert\psi_n(T)\Vert_V^2 \right).
		 			\end{array}
		 			\end{equation} Then, we can take the limit for $n\rightarrow \infty$ in \eqref{lessass} and the assertion follows as in the  first part of the proof.
		 	\end{proof}

		 	Moreover, we have the following Comparison principle for the coercive penalized problem.
		 	\begin{proposition} \label{CompPrinc1}
	
		 		\begin{enumerate}
		 			\item Assume that $\psi_i$  satisfies Assumption $\mathcal{H}^1$ for $i=1,2$ and $g$ satisfies Assumption $\mathcal{H}^0$.  Let $u^i_{\varepsilon,\lambda}$ be the unique solution of \eqref{PCP}   with  obstacle function $\psi_i$ and source function $g$.  If $\psi_1\leq \psi_2$, then $u^1_{\varepsilon,\lambda}\leq u^2_{\varepsilon,\lambda}$.
		 			\item Assume that  $\psi$ satisfies Assumption $\mathcal{H}^1$  and  $g_i$ satisfy Assumption $\mathcal{H}^0$ for $i=1,2$. Let $u^i_{\varepsilon,\lambda}$ 
		 			be the unique solution of \eqref{PCP} with obstacle function $\psi$ and source function $g_i$.  If $g_1\leq g_2$, then $u^1_{\varepsilon,\lambda}\leq u^2_{\varepsilon,\lambda}$.
		 			\item Assume that $\psi_i$  satisfies Assumption $\mathcal{H}^1$ for $i=1,2$ and $g$ satisfies Assumption $\mathcal{H}^0$.  Let $u^i_{\varepsilon,\lambda}$ be the unique solution of \eqref{PCP}   with  obstacle function $\psi_i$ and source function $g$.  If $\psi_1- \psi_2\in L^\infty$, then $u^1_{\varepsilon,\lambda}- u^2_{\varepsilon,\lambda}\in L^\infty $ and $\Vert  u^1_{\varepsilon,\lambda}- u^2_{\varepsilon,\lambda}\Vert_\infty \leq   \Vert \psi_1- \psi_2\Vert_\infty$.
		 		\end{enumerate}
		 	\end{proposition}
		 Proposition \ref{CompPrinc1} can be proved with standard techniques introduced in \cite[Chapter 3]{BL} so we omit the proof.
		 	\subsubsection{Coercive variational inequality}
		 	\begin{proposition}\label{coercive variational inequality}
		 		Assume that $\psi$ satisfies Assumption $\mathcal{H}^1$  and $g$ satisfies Assumption $\mathcal{H}^0$.
		 			 Moreover, assume that $ 0\leq \psi \leq \Phi$ with $\Phi\in L^2([0,T]; H^2(\mathcal{O},\m))$ such  that $\frac{\partial \Phi}{\partial t}+\mathcal{L}\Phi \leq 0$ and $0\leq g \leq -\frac{\partial \Phi}{\partial t} -\mathcal{L}^\lambda \Phi$.
		 		Then, there exists a unique function $u_\lambda $ such that $ u_\lambda\in L^2([0,T];V), \,\frac{\partial u_\lambda}{\partial t } \in L^2([0,T]; H)$ and 
		 		\begin{equation} \label{CVI}
		 		\begin{cases}
		 	-\left( \frac{\partial u_\lambda}{\partial t },v -u_\lambda  \right)_H + a_\lambda(u_\lambda,v-u_\lambda)\geq (g,v-u_\lambda)_H, \qquad \mbox{ a.e. in } [0,T]\qquad v\in  L^2([0,T];V), \ v\geq \psi, \\u_\lambda(T)=\psi(T),\\u_\lambda \geq \psi \mbox{ a.e. in } [0,T]\times \R \times (0,\infty).
		 		\end{cases}
		 		\end{equation}
		 		Moreover, $0 \leq u_\lambda \leq  \Phi$.
		 	\end{proposition}
		 	\begin{proof}
		 		The uniqueness of the solution of \eqref{CVI} follows by a standard monotonicity argument introduced in \cite[Chapter 3]{BL} (see \cite{T}). 
		 		As regards the existence of a solution, we follow  the lines of the proof of \cite[Theorem 2.1]{BL} but we repeat here the details since we use a compactness argument which is not present in the classical theory.
		 		
		 		For each fixed $\varepsilon >0$ we have the estimates \eqref{sp1} and \eqref{sp2}, so, for every $t\in[0,T]$, we can extract a subsequence $u_{\varepsilon,\lambda}$ such that $u_{\varepsilon,\lambda}(t) \rightharpoonup u_\lambda(t)$ in $V$ as $\varepsilon \rightarrow 0 $  and $u'_\varepsilon(t) \rightharpoonup u'_\lambda(t)$ in $H$ for some function $u_\lambda \in V$.
		 		
		 		Note that $u=0 $ is the unique solution of \eqref{PCP} when $\psi=g=0$, while $u=\Phi $ is the unique solution of \eqref{PCP} when $\psi=\Phi$ and $g=-\frac{\partial \Phi}{\partial t}-\mathcal{L}^\lambda \Phi=-\frac{\partial \Phi}{\partial t}-\mathcal{L} \Phi +\lambda(1+y)\Phi  $. Therefore, Proposition \ref{CompPrinc1} implies that $0\leq u_{\varepsilon,\lambda} \leq \Phi$. Recall that $u_{\varepsilon,\lambda}(t) \rightarrow u_\lambda(t)$ in $L^2(\mathcal{U}, \m ) $ for every relatively compact open $\mathcal{U} \subset \mathcal{O}$. This, together with the fact that $d\m$ is a finite measure, allows to conclude that we have strong convergence of $u_{\varepsilon,\lambda}$ to $u_\lambda$ in $H$. In fact, if $\delta >0$ and  $\mathcal{O}_{\delta}:=(-\frac 1 \delta, \frac 1 \delta )\times (\delta, \frac 1 \delta)$, 
		 		\begin{align*}
		 	\int_0^Tds	\int_{\mathcal{O}} |u_{\varepsilon,\lambda}(s)- u_\lambda(s)|^2d\m&\leq  	\int_0^Tds\int_{\mathcal{O}_{\delta}} |u_{\varepsilon,\lambda}(s)- u_\lambda(s)|^2d\m+	\int_0^Tds\int_{\mathcal{O}^c_{\delta}} |u_{\varepsilon,\lambda}(s)- u_\lambda(s)|^2d\m\\&\leq 	\int_0^Tds\int_{\mathcal{O}_{\delta}} |u_{\varepsilon,\lambda}(s)- u_\lambda(s)|^2d\m + 	\int_0^Tds\int_{\mathcal{O}^c_{\delta}} 4\Phi^2(s) d\m
		 		\end{align*} 
		 		and it is enough to let $\delta$ goes to 0.
		 		
		 		From \eqref{sp3} we also have that $(\psi(t)-u_{\varepsilon,\lambda}(t))^+ \rightarrow 0 $ strongly in $H$ as $\varepsilon \rightarrow 0$ . On the other hand $(\psi(t)-u_{\varepsilon,\lambda}(t))_+\rightharpoonup \chi(t)$ weakly in $H$ and $\chi =(\psi-u_\lambda)_+$ since there exists a subsequence of $u_{\varepsilon,\lambda} (t)
		 		$ which converges pointwise to $u_\lambda(t)$. Therefore, $(\psi(t)-u_\lambda(t))^+=0$, which means $u_\lambda(t) \geq \psi(t)$.
		 		
		 		Then we consider the penalized coercive equation in \eqref{PCP} replacing $v$ by $v-u_{\varepsilon,\lambda}(t)$, with $v \geq \psi(t)$. Since $ \zeta_\varepsilon(t,v)=0$ and $ (\zeta_\varepsilon(t,v)- \zeta_\varepsilon(t,u_{\varepsilon,\lambda}(t)),v-u_{\varepsilon,\lambda}(t))_H\geq 0$ we easily deduce that
		 		$$
		 		-\left( \frac{\partial u_{\varepsilon,\lambda}}{\partial t }(t),v-u_{\varepsilon,\lambda}(t)   \right)_H + a_\lambda(u_{\varepsilon,\lambda}(t),v-u_{\varepsilon,\lambda}(t))\geq (g(t),v-u_{\varepsilon,\lambda}(t))_H
		 		$$ 
		 		so that, letting $\varepsilon $ goes to 0, we have
		 		\begin{align*}
		 		-\left( \frac{\partial u_\lambda}{\partial t }(t),v-u_\lambda (t) \right)_H + a_\lambda(u_\lambda(t),v)&\geq (g(t),v-u_\lambda(t))_H+ \liminf_{\varepsilon\rightarrow 0} a_\lambda(u_{\varepsilon,\lambda}(t),u_{\varepsilon,\lambda}(t))\\& \geq (g(t),v-u_\lambda(t))_H+  a_\lambda(u_\lambda(t),u_\lambda(t)).
		 		\end{align*}
			 		
		 		Moreover, since $0\leq u_{\varepsilon,\lambda}\leq \Phi$ for every $\varepsilon>0$ and $u_\lambda=\lim_{\varepsilon  \rightarrow 0} u_{\varepsilon,\lambda}$, we have $0\leq u_\lambda \leq \Phi$ and the assertion follows.
		 	\end{proof}
		 	The following Comparison Principle is a direct consequence of Proposition \ref{CompPrinc1},.
		 	\begin{proposition}\label{CompPrinc2}
	
		 		\begin{enumerate}
		 			\item 
		 			For $i=1,\,2$, assume that  $\psi_i$ satisfies Assumption $\mathcal{H}^1$, $g$ satisfies Assumption $\mathcal{H}^0$ and $ 0\leq \psi_i \leq \Phi$ with $\Phi\in L^2([0,T];H^2(\mathcal{O},\m))$ such  that $\frac{\partial \Phi}{\partial t}+\mathcal{L}\Phi \leq 0$ and $0\leq g \leq -\frac{\partial \Phi}{\partial t} -\mathcal{L}^\lambda \Phi$. Let  $u^i_{\lambda}$  be the unique solution of \eqref{CVI}  with obstacle function $\psi_i$ and source function $g$.  If $\psi_1\leq \psi_2$, then $u^1_{\lambda}\leq u^2_{\lambda}$.
		 			\item For $i=1,\,2$, assume that  $\psi$ satisfies Assumption $\mathcal{H}^1$, $g_i$ satisfy Assumption $\mathcal{H}^0$ and $ 0\leq \psi \leq \Phi$ with $\Phi\in L^2([0,T];H^2(\mathcal{O},\m))$ such  that $\frac{\partial \Phi}{\partial t}+\mathcal{L}\Phi \leq 0$ and $0\leq g_i \leq -\frac{\partial \Phi}{\partial t} -\mathcal{L}^\lambda \Phi$. Let  $u^i_{\lambda}$  be the unique solution of \eqref{CVI}  with obstacle function $\psi$ and source function $g_i$. If $g_1\leq g_2$, then $u^1_{\lambda}\leq u^2_{\lambda}$.
		 			\item 	For $i=1,\,2$, assume that  $\psi_i$ satisfies Assumption $\mathcal{H}^1$, $g$ satisfies Assumption $\mathcal{H}^0$ and $ 0\leq \psi_i \leq \Phi$ with $\Phi\in L^2([0,T];H^2(\mathcal{O},\m))$ such  that $\frac{\partial \Phi}{\partial t}+\mathcal{L}\Phi \leq 0$ and $0\leq g \leq -\frac{\partial \Phi}{\partial t} -\mathcal{L}^\lambda \Phi$. Let  $u^i_{\lambda}$  be the unique solution of \eqref{CVI}  with obstacle function $\psi_i$ and source function $g$.  If $\psi_1- \psi_2\in L^\infty$, then $u^1_{\lambda}- u^2_{\lambda}\in L^\infty $ and $\Vert  u^1_{\lambda}- u^2_{\lambda}\Vert_\infty \leq   \Vert \psi_1- \psi_2\Vert_\infty$.
		 		\end{enumerate}
		 	\end{proposition}
		 	\subsubsection{Non-coercive variational inequality}
		We can finally prove Theorem \ref{variationalinequality}. Again, we first study the uniqueness of the solution and then we deal with the existence.
		 	\begin{proof}[Proof of uniqueness in Theorem \ref{variationalinequality}]
		 		Suppose that there are two functions $u_1 $ and $u_2$ which satisfy \eqref{VI}. As usual, we take $v=u_2$ in the equation satisfied by $u_1$ and $v=u_1$ in the one satisfied by $u_2$ and we add the resulting equations. 
		 		Setting $w:= u_2-u_1$,\label{key} 
		 		 we get that,  a.e. in $[0,T]$,
		 		$$
		 		\left( \frac{\partial w}{\partial t }(t),w(t) \right)_H - a(w(t),w(t))\geq 0. 
		 		$$
		 		From the energy estimate \eqref{ub}, we know that
		 		$$
		 		a(u(t),u(t)) \geq  C_1 \Vert u(t) \Vert^2_V -C_2\Vert (1 + y)^{\frac 1 2 }u(t)\Vert^2_H,
		 		$$
		 	so that
		 		$$
		 		\frac 1 2 \frac{d}{d t } \Vert w(t)\Vert_H^2+ C_2\Vert (1 + y)^{\frac 1 2 }w(t)\Vert^2_H \geq 0.
		 		$$
		 		By integrating from $t$ to $T$, since $w(T)=0$, we have
		 		\begin{align*}
		 		\Vert &w(t)\Vert_H^2\leq C_2\int_t^T \Vert (1 + y)^{\frac 1 2 }w(s)\Vert^2_Hds\\
		 		& \leq C_2 \bigg( \int_t^T ds \int_{\mathcal{O}} \mathbbm{1}_{\{y\leq \lambda\}}  (1 + y)w^2(s)d\m+ \int_t^T ds \int_{\mathcal{O}} \mathbbm{1}_{\{y> \lambda\}}(1 + y)w^2(s)d\m \bigg)\\
		 		&\leq C \bigg( \int_t^T ds \int_{\mathcal{O}}(1 + \lambda)w^2(s)y^{\beta-1} e^{-\gamma |x| } e^{-\mu y}dxdy + \int_t^T ds \int_{\mathcal{O}} \mathbbm{1}_{\{y> \lambda\}}(1 + y)w^2(s)y^{\beta-1} e^{-\gamma |x| } e^{-(\mu-\mu') y }e^{-\mu'y}dxdy \bigg)\\
		 		&\leq C \bigg( \int_t^T ds \int_{\mathcal{O}}dxdy (1 + \lambda)w^2(s)y^{\beta-1} e^{-\gamma |x| } e^{-\mu y}+e^{-(\mu-\mu') \lambda} \int_t^T ds \int_{\mathcal{O}}dxdy(1 + y)\Phi^2(s)y^{\beta-1} e^{-\gamma |x| } e^{-\mu'y}\bigg),
		 		\end{align*}
		 		where $\mu'<\mu$ and $\lambda>0$.
		 		Since $ C_2= \int_{\mathcal{O}}dxdy (1 + y)\Phi^2(s)y^{\beta-1} e^{-\gamma |x| } e^{-\mu' y }<\infty$, we have
		 		\begin{align*}
		 		\Vert w(t)\Vert_H^2&\leq C  (1+\lambda )\int_t^T  \Vert w(s)\Vert_H^2ds +C_2(T-t)e^{-(\mu-\mu') \lambda},
		 		\end{align*}
		 		so, by using the Gronwall Lemma, 
		 		$$
		 		\Vert w(t)\Vert_H^2\leq C_2Te^{-(\mu-\mu') \lambda+C(T-t)(1+\lambda)}.
		 		$$
		 	Sending  $\lambda\rightarrow \infty $, we deduce that $w(t)=0$ in $[T,t]$ for $t$ such that $T-t< \frac{\mu-\mu'}{C}$. Then, we iterate the same argument: we integrate between $t'$ and $t$ with $t-t'<\frac{\mu-\mu'}{C}$ and we have $w(t)=0$ in $[T,t']$  and so on. We deduce that $w(t)=0$ for all $t\in [0,T]$ so the assertion follows.
		 	\end{proof}

		 	\begin{proof}[Proof of existence in Theorem \ref{variationalinequality}]
		 		Given $u_0=\Phi$, we can construct a sequence $(u_n)_n\subset V$ such that
		 		
		 		\begin{equation}\label{claim1}
		 		u_n\geq \psi \mbox{ a.e. in } [0,T] \times \mathcal{O}, \qquad n\geq 1,
		 		\end{equation}
		 		\begin{equation}\label{claim2}
		 		\begin{split}
		 		-\left( \frac{\partial u_n}{\partial t },v -u_n  \right)_H +a(u_n,v-u_n) + \lambda ((1+y)u_n,v-u_n)_H \geq  \lambda ((1+y)u_{n-1},v-u_n)_H,  \\ v\in V, \quad v\geq \psi , \quad
	 \mbox{ a.e. on } [0,T]\times \mathcal{O}, \qquad n \geq 1,
	 \end{split}
		 		\end{equation}
		 		\begin{equation}
		 		u_n(T)=\psi(T), \qquad \mbox{ in } \mathcal{O},
		 		\end{equation}
		 		\begin{equation}\label{claim3}
		 		\Phi\geq u_1\geq u_2\geq \dots \geq u_{n-1}\geq u_n \geq \dots \geq 0 ,\qquad  \mbox{ a.e. on } [0,T]\times \mathcal{O}.
		 		\end{equation}
		 		In fact, if   we have $0\leq u_{n-1} \leq \Phi$ for all $ n\in \N$, then the assumptions of Proposition \ref{coercive variational inequality} are satisfied with 
		 		$$
		 		g_n= \lambda(1+y)u_{n-1}.
		 		$$
		 		Indeed, since $(1+y)^{\frac 3 2 }\Phi\in L^2([0,T];H)$, we have that  $g_n$ and $\sqrt{1+y}g_n$ belong to $ L^2([0,T];H) $  and, moreover, $0\leq g_n \leq \lambda(1+y)\Phi \leq - \frac{\partial \Phi}{\partial t} -\L_\lambda \Phi$. Therefore, step by step, we can deduce the existence and the uniqueness of a solution $u_n$ to \eqref{claim2} such that $0\leq u_n \leq \Phi$.  \eqref{claim3} is a simple consequence of Proposition \ref{CompPrinc2}. In fact, proceeding by induction, at each step we have 
		 		$$
		 		g_n=	\lambda(1+y)u_{n-1} \leq 	\lambda(1+y)u_{n-2}=g_{n-1}
		 		$$
		 		so that $u_{n} \leq 	u_{n-1}$.
		 		Now, recall that 
		 		\begin{equation*}
		 		\Vert u_{n} \Vert_{L^\infty([0,T],V)}\leq K,
		 		\end{equation*}
		 		\begin{equation*}
		 		\left\Vert \frac{\partial u_{n}}{\partial t }\right \Vert_{L^2([0,T];H)}\leq K,
		 		\end{equation*}
		 		where $K=C \left(    \Vert \Psi \Vert_{L^2([0,T];V)} + \Vert \sqrt{1+y}g_n\Vert_{L^2([0,T];H)}  + \Vert\sqrt{1+y}\psi \Vert_{L^2([0,T];V)}+\Vert\psi(T)\Vert_V \right)$.
		 		Note that the constant $K$ is independent of $n$ since $|g_n|=|\lambda(1+y)u_{n-1},|\leq \lambda (1+y)\Phi,$ for every $n \in \N.$
		 		Therefore, by passing to a subsequence, we can assume that there exists a function $u$ such that $u\in L^2([0,T];V) $, $ \frac{\partial u }{\partial t} \in L^2([0,T];H) $ and for every $t\in [0,T]$, $u'_n(t) \rightharpoonup u'(t) $ in $H$ and $u_n (t) \rightharpoonup u(t) $ in $V$. Indeed, again thanks to the fact that $0\leq u_n \leq \Phi$, we can deduce that $u_n(t)\rightarrow u(t)$ in $H$.  Therefore we can pass to the limit in 
		 		\begin{align*}
		 		-\left( \frac{\partial u_n}{\partial t },u_n-v \right)_H +a(u_n,v-u_n) + \lambda ((1+y)u_n,v-u_n)_H  \geq  \lambda((1+y)u_{n-1},v-u_n)_H
		 		\end{align*}
		 		and the assertion follows.
		 	\end{proof}
\begin{remark}
Keeping in mind our purpose of identifying  the solution of the variational inequality \eqref{VI} with the American option price we have considered the case without source term ($g=0$) in the variational inequality \eqref{VI}. However, under the same assumptions of Theorem \ref{variationalinequality}, we can prove in the same way the existence and the uniqueness of a solution of 
	\begin{equation*}
	\begin{cases}
	-\left( \frac{\partial u}{\partial t },v -u  \right)_H + a(u,v-u)\geq (g,v-u)_H, \quad \mbox{a.e. in } [0,T] \quad v\in  L^2([0,T];V), \ v\geq \psi,\\
	u\geq \psi \mbox{ a.e. in } [0,T]\times \R \times (0,\infty),\\
	u(T)=\psi(T),\\
	0\leq u \leq \Phi,
	\end{cases}
	\end{equation*}
where $g$ satisfies Assumption $\mathcal{H}^0$  and $0\leq g\leq -\frac{\partial \Phi}{\partial t }-\L\Phi$.
\end{remark}
		We conclude stating the following Comparison Principle, whose proof  is a direct consequence of  Proposition \ref{CompPrinc2} and the proof of Proposition \ref{variationalinequality}.
		 			 	\begin{proposition}\label{CompPrinc3}
		 			 	For $i=1,2$, assume that $\psi_i$ satisfies Assumption $\mathcal{H}^1$ and $0\leq \psi_i\leq \Phi$ with  $\Phi$ satisfying Assumption $\mathcal{H}^2$. Let $u^i_{\lambda}$  be the unique solution of \eqref{CVI} with obstacle function $\psi_i$. Then:
		 			 		\begin{enumerate}
		 			 			\item   If $\psi_1\leq \psi_2$, then $u^1_{\lambda}\leq u^2_{\lambda}$.
		 			 			\item  If $\psi_1- \psi_2\in L^\infty$, then $u^1_{\lambda}- u^2_{\lambda}\in L^\infty $ and $\Vert  u^1_{\lambda}- u^2_{\lambda}\Vert_\infty \leq   \Vert \psi_1- \psi_2\Vert_\infty$.
		 			 		\end{enumerate}
		 			 	\end{proposition}
		 	
 		 	\section{Connection with the optimal stopping problem}
		 	Once we have the existence and the uniqueness of a solution $u$ of the variational inequality \eqref{variationalinequality}, our aim is to prove that  it matches the solution of the optimal stopping problem, that is
		 	\begin{equation*}
		 	u(t,x,y)=u^*(t,x,y), \qquad \mbox{ on } [0,T] \times \bar{\mathcal{O}},
		 	\end{equation*}
		 	where $u^*$ is defined by
		 	$$
		 	u^*(t,x,y)= \sup_{\tau \in \mathcal{T}_{t,T}}\E \left[  \psi(\tau,X_\tau^{t,x,y}, Y_\tau^{t,y})     \right],
		 	$$ 
		 $	\mathcal{T}_{t,T}$ being the set of the stopping times with values in $[t,T]$.
		Since the function $u$ is not regular enough to apply It\^{o}'s Lemma, we use another strategy in order to prove the above identification.
So, we first show, by using the affine character of the underlying diffusion, that the semigroup associated with the bilinear form $a_\lambda$ coincides with the transition semigroup of the two dimensional diffusion $(X,Y)$ with a killing term. Then, we prove suitable estimates on the joint law of $(X,Y)$ and $L^p$-regularity results on the solution of the variational inequality and we deduce from them the probabilistic interpretation.

		 	\subsection{Semigroup associated with the bilinear form }

We introduce now the semigroup associated with the coercive  bilinear form $a_\lambda$. With a natural notation, we define the following spaces
$$
L^2_{loc}(\R^+;H)=\{f:\R^+\rightarrow H : \forall t\geq 0 \int_0^t \|f(s)\|_H^2ds<\infty  \},
$$
$$
L^2_{loc}(\R^+;V)=\{f:\R^+\rightarrow V :  \forall t\geq 0 \int_0^t \|f(s)\|_V^2ds<\infty   \}.
$$
First of all, we  state the following result:
	\begin{proposition}\label{propsg2}
For every $\psi\in V$,
 $f\in L^2_{loc}(\R^+;H)$ with $\sqrt{y}f\in L^2_{loc}(\R^+;H)$,   there exists a unique function $u\in L^2_{loc}(\R^+;V)$ such that
		$\frac{\partial u}{\partial t}\in L^2_{loc}(\R^+;H)$, $u(0)=\psi$ and
	\begin{equation}\label{var_eq}
		\left(\frac{\partial u}{\partial t},v\right)_H+a_\lambda(u,v)=(f,v)_H, \quad v\in V.
\end{equation}
		Moreover we have, for every $t\geq 0$,
		\begin{equation}
		\Vert u(t)\Vert _H^2+\frac{\delta_1}{2}\int_0^t\Vert u(s)\Vert ^2_Vds\leq \Vert \psi \Vert ^2_H
		+\frac{2}{\delta_1} \int_0^t \Vert f(s)\Vert _H^2ds \label{estimL2bis}
		\end{equation}
		and
		$$
		||u(t)||^2_V+	\int_0^t||u_t(s)||^2_Hds \leq C\left(||\psi||_V^2+\frac{1}{2}\int_0^t||\sqrt{1+y}f(s)||^2_H
			ds \right),
$$
with $C>0$.
	\end{proposition}
	The proof follows the same lines as the proof of Proposition \ref{coercive variational inequality} so we omit it. Moreover, we can prove a Comparison Principle for the equation \eqref{var_eq} as we have done for the variational inequality.  
	
	We denote $u(t)=\bar{P}^\lambda_t\psi$ the solution of \eqref{var_eq} corresponding to $u(0)=\psi$ and $f=0$. From \eqref{estimL2bis} we deduce that the operator $\bar{P}^\lambda_t$ is a linear contraction on $H$ and, from uniqueness, we have the semigroup property. 
	\begin{proposition}\label{propsg3}
		Let us consider $f:\R^+\to H$ such that
		$\sqrt{1+y}f\in L^2_{loc}(\R^+,H)$.  Then, the solution of
		\[
	\begin{cases}
 \left(\frac{\partial u}{\partial t},v\right)_H+a_\lambda(u,v)=(f,v)_H,\quad v\in V,\\
		u(0)=0,
			\end{cases}
		\]
		is given by $ u(t)=\int_0^t  \bar P^\lambda_sf(t-s)ds=\int_0^t\bar P^\lambda_{t-s}f(s)ds$.
	\end{proposition}
	\begin{proof}
		Note that $V$ is dense in $H$ and recall the estimate \eqref{estimL2bis}, so  it is enough to prove the assertion for  $f=\ind{(t_1,t_2]}\psi$, with $0\leq t_1<t_2$ and $\psi\in V$. If we set $u(t)=\int_0^t\bar P^\lambda_{t-s}f(s)ds$, we have
		\begin{align*}
			u(t)&=\ind{\{t\geq t_1\}}\int_{t_1}^{t\wedge t_2}\bar P^\lambda_{t-s}\psi ds\\
		&	=\begin{cases}
		\int_{t_1}^{t_2}\bar P^\lambda_{t-s}\psi ds=\int_{t-t_2}^{t-t_1}\bar P^\lambda_{s}\psi ds\quad&\mbox{ if } t\geq t_2\\
					\displaystyle \int_{t_1}^{t}\bar P^\lambda_{t-s}\psi ds=\int_{0}^{t-t_1}\bar P^\lambda_{s}\psi ds \quad &\mbox{ if } t\in[t_1, t_2)
		\end{cases}.
		\end{align*}
		Therefore, for every $v\in V$, we have $(u_t,v)_H+a_\lambda(u,v)=0$ if $t\leq t_1$ and, if $t\geq t_1$,
	$$
			\left(\frac{\partial u}{\partial t},v\right)_H+a_\lambda(u(t),v)=
		\begin{cases}
		\left(\bar P^\lambda_{t-t_1}\psi-\bar P^\lambda_{t-t_2}\psi, v\right)_H+
		a_\lambda\left(\int_{t-t_2}^{t-t_1}\bar P^\lambda_{s}\psi ds, v\right)\quad &\mbox{ if } t\geq t_2\\
		\left(\bar P^\lambda_{t-t_1}\psi, v\right)_H+
		a_\lambda\left(\int_{0}^{t-t_1}\bar P^\lambda_{s}\psi ds, v\right)  \quad&\mbox{ if } t\in[t_1, t_2)
		\end{cases}.
$$
		The assertion follows from $(\bar P^\lambda_t\psi,v)_H+\int_0^t a_\lambda(\bar P_s\psi,v)ds  =(\psi,v)_H$.
	\end{proof} 
	\begin{remark}
		It is not difficult to prove that  $\bar{P}^\lambda_t:L^p(\mathcal{O},\m) \rightarrow L^p(\mathcal{O},\m) $  is a contraction for every $ p\geq 2$, and it is an analytic semigroup. This is not useful to our purposes so we omit the proof.
	\end{remark}
\subsection{Transition semigroup}
	We define $\E_{x_0,y_0}(\quad)= \E(\quad| X_0=x_0, Y_0=y_0)$. 
	Fix $\lambda >0$.  For every measurable positive function $f$ defined on $\R\times [0,+\infty)$, we define
	\[
	P^\lambda_tf(x_0,y_0)=\E_{x_0,y_0}\left(e^{-\lambda\int_0^t(1+Y_s)ds} f(X_t,Y_t)\right).
	\]
The operator	$	P^\lambda_t$ is the transition semigroup of the two dimensional diffusion $(X,Y)$ with the killing term $e^{-\lambda\int_0^t(1+Y_s)ds}$.

	Set $\E_{y_0}(\quad)= \E(\quad| Y_0=y_0)$. 
We first prove some useful results about the Laplace transform of the pair $(Y_t,\int_0^tY_sds)$. These results rely on the affine structure of the model and have already appeared in slightly different forms in the literature (see, for example, \cite[Section 4.2.1]{Abook}). We include a proof for convenience.
	\begin{proposition}\label{Laplace}
		Let $z$  and  $w$ be two complex numbers with nonpositive real parts. The equation
		\begin{equation}\label{*}
		\psi'(t)=\frac{\sigma^2}{2}\psi^2(t)-\kappa \psi(t)+w
		\end{equation}
		has a unique solution $ \psi_{z,w}$ defined on $[0,+\infty)$, such that $\psi_{z,w}(0)=z$. Moreover,
		for every $t\geq 0$,
		\[
		\E_{y_0}\left( e^{z Y_t+w \int_0^t Y_s ds}\right)=e^{y_0\psi_{z,w}(t) +\theta\kappa\phi_{z,w}(t)},
		\]
		with $\phi_{z,w}(t)=\int_0^t \psi_{z,w}(s)ds$.
	\end{proposition}
	\begin{proof}
		Let $\psi$ be the solution of \eqref{*}. We define $\psi_1$ (resp. $w_1$) and $\psi_2$
		(resp. $w_2$) the real and the imaginary part of
		$\psi$ (resp. $w$). We have
		\[
		\left\{
		\begin{array}{l}
		\psi'_1(t)=\frac{\sigma^2}{2}\left(\psi_1^2(t)-\psi_2^2(t)\right)-\kappa \psi_1(t)+w_1,\\
		\psi'_2(t)=\sigma^2\psi_1(t)\psi_2(t)-\kappa \psi_2(t)+w_2.
		\end{array}
		\right.
		\]
		From the first equation we deduce that $ \psi'_1(t)\leq \frac{\sigma^2}{2}\left(\psi_1(t)-\frac{2\kappa}{\sigma^2}\right) \psi_1(t)+w_1$
		and, since $w_1\leq 0$, the function $t\mapsto \psi_1(t)e^{-\frac{\sigma^2}2\int_0^t(\psi_1(s)-\frac{2\kappa}{\sigma^2})ds}$ is nonincreasing. Therefore $\psi_1(t)\leq 0$
	if $\psi_1(0)\leq 0$. Multiplying the first equation by $\psi_1(t)$ and the second one by $\psi_2(t)$ and adding we get
		\begin{eqnarray*}
		\frac 1 2 	\frac{d}{dt}\left(|\psi(t)|^2\right)&=&\left(\frac{\sigma^2}{2}\psi_1(t)-\kappa\right)|\psi(t)|^2+w_1\psi_1(t)+w_2\psi_2(t)\\
			&\leq &\left(\frac{\sigma^2}{2}\psi_1(t)-\kappa\right)|\psi(t)|^2+|w||\psi(t)|\\
			&\leq &\left(\frac{\sigma^2}{2}\psi_1(t)-\kappa\right)|\psi(t)|^2+\epsilon |\psi(t)|^2+\frac{|w|^2}{4\epsilon}.
		\end{eqnarray*}
		We deduce that $|\psi(t)|$ cannot explode in finite time and, therefore,  $\psi_{z,w}$
		actually exists on $[0, +\infty)$. 
		
		Now, let us define the function  $F_{z,w}(t,y)=e^{y\psi_{z,w}(t) +\theta\kappa\phi_{z,w}(t)}$.
		$F_{z,w}$ is  $C^{1,2}$ on $[0,+\infty)\times \R$ and it satisfies by construction the following equation
		\[
		\frac{\partial F_{z,w}}{\partial t}=\frac{\sigma^2}{2}y\frac{\partial^2 F_{z,w}}{\partial y^2}+
		\kappa(\theta-y)\frac{\partial F_{z,w}}{\partial y}+w y F_{z,w}.
		\]
		Therefore, for every $T>0$, the process $(M_t)_{0\leq t\leq T}$ defined by \begin{equation}\label{Mt}
M_t=e^{w\int_0^t Y_sds}F_{z,w}(T-t, Y_t)
\end{equation} 
is a local martingale. On the other hand, note that 
$$
|M_t|= \left|  e^{w\int_0^t Y_sds} \right|\left|  e^{Y_t\psi_{z,w}(T-t) +\theta\kappa\phi_{z,w}(T-t)} \right|\leq 1
$$ 
since $w$, $\psi_{z,w}(t)$ and $\phi_{z,w}(t)=\int_0^t\psi_{z,w}(s)ds$ all have nonpositive real parts. Therefore the process $(M_t)_t$ is a  true martingale indeed. We deduce that 
		$F_{z,w}(T,y_0)=\E_{y_0}\left(e^{w\int_0^T Y_sds}e^{zY_T}\right)$ and the assertion follows.
	\end{proof}
		We also have the following result which specifies the behaviour of the Laplace transform  of $(Y_t,\int_0^tY_sds)$ when evaluated in two real numbers, not necessarily nonpositive.
	\begin{proposition}\label{Laplace2}
		Let $\lambda_1$  and $\lambda_2$ be two real numbers such that 
		\[
		\frac{\sigma^2}{2}\lambda_1^2-\kappa\lambda_1+\lambda_2\leq 0.
		\]
		Then, the equation
		\begin{equation}\label{**}
		\psi'(t)=\frac{\sigma^2}{2}\psi^2(t)-\kappa \psi(t)+\lambda_2
		\end{equation}
		has a  unique solution $ \psi_{\lambda_1,\lambda_2}$ defined on $[0,+\infty)$ such that $\psi_{\lambda_1,\lambda_2}(0)=\lambda_1$. Moreover, for every  $t\geq 0$, we have
		\[
		\E_{y_0}\left( e^{\lambda_1 Y_t+\lambda_2 \int_0^t Y_s ds}\right)\leq 
		e^{y_0\psi_{\lambda_1,\lambda_2}(t) +\theta\kappa\phi_{\lambda_1,\lambda_2}(t)},
		\]
		with $\phi_{\lambda_1,\lambda_2}(t)=\int_0^t \psi_{\lambda_1,\lambda_2}(s)ds$.
	\end{proposition} 
	\begin{proof}
	Let $\psi$ be the solution of \eqref{**} with $\psi(0)=\lambda_1$. We have
		\[
		\psi''(t)=(\sigma^2 \psi(t)-\kappa)\psi'(t).
		\]
		Therefore, the function $t\mapsto \psi'(t)e^{-\int_0^t(\sigma^2\psi(s)-\kappa)ds}$is a constant, hence $\psi'(t)$ has constant sign.  Moreover, the assumption on $\lambda_1$
		and $\lambda_2$ ensures that  $\psi'(0)\leq 0$. We deduce that $\psi'(t)\leq 0$
		and $\psi(t)$ remains between the solutions of the equation
		\[
		\frac{\sigma^2}{2}\lambda^2-\kappa\lambda+\lambda_2=0.
		\]
		This proves that the solution is defined on the whole interval
		$[0,+\infty)$. Now the assertion follows as in the proof of Proposition \ref{Laplace}: just note that  the process $(M_t)_t$ defined as in \eqref{Mt} is no more uniformly bounded, so we cannot directly deduce that it is a martingale. However,  it remains a positive local martingale, hence a supermartingale.
			\end{proof}
			
		\begin{remark}\label{remarksemigroup}
	Let us now consider	two real numbers $\lambda_1$  and $\lambda_2$  such that 
	\[
	\frac{\sigma^2}{2}\lambda_1^2-\kappa\lambda_1+\lambda_2<0.
	\]	 From the proof of Proposition \ref{Laplace2}, by using the optional sampling theorem we have
		\begin{eqnarray*}
			\sup_{\tau\in\T_{0,T}}\E_y\left(e^{\lambda_2\int_0^\tau Y_sds}e^{\psi_{\lambda_1,\lambda_2}(T-\tau)Y_\tau+\theta\kappa\phi_{\lambda_1,\lambda_2}(T-\tau)}\right)
			&\leq&e^{y\psi_{\lambda_1,\lambda_2}(T)+\theta\kappa\phi_{\lambda_1,\lambda_2}(T)}.
		\end{eqnarray*}
		
		Consider now $\epsilon>0$ and let $\lambda_1^\epsilon=(1+\epsilon)\lambda_1$ and $\lambda_2^\epsilon = (1+\epsilon)\lambda_2$.
		For $\epsilon$ small enough, we have $ \frac{\sigma^2}{2}(\lambda_1^\epsilon)^2-\kappa\lambda_1^\epsilon+\lambda_2^\epsilon< 0$.
		Therefore
		\begin{eqnarray*}
			\sup_{\tau\in\T_{0,T}}\E_y\left(e^{\lambda_2^\epsilon\int_0^\tau Y_sds}e^{\psi_{\lambda_1^\epsilon,\lambda_2^\epsilon}(T-\tau)Y_\tau+
				\theta\kappa\phi_{\lambda_1^\epsilon,\lambda_2^\epsilon}(T-\tau)}\right)
			&\leq&e^{y\psi_{\lambda_1^\epsilon,\lambda_2^\epsilon}(T)+\theta\kappa\phi_{\lambda_1^\epsilon,\lambda_2^\epsilon}(T)}.
		\end{eqnarray*}
		If we have $\psi_{\lambda_1^\epsilon,\lambda_2^\epsilon}\geq (1+\epsilon)\psi_{\lambda_1,\lambda_2}$,
		we can deduce that
		\begin{eqnarray*}
			\sup_{\tau\in\T_{0,T}}\E_y\left(e^{\lambda_2(1+\epsilon)\int_0^\tau Y_sds}
			e^{(1+\epsilon)\left(\psi_{\lambda_1,\lambda_2}(T-\tau)Y_\tau+\theta\kappa\phi_{\lambda_1,\lambda_2}(T-\tau)\right)}\right)
			&\leq&e^{y\psi_{\lambda_1^\epsilon,\lambda_2^\epsilon}(T)+\theta\kappa\phi_{\lambda_1^\epsilon,\lambda_2^\epsilon}(T)},
		\end{eqnarray*}
		and, therefore, that the family  $\left(e^{\lambda_2\int_0^\tau Y_sds}e^{\psi_{\lambda_1,\lambda_2}(T-\tau)Y_\tau+\theta\kappa\phi_{\lambda_1,\lambda_2}(T-\tau)}\right)_{\tau\in\T_{0,T}}$
		is uniformly integrable. As a consequence, the process $(M_t)_t$ is a true  martingale and we have
		\[
		\E_{y}\left( e^{\lambda_1 Y_t+\lambda_2 \int_0^t Y_s ds}\right) = 
		e^{y\psi_{\lambda_1,\lambda_2}(t) +\theta\kappa\phi_{\lambda_1,\lambda_2}(t)}.
		\]
		So, it remains to show  that $\psi_{\lambda_1^\epsilon,\lambda_2^\epsilon}\geq (1+\epsilon)\psi_{\lambda_1,\lambda_2}$. In order to do this
		we set $g_\epsilon(t)=\psi_{\lambda_1^\epsilon,\lambda_2^\epsilon}(t)-(1+\epsilon)\psi_{\lambda_1,\lambda_2}(t)$.
		From the equations satisfied by $\psi_{\lambda_1^\epsilon,\lambda_2^\epsilon}$ and
		$\psi_{\lambda_1,\lambda_2}$ we deduce that
		\begin{eqnarray*}
			g_\epsilon'(t)&=&\frac{\sigma^2}{2}\left(\psi^2_{\lambda_1^\epsilon,\lambda_2^\epsilon}(t)-(1+\epsilon)
			\psi^2_{\lambda_1,\lambda_2}(t)\right)-\kappa \left(\psi_{\lambda_1^\epsilon,\lambda_2^\epsilon}(t)-
			(1+\epsilon) \psi_{\lambda_1,\lambda_2}(t)\right)\\
			&=&
			\frac{\sigma^2}{2}\left(\psi^2_{\lambda_1^\epsilon,\lambda_2^\epsilon}(t)-(1+\epsilon)^2
			\psi^2_{\lambda_1,\lambda_2}(t)\right)
			-\kappa g_\epsilon(t)+\frac{\sigma^2}{2}\left((1+\epsilon)^2-(1+\epsilon)\right)\psi^2_{\lambda_1,\lambda_2}(t)\\
			&=&\frac{\sigma^2}{2}\left(\psi_{\lambda_1^\epsilon,\lambda_2^\epsilon}(t)+(1+\epsilon)
			\psi_{\lambda_1,\lambda_2}(t)\right)g_\epsilon(t)
			-\kappa g_\epsilon(t)+\frac{\sigma^2}{2}\epsilon(1+\epsilon)\psi^2_{\lambda_1,\lambda_2}(t)\\
			&=&f_\epsilon(t) g_\epsilon(t)+\frac{\sigma^2}{2}\epsilon(1+\epsilon)\psi^2_{\lambda_1,\lambda_2}(t),
		\end{eqnarray*}
		where
		\[
		f_\epsilon(t)=\frac{\sigma^2}{2}\left(\psi_{\lambda_1^\epsilon,\lambda_2^\epsilon}(t)+(1+\epsilon)
		\psi_{\lambda_1,\lambda_2}(t)\right)
		-\kappa.
		\]
		Therefore, the function $g_\epsilon(t)e^{-\int_0^t f_\epsilon(s)ds}$ is nondecreasing  and, since $g_\epsilon(0)=0$,
		we have $g_\epsilon(t)\geq 0$.
\end{remark}    

We can now prove the following Lemma, which will be useful in  Section \ref{sect-estimjointlaw} to  prove suitable estimates on the joint law of the process $(X,Y)$.
	\begin{lemma}\label{RemarqueMomentsNeg}
		For every $q>0$  there exists $C>0$ such that for all $y_0\geq 0$,
		\begin{equation}
			\E_{y_0}\left(\int_0^t Y_v dv\right)^{-q}\leq \frac{C}{t^{2q}}.
		\end{equation}
	\end{lemma}
		\begin{proof}
		If we take $\lambda_1=0$ and $\lambda_2=-s$ with $s>0$  in Proposition \ref{Laplace2},
		we get
		\[
		\E_{y_0}\left( e^{-s \int_0^t Y_v dv}\right)=e^{y_0\psi_{0,-s}(t) +\theta\kappa\phi_{0,-s}(t)}.
		\]
		Since $\psi'_{0,-s}(0)=-s<0$, we can deduce by the  proof of Proposition \ref{Laplace2} that $\psi'_{0,-s}(t)=-se^{\int_0^t(\sigma^2\psi(u)-\kappa)du}$. Therefore, since $\psi_{0,-s}=0$, we have  
\begin{equation}\label{psi0-s}
\psi_{0,-s}(t)=-s\int_0^te^{\int_0^u(\sigma^2\psi(v)-\kappa)dv}du.
\end{equation}
 Again from the proof of Proposition \ref{Laplace2},
		\[
		\psi_{0,-s}(t)\geq 
		\frac{\kappa}{\sigma^2}-\sqrt{\left(\frac{\kappa}{\sigma^2}\right)^2+2\frac{s}{\sigma^2}}
		\geq -\sqrt{2s/\sigma^2},
		\]
so, by using \eqref{psi0-s}, we deduce that
		$$
		\psi_{0,-s}(t)\leq 
	-s\int_0^te^{\int_0^u-(\sigma \sqrt{2s}+\kappa)dv}du=-s\int_0^te^{-\lambda_su}du=-\frac{s}{\lambda_s}(1-e^{-t\lambda_s}).
$$	
where $\lambda_s=\sigma\sqrt{2s}+\kappa$. Since $\phi_{0,-s}(t)=\int_0^t \psi_{0,-s}(u)du$, we have
$$
\phi_{0,-s}(t)\leq-\frac{s}{\lambda^2_s} \left( t\lambda_s   -1+e^{-t\lambda_s}\right).
$$
Therefore, since $\psi_{0,-s}(t)\leq 0$,  for any $y_0\geq 0$ we get
\begin{align*}
\E_{y_0}\left(e^{-s\int_0^tY_vdv}\right)&\leq e^{\kappa\theta\phi_{0,-s}(t)}\leq e^{ -\frac{\kappa\theta s}{\lambda_s^2} (t\lambda_s-1+e^{-t\lambda_s})    } .
\end{align*}
		Now, recall that for every $q>0$ we can write
		$$
		\frac{1}{y^q}=\frac{1}{\Gamma(q)}\int_0^\infty s^{q-1}e^{-sy}ds.
		$$
	Therefore
		\begin{align*}
			\E_{y_0}\left(\int_0^t Y_v dv\right)^{-q}&=\E_{y_0}\left(\frac{1}{\Gamma(q)}\int_0^\infty s^{q-1}
			e^{-s\int_0^t Y_v dv}ds\right)\\
			&\leq\frac{1}{\Gamma(q)}\int_0^1s^{q-1}
	e^{  -\frac{\kappa\theta s}{\lambda_s^2} (t\lambda_s-1+e^{-t\lambda_s})} ds+ \frac{1}{\Gamma(q)}\int_1^\infty s^{q-1}
		e^{  -\frac{\kappa\theta s}{\lambda_s^2} (t\lambda_s-1+e^{-t\lambda_s})} ds.
		\end{align*}
		Recall that $\lambda_s=\sigma\sqrt{2s}+\kappa$, so the first terms in the right hand side is finite. Moreover, for $s>1$, we have $\frac{\kappa\theta s}{\lambda_s^2}\leq C$.  Then, by noting that the function $u\mapsto tu-1+e^{-tu}$ is nondecreasing, we have
			\begin{align*}
	\E_{y_0}\left(\int_0^t Y_v dv\right)^{-q}&\leq C+ \frac{1}{\Gamma(q)}\int_1^\infty s^{q-1}
		e^{  -C (t\sigma\sqrt{2s}-1+e^{-t\sigma\sqrt{2s}})} ds\\
		&\leq C+ \frac{1}{t^{2q}\Gamma(q)}\int_0^\infty v^{q-1}
		e^{  -C (\sigma\sqrt{2v}-1+e^{-\sigma\sqrt{2v}})} dv\\
		&\leq \frac{C}{t^{2q}},
		\end{align*}
		which concludes  the proof.
	\end{proof}

	Now recall that the diffusion $(X,Y)$ evolves according to the following stochastic differential system
	
	\begin{equation*}
	\begin{cases}
	dX_t=\left( \frac{\rho\kappa\theta}{\sigma}-\frac{Y_t}{2}\right)dt + \sqrt{Y_t}dB_t,\\
	dY_t=\kappa(\theta-Y_t)dt+\sigma\sqrt{Y_t}dW_t.
	\end{cases}
	\end{equation*}	 	
	If we set $\tilde{X}_t=X_t-\frac{\rho}{\sigma}Y_t$, we have 
\begin{equation}
\label{model2}
\begin{cases}
d\tilde{X}_t=\left( \frac{\rho\kappa}{\sigma}-\frac{1}{2}\right)Y_tdt+\sqrt{1-\rho^2}\sqrt{Y_t} d\tilde{B}_t,\\
dY_t=\kappa(\theta-Y_t)dt+\sigma\sqrt{Y_t}dW_t.
\end{cases}
\end{equation}
	where $\tilde{B}_t=(1-\rho^2)^{-1/2}\left(B_t-\rho W_t\right)$. Note that $\tilde B $ is a standard Brownian motion with $\langle \tilde{B},W\rangle_t=0$.
		\begin{proposition}\label{prop-TF}
			For all  $u, \, v \in \R$, for all $\lambda\geq 0$
			and for all $(x_0,y_0)\in \R\times [0,+\infty)$ we have
			\[
			\E_{x_0,y_0}\left( e^{iu X_t+ivY_t}e^{-\lambda\int_0^t Y_s  ds}\right)=
			e^{iu x_0+y_0(\psi_{\lambda_1,\mu}(t)-iu\frac{\rho}{\sigma})+\theta\kappa\phi_{\lambda_1,\mu}(t)},
			\]
			where $\lambda_1= i(u\frac{\rho}{\sigma}+v)$, $\mu=iu\left(\frac{\rho\kappa}{\sigma}-\frac{1}{2}\right)-\frac{u^2}{2}(1-\rho^2)-\lambda$
			and the function $\psi_{\lambda_1,\mu}$
			and $\phi_{\lambda_1,\mu}$ are defined in Proposition \ref{Laplace}.
		\end{proposition}
		\begin{proof}
			We have
			\[
			\E_{x_0,y_0}\left( e^{iu X_t+ivY_t-\lambda\int_0^t Y_s  ds}\right)=\E_{x_0,y_0}\left( e^{iu (\tilde{X_t}+\frac{\rho}{\sigma}Y_t)+ivY_t-\lambda\int_0^t Y_s  ds}\right)
			\]
			and 
			\[
			\tilde{X_t}=x_0-\frac{\rho}{\sigma}y_0+\int_0^t \left( \frac{\rho\kappa}{\sigma}-\frac{1}{2}\right)Y_sds+\int_0^t\sqrt{(1-\rho^2)Y_s}d\tilde{B}_s.
			\]
			Since $\tilde{B}$ and $W$ are independent,
			\[
			\E\left(e^{iu\tilde{X}_t} \;|\; W\right)=e^{iu\left(x_0-\frac{\rho}{\sigma}y_0+\int_0^t \left(\frac{\rho\kappa}{\sigma}-\frac{1}{2}\right)Y_sds\right)
				-\frac{u^2}{2}(1-\rho^2)\int_0^t Y_s ds}
			\]
			and
			\begin{eqnarray*}
				\E_{x_0,y_0}\left( e^{iu X_t+ivY_t-\lambda\int_0^t Y_s  ds}\right)=e^{iu\left(x_0-\frac{\rho}{\sigma}y_0\right)}
				\E_{y_0}\left(e^{i\left(u\frac{\rho}{\sigma}+v\right)Y_t+
					\left(iu(\frac{\rho\kappa}{\sigma}-\frac{1}{2})-\frac{u^2}{2}(1-\rho^2)-\lambda\right)\int_0^tY_sds}\right).
			\end{eqnarray*}
			Then the assertion follows by using  Proposition~\ref{Laplace}.
		\end{proof}

\subsection{Identification of the semigroups}
We now show that the semigroup $\bar P^\lambda_t$ associated with the coercive bilinear form  can be actually identified with the transition semigroup $ P^\lambda_t$.  Recall  the Sobolev  spaces $L^p(\O,\m_{\gamma,\mu})$ introduced in Definition \ref{def-sob} for $p\geq 1$. In order to prove the identification of the semigroups, we need the following property of the transition semigroup.
\begin{theorem}\label{theotransition}
	For all $p>1$, $\gamma>0$ and $\mu>0$ there exists $\lambda >0$ such that,
	for every compact $K\subseteq\R\times[0,+\infty)$ and for every  $T>0$, there is $C_{p,K, T}>0$ such that
	\[
	P^\lambda_t f(x_0,y_0)\leq \frac{C_{p,K, T}}{t^{\frac{\beta}{p}+\frac{3}{2p}}}||f||_{L^p(\O,\m_{\gamma,\mu})},\qquad (x_0,y_0)\in K.
	\]
	for every measurable positive function  $f$ on $\R\times[0,+\infty)$ and for every $t\in(0,T]$.
\end{theorem}
Theorem \ref{theotransition} will  also play a crucial role in order to prove Theorem \ref{theorem2}.   Its proof relies on suitable estimates on the joint law of the diffusion $(X,Y)$ and we postpone it to the following section. Then, we can prove the following result.
\begin{proposition}\label{identificationsemigroup}
	There exists $\lambda>0$ such that, for every function $f\in H$ and for every $t\geq 0$,
	\[
	\bar{P}^\lambda_t f(x,y)=P^\lambda_tf(x,y),\quad dxdy \mbox{ a.e.}
	\]
\end{proposition}
\begin{proof}
	We can easily deduce from Theorem \ref{theotransition} with $p=2$ that, for $\lambda$ large enough, if $(f_n)_n$ is a sequence of functions which converges to $f$ in $H$, then the sequence  $(P^\lambda_tf_n)_n$ converges uniformly to $P^\lambda_tf$ on the compact sets. 	On the other hand, recall  that $\bar P^\lambda_t$ is a contraction  semigroup on $H$ so that  the function $f\mapsto \bar P^\lambda_tf$ is continuous and we have $\bar P^\lambda_t f_n \rightarrow \bar P^\lambda_t f$ in $H$.

	Therefore, by density arguments, it is enough to prove the equality for $
	f(x,y)=e^{iux+ivy}$
	with $u$, $v\in \R$. We  have, by using  Proposition \ref{prop-TF},
	\begin{eqnarray*}
		P^\lambda_tf(x,y)&=&\E_{x,y}\left(e^{-\lambda\int_0^t(1+Y_s)ds} e^{iuX_t+ivY_t}\right)\\
		&=& e^{-\lambda t}e^{iu x+y\left(\psi_{\lambda_1,\mu}(t)-iu\frac{\rho}{\sigma}\right)+\theta\kappa\phi_{\lambda_1,\mu}(t)},
	\end{eqnarray*}
	with $\lambda_1= i(u\frac{\rho}{\sigma}+v)$, $\mu=iu\left(\frac{\rho\kappa}{\sigma}-\frac{1}{2}\right)-\frac{u^2}{2}(1-\rho^2)-\lambda$.
	The function $F(t,x,y)$ defined by
	$F(t,x,y)=e^{-\lambda t}e^{iu x+y\left(\psi_{\lambda_1,\mu}(t)-iu\frac{\rho}{\sigma}\right)+\theta\kappa\phi_{\lambda_1,\mu}(t)}$
	satisfies $F(0,x,y)=e^{iux+ivy}$ and
	\[
	\frac{\partial F}{\partial t}=\left(\mathcal{L}-\lambda (1+y)\right)F.
	\]
	Moreover, since the real parts of $\lambda_1$ and $\mu$ are nonnegative, we can deduce from the proof of Proposition \ref{Laplace} that the real part of the function $t\rightarrow \psi(t)$ is nonnegative. Then, it is straightforward to see  that, for every $t\geq 0$, we have $F(t,\cdot,\cdot)\in H^2(\O,\m)$ and $t\mapsto F(t,\cdot,\cdot)$ is continuous,
so that,
	for every $v\in V$,  $(\mathcal{L} F(t,.,.), v)_H=-a(F(t,.,.),v)$. Therefore
	\[
	\left(\frac{\partial F}{\partial t}, v\right)_H +a_\lambda(F(t,.,.),v)=0 \quad v\in V,
	\]
	and $F(t,.,.)=\bar{P}^\lambda_tf$.
\end{proof}

		 		\subsection{Estimates on the joint law}\label{sect-estimjointlaw}
		 		In this section we prove Theorem \ref{theotransition}. We first recall some results about the density of the process $Y$.
		 		
		 	With the notations
		 		\[
		 		\nu=\beta-1=\frac{2\kappa\theta}{\sigma^2}-1,\quad y_t=y_0e^{-\kappa t}, \quad L_t=\frac{\sigma^2}{4\kappa}\left(1-e^{-\kappa t}\right),
		 		\]
		 		it is well known (see, for example, \cite[Section 6.2.2]{LL}) that the transition density of the process $Y$ is given by
		 		\[
		 		p_t(y_0,y)=\frac{e^{-\frac{y_t}{2L_t}}}{2y_t^{\nu/2}L_t}e^{-\frac{y}{2L_t}}y^{\nu/2}I_\nu\left(\frac{\sqrt{yy_t}}{L_t}\right),
		 		\]
		 		where $I_\nu$ is the first-order modified Bessel function with index $\nu$, defined by
		 		\[
		 		I_\nu(y)=\left(\frac{y}{2}\right)^\nu\sum_{n=0}^\infty \frac{(y/2)^{2n}}{n!\Gamma(n+\nu+1)}.
		 		\]
		 	It is clear that near $y=0$ we have $I_\nu(y)\sim \frac{1}{\Gamma(\nu+1)}\left(\frac{y}{2}\right)^\nu$ while, for $y\rightarrow \infty$, we have the asymptotic behaviour  $I_\nu(y)\sim e^y/\sqrt{2\pi y}$ (see \cite[page 377]{AS}).
		 	\begin{proposition}\label{density_CIR}
		 	There exists a constant $C_\beta>0$ (which depends only on $\beta$) such that, for every $t>0$,
		 		\[
		 		p_t(y_0,y)\leq
		 		\frac{C_\beta}{L_t^{\beta+\frac{1}{2}}}
		 		e^{-\frac{(\sqrt{y}-\sqrt{y_t})^2}{2L_t}}
		 		y^{\beta -1}
		 		\left(L_t^{1/2}
		 		+
		 		(yy_t)^{1/4}\right), \qquad (y_0,y)\in [0,+\infty)\times]0,+\infty).
		 		\]
		 	\end{proposition}
		 	\begin{proof}
		 From the asymptotic behaviour of  $I_\nu$ near $0$ and $\infty$ we deduce the existence of a constant  $C_\nu>0$ such that
		 		\[
		 		I_\nu(x)\leq C_\nu\left(x^\nu\ind{\{x\leq 1\}}+ \frac{e^x}{\sqrt{x}}\ind{\{x>1\}}\right).
		 		\]
		 	Therefore 
		 		\begin{eqnarray*}
		 			p_t(y_0,y)&=&\frac{e^{-\frac{y_t+y}{2L_t}}}{2y_t^{\nu/2}L_t}y^{\nu/2}I_\nu\left(\frac{\sqrt{yy_t}}{L_t}\right)\\
		 			&\leq &\frac{e^{-\frac{y_t+y}{2L_t}}}{2y_t^{\nu/2}L_t}y^{\nu/2}
		 			C_\nu\left(\frac{(yy_t)^{\nu/2}}{L_t^\nu}
		 			\ind{\{yy_t\leq L_t^2\}}+ \frac{e^{\frac{\sqrt{yy_t}}{L_t}}}{(yy_t)^{1/4}/L_t^{1/2}}
		 			\ind{\{yy_t> L_t^2\}}\right)\\
		 			&=&
		 			\frac{C_\nu}{2}e^{-\frac{y_t+y}{2L_t}}\left(\frac{y^\nu}{L_t^{\nu+1}}\ind{\{yy_t\leq L_t^2\}}
		 			+
		 			\frac{y^{\frac{\nu}{2}-\frac{1}{4}}e^{\frac{\sqrt{yy_t}}{L_t}}}{(y_t)^{\frac{\nu}{2}+\frac{1}{4}}L_t^{1/2}}
		 			\ind{\{yy_t> L_t^2\}}\right).
		 		\end{eqnarray*}
		 		On $\{yy_t> L_t^2\}$, we have $y_t^{-1}\leq y/L_t^2$ and, since $\nu+1>0$,
		 		\begin{eqnarray*}
		 			\frac{y^{\frac{\nu}{2}-\frac{1}{4}}}{(y_t)^{\frac{\nu}{2}+\frac{1}{4}}}=
		 			y_t^{1/4} \frac{y^{\frac{\nu}{2}-\frac{1}{4}}}{(y_t)^{\frac{\nu}{2}+\frac{1}{2}}}
		 			\leq 
		 			y_t^{1/4} \frac{y^{\nu+\frac{1}{4}}}{L_t^{\nu+1}}.
		 		\end{eqnarray*}
	So
		 		\begin{eqnarray*}
		 			p_t(y_0,y)&\leq&
		 			\frac{C_\nu}{2}e^{-\frac{y_t+y}{2L_t}}\left(\frac{y^\nu}{L_t^{\nu+1}}\ind{\{yy_t\leq L_t^2\}}
		 			+
		 			\frac{(yy_t)^{1/4}y^{\nu}
		 				e^{\frac{\sqrt{yy_t}}{L_t}}}{L_t^{\nu+\frac{3}{2}}}\ind{\{yy_t> L_t^2\}}\right)\\
		 			&\leq&
		 			\frac{C_\nu}{2L_t^{\nu+\frac{3}{2}}}
		 			e^{-\frac{y_t+y}{2L_t}}
		 			y^{\nu}e^{\frac{\sqrt{yy_t}}{L_t}}
		 			\left(L_t^{1/2}\ind{\{yy_t\leq L_t^2\}}
		 			+
		 			(yy_t)^{1/4}\ind{\{yy_t> L_t^2\}}\right)  \\
		 			&=&
		 			\frac{C_\nu}{2L_t^{\nu+\frac{3}{2}}}
		 			e^{-\frac{(\sqrt{y}-\sqrt{y_t})^2}{2L_t}}
		 			y^{\nu}
		 			\left(L_t^{1/2}\ind{\{yy_t\leq L_t^2\}}
		 			+
		 			(yy_t)^{1/4}\ind{\{yy_t> L_t^2\}}\right),
		 		\end{eqnarray*}
		 and the assertion follows.
		 	\end{proof}
		 	We are now ready to prove Theorem \ref{theotransition}, which we have used in order to prove the identification of the semigroups in Proposition \ref{identificationsemigroup} and which we will use again later on in this paper. 
\begin{proof}[Proof of Theorem \ref{theotransition}]
	Note that
	\begin{eqnarray*}
		P^\lambda_tf(x_0,y_0)&=&\E_{x_0,y_0}\left(e^{-\lambda\int_0^t(1+Y_s)ds} \tilde{f}(\tilde X_t,Y_t)\right),
	\end{eqnarray*}
where
	\[
	\tilde{f}(x,y)=f\left(x+\frac{\rho}{\sigma}y,y\right) \quad\mbox{and}\quad \tilde X_t=X_t-\frac{\rho}{\sigma} Y_t.
	\]
	Recall that the dynamics of $\tilde X$ is given by \eqref{model2} so we have
	\[
	\tilde X_t=\tilde{x}_0+\bar\kappa\int_0^tY_sds+\bar\rho\int_0^t\sqrt{Y_s}d\tilde{B}_s,
	\]
with
	\[
	\tilde{x}_0=x_0-\frac{\rho}{\sigma}y_0,\quad \bar\kappa=\frac{\rho\kappa}{\sigma}-\frac{1}{2},\quad
	\bar\rho=\sqrt{1-\rho^2}.
	\]
Recall that the Brownian motion $\tilde{B}$ is independent of the process $Y$. We set
$
	\Sigma_t=\sqrt{\int_0^t Y_sds}
	$
  and $n(x)= \frac{1}{\sqrt{2\pi}}e^{-x^2/2}$. Therefore
	\begin{eqnarray*}
		P^\lambda_t f(x_0,y_0)&=&\E_{y_0}\left( e^{-\lambda t-\lambda\Sigma_t^2}
		\int \tilde{f}\left(\tilde{x}_0+\bar\kappa\Sigma_t^2+\bar\rho\Sigma_t z,Y_t\right)n(z)dz\right)\\
		&\leq&\E_{y_0}\left( e^{-\lambda\Sigma_t^2}
		\int \tilde{f}\left(\tilde{x}_0+\bar\kappa\Sigma_t^2+\bar\rho\Sigma_t z,Y_t\right)n(z)dz\right)\\
		&=&\E_{y_0}\left( e^{-\lambda\Sigma_t^2}
		\int \tilde{f}\left(
		\tilde{x}_0+z,Y_t\right)n\left(\frac{z-\bar\kappa\Sigma_t^2}{\bar\rho\Sigma_t} \right)
		\frac{dz}{\bar\rho\Sigma_t}\right).
	\end{eqnarray*}
H\"{o}lder's inequality with respect to the measure $e^{-\gamma|z|-\bar{\mu}Y_t}dzd\P_{y_0}$,
	where $\gamma >0$ and $\bar \mu>0$ will be chosen later on, gives, for every $p>1$
	\begin{eqnarray}
	P^\lambda_t f(x_0,y_0)&\leq &
	\left[\E_{y_0}\left(\int e^{-\gamma|z|-\bar\mu Y_t} \tilde{f}^p\left(
	\tilde{x}_0+z,Y_t\right)dz\right)\right]^{1/p}J_q,\label{eq-P1}
	\end{eqnarray}
	with $q=p/(p-1)$ and
	\[
	(J_q)^q=\E_{y_0}\left(\int e^{(q-1)\gamma|z|+(q-1)\bar\mu Y_t-q\lambda \Sigma_t^2}
	n^q\left(\frac{z-\bar\kappa\Sigma_t^2}{\bar\rho\Sigma_t} \right)
	\frac{dz}{(\bar\rho\Sigma_t)^q}\right).
	\]
	Using Proposition \ref{density_CIR} we can write, for every $z\in\R$,
	\begin{eqnarray*}
		\E_{y_0} \left(e^{-\bar\mu Y_t} \tilde{f}^p\left(
		\tilde{x}_0+z,Y_t\right)\right)&=&\int_0^\infty dy p_t(y_0,y) e^{-\bar\mu y}
		\tilde{f}^p\left(
		\tilde{x}_0+z,y \right)\\
		&\leq& \frac{C_\beta\left(\sqrt{\frac{\sigma^2}{4\kappa}}+y_0^{1/4}\right)}{L_t^{\beta+\frac{1}{2}}}\int_0^\infty \!\!\!dy
		e^{-\frac{(\sqrt{y}-\sqrt{y_t})^2}{2L_t}-\bar\mu y}
		y^{\beta -1}
		\left(1
		+
		y^{1/4}\right) \tilde{f}^p\left(
		\tilde{x}_0+z,y \right).
	\end{eqnarray*}
If we set $L_\infty=\sigma^2/(4\kappa)$,  for every $\epsilon\in(0,1)$  we have
	\begin{align*}
		e^{-\frac{(\sqrt{y}-\sqrt{y_t})^2}{2L_t}}&\leq e^{-\frac{(\sqrt{y}-\sqrt{y_t})^2}{2L_\infty}}\\
		&= e^{-\frac{y}{2L_\infty}}e^{\frac{\sqrt{yy_t}}{L_\infty}-\frac{y_t}{2L_\infty}}\\
		&\leq e^{-\frac{y}{2L_\infty}}e^{\epsilon\frac{y}{2L_\infty}}e^{\frac{y_t}{2\epsilon L_\infty}}
		e^{-\frac{y_t}{2L_\infty}}\\
		&=e^{-(1-\epsilon)\frac{y}{2L_\infty}}e^{\frac{y_t}{2\epsilon L_\infty}(1-\epsilon)}\\
		&\leq e^{-(1-\epsilon)\frac{y}{2L_\infty}}e^{\frac{y_0}{2\epsilon L_\infty}(1-\epsilon)}.
	\end{align*}
It is easy to see that $e^{-y\left( \bar \mu +\frac{1-\epsilon}{2L_\infty}   \right)}(1+y^{1/4})\leq C_{\epsilon,\sigma,\kappa}e^{-y\left( \bar \mu +\frac{1-2\epsilon}{2L_\infty}   \right)}$. Therefore,  we can write
	\begin{eqnarray*}
		\E_{y_0} \left(e^{-\bar\mu Y_t} \tilde{f}^p\left(
		\tilde{x}_0+z,Y_t\right)\right)&\leq&
		\frac{C_\beta e^{\frac{y_0(1-\epsilon)}{2\epsilon L_\infty}}
		\left(\sqrt{\frac{\sigma^2}{4\kappa}}+y_0^{1/4}\right)}{L_t^{\beta+\frac{1}{2}}}\int_0^\infty \!\!\!dy
		e^{-y\left(\bar\mu + \frac{1-\epsilon}{2L_\infty}\right)}
		y^{\beta -1}
		\left(1
		+
		y^{1/4}\right) \tilde{f}^p\left(
		\tilde{x}_0+z,y \right)\\
		&\leq&
		\frac{C_{\beta, \sigma,\kappa, \epsilon} e^{\frac{y_0(1-\epsilon)}{\epsilon L_\infty}}
		}{L_t^{\beta+\frac{1}{2}}}\int_0^\infty \!\!\!dy
		e^{-y\left(\bar\mu + \frac{1-2\epsilon}{2L_\infty}\right)}
		y^{\beta -1}
		\tilde{f}^p\left(
		\tilde{x}_0+z,y \right).
	\end{eqnarray*}
	As regards $J_q$, setting $z'=\frac{z-\bar\kappa\Sigma_t^2}{\bar\rho\Sigma_t}$, we have
	\begin{eqnarray*}
		(J_q)^q&=&\E_{y_0}\left(\int e^{(q-1)\gamma|z'\bar\rho\Sigma_t +\bar\kappa\Sigma_t^2|+(q-1)\bar\mu Y_t-q\lambda \Sigma_t^2}
		n^q\left(z' \right)
		\frac{dz'}{(\bar\rho\Sigma_t)^{q-1}}\right)\\
		&\leq &
		\E_{y_0}\left(
		\int e^{(q-1)\gamma\bar\rho\Sigma_t |z|
			+(q-1)\bar\mu Y_t+ ((q-1)| \bar\kappa|\gamma-q\lambda )\Sigma_t^2}
		n^q\left(z \right)
		\frac{dz}{(\bar\rho\Sigma_t)^{q-1}}\right).
	\end{eqnarray*}
Note that
	\begin{eqnarray*}
		\int e^{(q-1)\gamma\bar\rho\Sigma_t |z|
		}
		n^q\left(z \right)
		dz
		&=&\frac{1}{(\sqrt{2\pi})^{q}}\int e^{(q-1)\gamma\bar\rho\Sigma_t |z|
		}
		e^{-qz^2/2}
		dz\\
			&\leq &\frac{2}{\sqrt{2\pi}} \int e^{(q-1)\gamma\bar\rho\Sigma_t z
		}
		e^{-qz^2/2}
		dz\\
		&=&\frac{2}{\sqrt{2\pi}}e^{\frac{(q-1)^2}{2q}\gamma^2\bar\rho^2 \Sigma_t^2}
		\int e^{-\frac 1 2\left( \sqrt{q }z-\frac{(q-1)\gamma\bar\rho\Sigma_t }{\sqrt{ q}}\right)^2}	
	dz\\
		&=&\frac 2 {\sqrt q}e^{\frac{(q-1)^2}{2q}\gamma^2\bar\rho^2 \Sigma_t^2},
	\end{eqnarray*}
	so that
	\begin{eqnarray*}
		(J_q)^q&\leq&\frac 2 {\sqrt q}\E_{y_0}\left(
		e^{
			(q-1)\bar\mu Y_t+ \bar\lambda_q\Sigma_t^2}
		\frac{1}{(\bar\rho\Sigma_t)^{q-1}}
		\right),
	\end{eqnarray*}
	with
	\[
	\bar\lambda_q= (q-1)| \bar\kappa|\gamma +\frac{(q-1)^2}{2q}\gamma^2\bar\rho^2 -q\lambda
	=\frac{1}{p-1}\left(| \bar\kappa|\gamma+\frac{1}{2p}\gamma^2\bar\rho^2-p\lambda\right).
	\]
	Using  H\"older's  inequality again
 we get, for every $p_1>1$ and $q_1=p_1/(p_1-1)$,
	\begin{eqnarray*}
		(J_q)^q&\leq&\sqrt{\frac 2 q}\left(
		\E_{y_0}\left(
		e^{
			p_1(q-1)\bar\mu Y_t+ p_1\bar\lambda_q\Sigma_t^2}\right)\right)^{1/p_1}
		\left(
		\E_{y_0}\left(
		\frac{1}{(\bar\rho\Sigma_t)^{q_1(q-1)}}\right)\right)^{1/q_1}\\
		&\leq&
		\frac{C_{q,q_1}}{t^{q-1}} \left(\E_{y_0}\left(
		e^{
			p_1(q-1)\bar\mu Y_t+ p_1\bar\lambda_q\Sigma_t^2}\right)\right)^{1/p_1},
	\end{eqnarray*}
where the last inequality follows from Lemma \ref{RemarqueMomentsNeg}.
	
	We now apply Proposition~\ref{Laplace2} with $\lambda_1=p_1(q-1)\bar\mu$
	and $\lambda_2=p_1\bar\lambda_q$. The assumption on $\lambda_1$ and $\lambda_2$ becomes
	\[ \frac{\sigma^2}{2}p_1(q-1)\bar\mu^2-\kappa \bar\mu+| \bar\kappa|\gamma+\frac{1}{2p}\gamma^2\bar\rho^2-p\lambda\leq 0
	\]
or, equivalently, 
	\[
	\lambda\geq \frac{\sigma^2}{2p(p-1)}p_1\bar\mu^2-\kappa \frac{\bar\mu}{p}+| \bar\kappa|\frac{\gamma}{p}+\frac{1}{2p^2}\gamma^2\bar\rho^2.
	\]
Note that the last inequality is satisfied for at least a  $p_1>1$ if and only if
\begin{equation}\label{assumption_lambda}
	\lambda> \frac{\sigma^2}{2p(p-1)}\bar\mu^2-\kappa \frac{\bar\mu}{p}+| \bar\kappa|\frac{\gamma}{p}+\frac{1}{2p^2}\gamma^2\bar\rho^2.
\end{equation}
	Going back to \eqref{eq-P1} under the condition \eqref{assumption_lambda}, we have
\begin{eqnarray*}
	P^\lambda_t f(x_0,y_0)&\leq &
	\frac{C_{p,\epsilon}}{L_t^{\frac{\beta}{p}+\frac{1}{2p}}t^{1/p}}e^{A_{p, \epsilon}y_0}
	\left(\int dz e^{-\gamma|z|}\int_0^\infty \!\!\!dy
	e^{-y\left(\bar\mu + \frac{1-2\epsilon}{2L_\infty}\right)}
	y^{\beta -1}
	\tilde{f}^p\left(
	\tilde{x}_0+z,y \right)\right)^{1/p}\\
	&\leq&
	\frac{C_{p,\epsilon}e^{A_{p, \epsilon}y_0}}{t^{\frac{\beta}{p}+\frac{3}{2p}}}
	\left(\int dz e^{-\gamma|z|}\int_0^\infty \!\!\!dy
	e^{-y\left(\bar\mu + \frac{1-2\epsilon}{2L_\infty}\right)}
	y^{\beta -1}
	f^p\left(
	\tilde{x}_0+z+\frac{\rho}{\sigma}y,y \right)\right)^{1/p}\\
	&=&
	\frac{C_{p,\epsilon}e^{A_{p, \epsilon}y_0}}{t^{\frac{\beta}{p}+\frac{3}{2p}}}
	\left(\int dz e^{-\gamma|z-\tilde{x}_0-\frac{\rho}{\sigma}y|}\int_0^\infty \!\!\!dy
	e^{-y\left(\bar\mu + \frac{1-2\epsilon}{2L_\infty}\right)}
	y^{\beta -1}
	f^p\left(
	z,y \right)\right)^{1/p}\\
	&\leq&
	\frac{C_{p,\epsilon}e^{A_{p, \epsilon}y_0+\gamma|\tilde{x}_0|}}{t^{\frac{\beta}{p}+\frac{3}{2p}}}
	\left(\int dz e^{-\gamma|z|}\int_0^\infty \!\!\!dy
	e^{-y\left(\bar\mu -\gamma\frac{|\rho|}{\sigma}+ \frac{1-2\epsilon}{2L_\infty}\right)}
	y^{\beta -1}
	f^p\left(
	z,y \right)\right)^{1/p}.
\end{eqnarray*}
		If we choose $\epsilon=1/2$ and $\bar\mu=\mu+\gamma\frac{|\rho|}{\sigma}$,  the assertion follows provided $\lambda$ satisfies
	\[
	\lambda>\frac{\sigma^2}{2p(p-1)}\left(\mu+\gamma\frac{|\rho|}{\sigma}\right)^2-
	\kappa \frac{\mu+
		\gamma\frac{|\rho|}{\sigma}}{p}+| \bar\kappa|\frac{\gamma}{p}+\frac{1}{p^2}\gamma^2\bar\rho^2.
	\]
	\end{proof}

\subsection{Proof of Theorem  \ref{theorem2}}
We are finally ready to prove the identification Theorem \ref{theorem2}. We first prove the result under further regularity assumptions on the payoff function $\psi$, then we deduce the general statement by an approximation technique. 

\subsubsection{Case with a regular function $\psi$ }
The following regularity result paves the way for the identification theorem in the case of a  regular payoff function.
		 		\begin{proposition}\label{reg2}
		 		Assume that $\psi$ satisfies Assumption $\mathcal H^1$ and $0\leq \psi\leq \Phi$ with $\Phi$ satisfying Assumption $\mathcal{H}^2$. If moreover we assume  $\psi \in L^2([0,T];H^2(\O,\m))$ and $ \frac{\partial \psi}{\partial t}+\L \psi , \,(1+y)\Phi\in L^p([0,T];L^p(\mathcal{O,\m}))$ for some $p\geq 2$, then there exist $\lambda_0>0$ and $F\in L^p([0,T];L^p(\mathcal{O,\m}))$ such that for all $ \lambda \geq \lambda_0$ the solution $u$ of \eqref{VI} satisfies
		 			\begin{equation}
		 			-\left(\frac{\partial u}{\partial t}	 ,v\right)_H + a_\lambda(u,v)=(F,v)_H,\qquad \mbox{a.e. in } [0,T], \quad v\in V.
		 			\end{equation}
		 		\end{proposition}
		 		\begin{proof}
		 			Note that, for $\lambda$ large enough, $u$ can be seen as the solution $u_\lambda$ of an equivalent coercive variational inequality, that is 
		 			\begin{equation*}
		 			-\left(\frac{\partial u_{\lambda}}{\partial t}	 ,v -u_\lambda  \right)_H + a_\lambda(u_\lambda,v-u_\lambda)\geq (g,v-u_\lambda)_H,
		 			\end{equation*}
		 			where $g= \lambda(1+y)u $ satisfies the assumptions of Proposition \ref{coercive variational inequality}. Therefore,  there exists a sequence $(u_{\varepsilon, \lambda})_\varepsilon$ of non  negative functions such that $\lim_{\varepsilon\rightarrow 0 }u_{\varepsilon, \lambda}=u_\lambda$ and
		 			$$
		 			-\left( \frac{\partial u_{\varepsilon,\lambda}}{\partial t },v   \right)_H + a_\lambda(u_{\varepsilon,\lambda},v)-\left(\frac 1 \varepsilon (\psi-u_{\varepsilon,\lambda})_+,v\right)_H= (g,v)_H,\qquad v\in V.
		 			$$
Since  both $u_{\varepsilon,\lambda}$ and $\psi$ are positive and $\psi$ belongs to $L^p([0,T]; L^p(\mathcal{O},\m))$, we have $(\psi-u_{\varepsilon,\lambda})_+\in L^p([0,T]; L^p(\mathcal{O},\m))$. In order to simplify the notation, we set $w=(\psi-u_{\varepsilon,\lambda})_+$. Taking $v=w^{p-1}$ and assuming that $\psi$ is bounded we observe that $v\in L^2([0,T]; V  )$ and we can write
		 			$$
		 			-\left(	\frac{\partial u_{\varepsilon,\lambda}}{\partial t}, w^{p-1}\right)_H
		 			+a_\lambda(u_{\varepsilon,\lambda},w^{p-1})-\frac 1 \varepsilon \Vert w\Vert^p_{L^p(\mathcal{O},\m)} =\left( g, w^{p-1} \right)_H,
		 			$$
		 			so that
		 			$$
		 			\frac 1 p \frac{d }{dt }\Vert w\Vert_{ L^p(\mathcal{O},\m)}^p  -a_\lambda(\psi-u_{\varepsilon,\lambda},w^{p-1})-\frac 1 \varepsilon \Vert w\Vert^p_{L^p(\mathcal{O},\m)} 
		 					=\left( g, w^{p-1} \right)_H-\left( \frac{\partial \psi }{\partial t },w^{p-1} \right)_H+ a_\lambda(\psi,w^{p-1}).
		 			$$
		 			Integrating from $0$ to $T$ we get
		 			\begin{equation}\label{regpaux}
		 			\begin{split}
		 		&	-	\frac 1 p\Vert w(0)\Vert_{ L^p(\mathcal{O},\m)}^p -\int_0^T a_\lambda((\psi-u_{\varepsilon,\lambda})(t),w^{p-1}(t))dt  -\frac 1 \varepsilon\int_0^T \Vert w(t)\Vert_{ L^p(\mathcal{O},\m)}^p dt \\&\quad
		 			=\int_0^T\left( g(t), w^{p-1}(t) \right)_Hdt-\int_0^T \left( \frac{\partial \psi }{\partial t }(t),w_+^{p-1}(t) \right)_Hdt+ \int_0^Ta_\lambda(\psi(t),w^{p-1}(t))dt.
		 			\end{split}
		 			\end{equation}

		 			Now, with the usual integration by parts,
		 			\begin{align*}
		 			&a_{\lambda}(w,w^{p-1})
		 		= \into \frac y 2 (p-1 ) w^{p-2 }\left[ \left( \frac{\partial w}{\partial x }\right)^2 +2\rho\sigma      \frac{\partial w}{\partial x }\frac{\partial w}{\partial y }+ \sigma^2 \left( \frac{\partial w}{\partial y }\right)^2 \right]d\m
		 			\\&		 \qquad		+\into y\left( j_{\gamma,\mu }(x)  \frac{\partial w}{\partial x }+ k_{\gamma,\mu}(x)   \frac{\partial w}{\partial y }  \right)  w^{p-1} d\m + \lambda \into  ( 1+y ) w^p d\m 
		 			\\& 	\geq \delta_1(p-1) \into y w^{p-2 } \left[ \left( \frac{\partial w}{\partial x }\right)^2 + \left( \frac{\partial w}{\partial y }\right)^2 \right]d\m
		 			+\into y\left( j_{\gamma,\mu }(x)  \frac{\partial w}{\partial x }+ k_{\gamma,\mu}(x)   \frac{\partial w}{\partial y }  \right)  w^{p-1} d\m + \lambda \into  y w^p d\m 
		 			\\& 		= \into y  w^{p-2 }   \bigg[    \delta_1(p-1)   \left( \frac{\partial w}{\partial x }\right)^2  
		 		+  j_{\gamma,\mu }(x)  \frac{\partial w}{\partial x } w + \frac{\lambda}{2 }  w^2 \bigg] d\m\\&\qquad
+ \into y  w^{p-2 }   \bigg[    \delta_1(p-1)   \left( \frac{\partial w}{\partial y }\right)^2 
		 		+  k_{\gamma,\mu }(x)  \frac{\partial w}{\partial y } w + \frac{\lambda}{2 }  w^2 \bigg] d\m
		 		\geq 0,
		 			\end{align*}
		 			since, for $\lambda$ large enough, the quadratic forms $(a,b)\rightarrow  \delta_1(p-1)a^2 + j_{\gamma.\mu}  ab + \frac \lambda 2 b^2 $ and $(a,b)\rightarrow  \delta_1(p-1)a^2 + k_{\gamma.\mu}  ab + \frac \lambda 2 b^2 $ are both positive definite.
		 			
		 			Recall that $\psi\in L^2([0,T];H^2(\O,\m)), $  $\frac{\partial \psi }{\partial_t}+\mathcal{L }\psi\in L^p([0,T],L^p(\O,\m)), \, (1+y)\psi\leq(1+y)\Phi\in L^p([0,T],L^p(\O,\m))$  and $g=(1+y)u\leq(1+y)\Phi\in L^p([0,T];L^p(\mathcal{O,\m}))$.
		 			Therefore, going back to \eqref{regpaux} and using Holder's inequality,
		 			\begin{align*}
		 			\frac 1 \varepsilon\int_0^T\!\! \Vert w(t)\Vert_{ L^p(\mathcal{O},\m)}^pdt \leq \left[\left(\int_0^T \!\!\Vert g(t)\Vert_{ L^p(\mathcal{O},\m)}^pdt  \right)^{\frac 1 p}\!+ \left(\int_0^T\!\left \Vert   \frac{\partial \psi }{\partial t} (t)+ \mathcal{L}^\lambda\psi (t) \right \Vert_{ L^p(\mathcal{O},\m)}^p\!\!\!\!dt  \right)^{\frac 1 p} 
		 			\right]  \left( \int_0^T\!\! \Vert w\Vert_{ L^p(\mathcal{O},\m)}^pdt  \right)^{\frac {p-1} p}\!\!\!.
		 			\end{align*}
		 		Recalling that $w=(\psi-u_{\varepsilon,\lambda})_+$, we deduce that
		 			\begin{equation}
		 				\left \Vert \frac 1 \varepsilon (\psi-u_{\varepsilon,\lambda})_+ \right \Vert_{L^p([0,T]; L^p(\mathcal{O},\m))} \leq C,
		 			\end{equation}
		 			for a positive constant $C$ independent of $\varepsilon$. 	Note that the estimate does not involve the $L^\infty$-norm of $\psi$ (which we assumed to be bounded for the payoff) so that by a standard approximation argument, it remains valid for unbounded $\psi$. 
		 			The assertion then follows passing to the limit for $\varepsilon\rightarrow0$ in 
		 			$$
		 			-\left( \frac{\partial u_{\varepsilon,\lambda}}{\partial t },v   \right)_H + a_\lambda(u_{\varepsilon,\lambda},v)= \left(\frac 1 \varepsilon (\psi-u_{\varepsilon,\lambda})_+,v\right)_H +(g,v)_H,\qquad  v\in V.
		 			$$
		 		\end{proof}
		 		Now, note that we can easily prove the continuous dependence of the process $X$ with respect to the initial state. 
		 		\begin{lemma}\label{propflo}
		 			Fix $(x,y)\in \R\times [0,+\infty)$. Denote by $(X^{x,y}_t, Y^y_t)_{t\geq 0}$ the solution of the system
		 			\[
		 			\left\{
		 			\begin{array}{l}
		 			dX_t=\left( \frac{\rho\kappa\theta}{\sigma}-\frac{Y_t}{2}\right)dt+\sqrt{Y_t} dB_t,\\
		 			dY_t=\kappa(\theta-Y_t)dt+\sigma\sqrt{Y_t}dW_t,
		 			\end{array}
		 			\right.
		 			\]
		 			with $X_0=x$, $Y_0=y$ and   $\langle B,W\rangle_t=\rho t$.
		 			We have, for every  $ t\geq 0$ and for every $(x,y), \, (x',y')\in \R\times [0,+\infty)$,
		 			$\E\left|Y^{y'}_t-Y^y_t\right| \leq |y'-y|$ and
		 			\[
		 			\E\left|X^{x',y'}_t-X^{x,y}_t\right| \leq |x'-x|+\frac{t}{2}|y'-y|+\sqrt{t|y'-y|}.
		 			\]    
		 		\end{lemma}
		 		The proof of Lemma \ref{propflo} is straightforward so we omit the details: the inequality $\E\left|Y^{y'}_t-Y^y_t\right| \leq  |y'-y|$ can be proved by using standard techniques introduced in \cite{IW} (see the proof of Theorem 3.2 and its Corollary in Section IV.3) and the other inequality easily follows.

		 		Then, we can prove the following result.
		 			\begin{proposition}\label{propconti1}
		 			 Let $\psi:\R\times[0,\infty)\rightarrow\R$ be  continuous 
		 				and such that there exist $C>0$ and $ a,\, b \geq0$ with  $|\psi(x,y)|\leq Ce^{a|x|+by}$ for every $(x,y)\in \R\times [0,+\infty)$. Then, if
		 				\[
		 				\lambda> ab|\rho|\sigma+\frac{b^2\sigma^2}{2}-\kappa b+\frac{a^2-a}{2},
		 				\]
		 				we have $P^\lambda_t|\psi|(x,y)<\infty$ for every $t\geq 0$, $(x,y)\in \R\times [0,+\infty)$ and the function
		 				$(t,x,y)\mapsto P^\lambda_t\psi(x,y)$ is continuous on $[0,\infty)\times\R\times[0,\infty)$.
		 		\end{proposition}
		 		\begin{proof}
		 	 We can prove, as in the proof of  Proposition \ref{prop-TF},  that
		 	 \begin{eqnarray*}
		 	 	\E_{x,y}\left(e^{a X_t+bY_t-\lambda\int_0^t Y_s ds}\right)&=&
		 	 	e^{a\left(x-\frac{\rho}{\sigma}y\right)}
		 	 	\E_{y}\left(e^{\left(a\frac{\rho}{\sigma}+b\right)Y_t+
		 	 		\left(a(\frac{\rho\kappa}{\sigma}-\frac{1}{2})+\frac{a^2}{2}(1-\rho^2)-\lambda\right)\int_0^tY_sds}\right).
		 	 \end{eqnarray*}
		 	 Thanks to Proposition~\ref{Laplace2}, if
		 	 \begin{equation}\label{6*}
		 	 \frac{\sigma^2}{2}\left(a\frac{\rho}{\sigma}+b\right)^2-\kappa\left(a\frac{\rho}{\sigma}+b\right)+
		 	 \left(a(\frac{\rho\kappa}{\sigma}-\frac{1}{2})+\frac{a^2}{2}(1-\rho^2)-\lambda\right)<0,
		 	 \end{equation}
		 	 we have, for any $T>0$ and for any compact $K\subseteq \R\times [0,+\infty[$,
		 	 \[
		 	 \sup_{(t,x,y)\in[0,T]\times K} \E_{x,y}\left(e^{a X_t+bY_t-\lambda\int_0^t Y_s ds}\right)<\infty.
		 	 \]
		 	 Note that \eqref{6*} is equivalent to
		 	 \[
		 	 \lambda>ab\rho\sigma+\frac{b^2\sigma^2}{2}-\kappa b+\frac{a^2-a}{2}.
		 	 \]
		 	 Therefore, under the assumptions of the Proposition, we have,
		 	 for any $T>0$ and for any compact set $K\subseteq \R\times [0,+\infty[$,
		 	 \[
		 	 \sup_{(t,x,y)\in[0,T]\times K} \E_{x,y}\left(e^{a |X_t|+bY_t-\lambda\int_0^t Y_s ds}\right)<\infty.
		 	 \]
		 	Moreover, for $\epsilon$ small enough,
		 	 \begin{equation}\label{6**}
		 	 \sup_{(t,x,y)\in[0,T]\times K} \E_{x,y}\left(e^{a(1+\epsilon) |X_t|+b(1+\epsilon)Y_t-\lambda(1+\epsilon)\int_0^t Y_s ds}\right)<\infty.
		 	 \end{equation}
		 	Then, let  $\psi$ be a continuous function on $\R\times[0,+\infty[$ such that $|\psi(x,y)|\leq Ce^{a|x|+by}$. It is evident that $P^\lambda_t |\psi|(x,y)<\infty$
		 	 and we have
		 	 \[
		 	 P^\lambda_t\psi(x,y)=\E\left(e^{-\lambda\int_0^t(1+Y^y_s)ds}\psi(X^{x,y}_t,Y^{y}_t)\right).
		 	 \]
		 	 	If $((t_n,x_n,y_n))_n$ converges to $(t,x,y)$,  we deduce from Lemma~\ref{propflo} that $X^{x_n,y_n}_{t_n}\rightarrow X^{x,y}_t$, $Y^{y_n}_{t_n}\rightarrow Y^{y}_t$ and  $\int_0^{t_n} Y^{y_n}_sds \rightarrow\int_0^t Y^y_sds$ in probability. 
		 	 Therefore $e^{-\lambda\int_0^{t_n}(1+Y_s)ds}\psi(X^{x_n,y_n}_{t_n},Y^{y_n}_{t_n})$
		 	converges to $e^{-\lambda\int_0^{t}(1+Y_s)ds}\psi(X^{x,y}_{t},Y^{y}_{t})$ in probability. The estimate~\eqref{6**}
		 	 ensures the uniformly integrability of $e^{-\lambda\int_0^{t_n}(1+Y_s)ds}\psi(X^{x_n,y_n}_{t_n},Y^{y_n}_{t_n})$
		 	so that  $\lim_{n\to \infty}P^\lambda_{t_n}\psi(x_n,y_n)=P^\lambda_{t}\psi(x,y)$ which concludes the proof.		
		 		\end{proof} 

		 		\begin{proposition}\label{martingale}
		 			Fix  $p>\beta+\frac{5}{2}$ and $\lambda$ as in Theorem \ref{theotransition}. Let us consider
		 			$u\in C([0,T];H)\cap L^2([0,T];V)$, with $\frac{\partial u}{\partial t}\in L^2([0,T];H)$ such that
		 			\[
		 		\begin{cases}
		 		\left	(\frac{\partial u}{\partial t},v\right)_H+a_\lambda(u(t),v)=(f(t),v)_H,\qquad v\in V,\\
		 			u(0)=\psi,
		 		\end{cases}
		 			\]
		 			with $\psi$ continuous, $\psi\in V$, $\sqrt{1+y}f\in L^2([0,T];H)$  and $f\in L^p([0,T];L^p(\mathcal{O}, \m))$. Then, if $\psi$
		 			and $\lambda$ satisfy the assumptions of Proposition ~\ref{propconti1},  we have 
		 			\begin{enumerate}
		 				\item For every $t\in [0,T]$, $u(t)=P^\lambda_t\psi+\int_0^tP^\lambda_sf(t-s)ds$.
		 				\item  The function  $(t,x,y)\mapsto u(t,x,y)$ is continuous on $[0,T]\times \R\times [0,+\infty)$.
		 				\item If $\Lambda_t=\lambda\int_0^t(1+Y_s)ds$, 
		 				the process $(M_t)_{0\leq t\leq T}$, defined by
		 				\[
		 				M_t=e^{-\Lambda_t}u(T-t,X_t,Y_t)+\int_0^te^{-\Lambda_s}f(T-s,X_s,Y_s)ds,
		 				\]
		 			with $X_0=x,\, Y_0=y$	is a martingale
		 				for every $(x,y)\in \R\times [0,+\infty)$.
		 			\end{enumerate}
		 		\end{proposition}
		 		\begin{proof}
		 			The first assertion follows from Proposition~\ref{propsg3}.

		 			The continuity of $(t,x,y)\mapsto P^\lambda_t\psi(x,y)$ is given by Proposition~\ref{propconti1}. The continuity of  $(t,x,y)\mapsto \int_0^t P^\lambda_sf(t-s,.)(x,y)ds$  is trivial if $(t,x,y)\mapsto f(t,x,y)$
		 		is bounded continuous.  If $f\in L^p([0,T];L^p(\mathcal{O}, \m))$, $f$ is the limit in  $L^p$ of a sequence of  bounded continuous functions and we have $\int_0^t P^\lambda_s f_n(t-s,\cdot)ds \rightarrow \int_0^t P^\lambda_s f(t-s,\cdot)ds$ uniformly in $[0,T]\times K$ for every compact
		 			$K$ of $\R\times [0,+\infty)$). In fact, thanks to Theorem~\ref{theotransition}, we can write for 
		 			$t\in[0,T]$ and $(x,y)\in K$ 
		 			\begin{equation}
		 			\begin{split}
		 				\int_0^tP^\lambda_s|f_n-f|(t-s,\cdot,\cdot)(x,y)ds&\leq \int_0^t\frac{C_{p,K,T}}{s^{\frac{2\beta+3}{2p}}}ds
		 				||(f_n-f)(t-s,\cdot,\cdot)||_{L^p(\O,\m)}\\
		 				&\leq C_{p,K,T}\left(\int_0^t||(f_n-f)(t-s,\cdot,\cdot)||^p_{L^p(\O,\m)}ds\right)^{1/p}\left(\int_0^t\frac{ds}{s^{\frac{2\beta+3}{2(p-1)}}}\right)^{1-\frac{1}{p}}\\
		 				&\leq C_{p,K,T}\left(\int_0^T||(f_n-f)(s,\cdot,\cdot)||^p_{L^p(\O,\m)}ds\right)^{1/p}\left(\int_0^T\frac{ds}{s^{\frac{2\beta+3}{2(p-1)}}}\right)^{1-\frac{1}{p}}.
		 				\end{split}
		 			\end{equation}
		 			The assumption $p>\beta+\frac{5}{2}$ ensures the convergence of the integral in the right hand side.

	 			For the last assertion, note that $M_T=e^{-\Lambda_T}\psi(X_T,Y_T)+\int_0^Te^{-\Lambda_s}f(T-s,X_s,Y_s)ds$.
	 			Then, we can prove that $M_t$ is integrable with the same arguments that we used to show the continuity of $(t,x,y)\mapsto u(t,x,y)$.
		 					Moreover, by using the Markov property,
		 			\begin{eqnarray*}
		 				\E_{x,y}\left( M_T \,|\, \mathcal{F}_t\right)&=&e^{-\Lambda_t}P^{\lambda}_{T-t}\psi(X_t,Y_t)
		 				+\int_0^te^{-\Lambda_s}f(T-s,X_s,Y_s)ds
		 				+e^{-\Lambda_t}\int_t^TP^\lambda_{s-t}f(T-s,.,.)(X_t,Y_t)ds\\
		 				&=&e^{-\Lambda_t}\left(P^{\lambda}_{T-t}\psi(X_t,Y_t)+\int_0^{T-t}P^\lambda_{s}f(T-t-s,.,.)(X_t,Y_t)ds\right)+\int_0^te^{-\Lambda_s}f(T-s,X_s,Y_s)ds\\
		 				&=&e^{-\Lambda_t}u(T-t, X_t, Y_t)+\int_0^te^{-\Lambda_s}f(T-s,X_s,Y_s)ds=M_t.
		 			\end{eqnarray*}
		 			
		 		\end{proof}
		 		
		 We are now ready to prove the following proposition.
		 		
		 		\begin{proposition}\label{identification}
		 		Assume that $\psi$ satisfies Assumption $\mathcal{H}^*$.
		 		Moreover, fix $p>\beta+\frac{5}{2}$ and assume that $\psi\in L^2([0,T];H^2(\O,\m))$ and $\frac{\partial \psi}{\partial t}+\L\psi\in L^p([0,T];L^p(\O,\m))$.
		 		  Then, the solution u of the variational inequality \eqref{VI} satisfies
	\begin{equation} \label{u=u*}
		 			u(t,x,y)=u^*(t,x,y), \qquad \mbox{on }[0,T] \times \bar{\mathcal{O}},
		 			\end{equation}
		 			where $u^*$ is defined by
		 			$$
		 			u^*(t,x,y)= \sup_{\tau \in \mathcal{T}_{t,T}}\E \left[   \psi(\tau,X_\tau^{t,x,y},
		 			Y_\tau^{t,y})     \right].
		 			$$
		 		\end{proposition}
		 		\begin{proof}
		 			 			We first check that $\psi$ satisfies the assumptions of Proposition \ref{reg2}. Note that, thanks to the growth condition \eqref{boundonpsi}, it is possible to write  $0\leq \psi(t,x,y)\leq \Phi(t,x,y)$ with $\Phi(t,x,y)=C_T(e^{x-\frac{\rho \kappa\theta}{\sigma}t}+e^{Ly-\kappa\theta L t})$, where $L\in\left[0,\frac{2\kappa}{\sigma^2}\right)$ and $C_T$ is a positive constant which depends on $T$. Moreover,  recall the growth condition on the derivatives  \eqref{boundonderivatives}. Then, it is easy to see that we can choose $\gamma $ and $\mu$ in the definition of the measure $\m$ (see \eqref{m}) such that $\psi$ satisfies Assumption $\mathcal{H}^1$, $ \Phi$ satisfies Assumption $\mathcal{H}^2$ (note that  $\frac{\partial \Phi}{\partial t}+\L\Phi\leq 0$) and
		 			 				 $(1+y)\Phi$, $\frac{\partial \psi}{\partial t}+\L\psi\in L^p([0,T];L^p(\O,\m))$.  Therefore we can apply  Proposition \ref{reg2} and we get  that, for $\lambda$ large enough,  there exists $F\in L^p([0,T];L^p(\mathcal{O,\m}))$ such that  $u$ satisfies
		 			 			\begin{equation*}
		 			 			-\left(\frac{\partial u}{\partial t}	 ,v\right)_H + a_\lambda(u,v)=(F,v)_H,\qquad v\in V, 
		 			 			\end{equation*}
		 			 			that is
		 			 						\begin{equation*}
		 			 						-\left(\frac{\partial u}{\partial t}	 ,v\right)_H + a(u,v)=(F-\lambda(1+y)u,v)_H,\qquad v\in V.
		 			 						\end{equation*}
		 			 		On the other hand we know that
		 			 			\begin{equation*}
		 			 		\begin{cases}
		 			 		-\left( \frac{\partial u}{\partial t },v -u  \right)_H + a(u,v-u)\geq 0 , \qquad \mbox{ a.e. in } [0,T]\qquad v\in V, \ v\geq \psi, \\u(T)=\psi(T),\\u \geq \psi \mbox{ a.e. in } [0,T]\times \R \times (0,\infty).
		 			 		\end{cases}
		 			 		\end{equation*}
		 			 	 From the previous relations we easily   deduce that $F-\lambda(1+y)u\geq 0$ a.e. and, taking $v=\psi$, that $(F-\lambda(1+y)u,\psi-u)_H=0.$ 
		 			 			Moreover, note that  the assumptions of Proposition \ref{martingale} are satisfied, so
		 			 			the process $(M_t)_{0\leq t\leq T}$ defined by
		 			 		\begin{equation}\label{M}
		 			 			M_t=e^{-\Lambda_t}u(t,X_t,Y_t)+\int_0^te^{-\Lambda_s}F(s,X_s,Y_s)ds,
		 			 		\end{equation}
		 			 		with $X_0=x, \, Y_0=y$	is a martingale
		 			 			for every $(x,y)\in \R\times [0,+\infty)$. Then, we deduce that
		 			 			 the process
		 			 				\[
		 			 				\tilde M_t= u(t,X_t,Y_t)+\int_0^{t} \left(F(s,X_s,Y_s)   -\lambda(1+Y_s)u(s,X_s,Y_s)\right)ds
		 			 				\]
		 			 			is a local martingale. In fact, from \eqref{M} we can write
		 			 			\begin{align*}
		 			 			d	\tilde M_t&=d\left[e^{\Lambda_t}M_t-e^{\Lambda_t}\int_0^te^{-\Lambda_s}F(s,X_s,Y_s)ds
		 			 			 \right]+ F(t,X_t,Y_t)dt-\lambda(1+Y_t)u(t,X_t,Y_t)dt\\&
		 			 			 =e^{\Lambda_t}dM_t+\Big[ \lambda(1+Y_t)e^{\Lambda_t}M_t-\lambda(1+Y_t)e^{\Lambda_t}\int_0^te^{-\Lambda_s}F(s,X_s,Y_s)ds\\&\quad-e^{\Lambda_t}e^{-\Lambda_t} F(t,X_t,Y_t)   + F(t,X_t,Y_t)-\lambda(1+Y_t)u(t,X_t,Y_t)\Big]dt\\
		 			 			 &=e^{\Lambda_t}dM_t.
		 			 			\end{align*}

		 			 			So, for any stopping time $\tau$ there exists an increasing sequence of stopping times $(\tau_n)_n$ such that $\lim_n\tau_n=\infty$ and 
		 			 			\begin{equation}\label{relationmgloc}
		 			 		\E_{x,y}[u(\tau\wedge \tau_n,X_{\tau\wedge \tau_n}, Y_{\tau\wedge \tau_n})]=u(0,x,y)-\E_{x,y}\left[\int_0^{\tau\wedge\tau_n}(  F(s,X_s,Y_s)-\lambda(1+Y_s)u(s,X_s,Y_s))ds\right].
		 			 			\end{equation}
		 			 			Since $F-\lambda(1+y)u\geq0$ we can pass to the limit in the right hand side of \eqref{relationmgloc} thanks to the monotone convergence theorem.  Recall now that an adapted right continuous process $(Z_{t})_{t \geq 0}$ is said to be 
		 			 			of class $\mathcal{D}$ if the family $(Z_{\tau})_{\tau \in \mathcal{T}_{0,\infty} }$, where $\mathcal{T}_{0,\infty}$ is the set of all stopping times with values in $[0,\infty)$,  is uniformly integrable. Moreover, recall that $0\leq u(t,x,y)\leq \Phi(x,y)=C_T(e^{x-\frac{\rho \kappa\theta}{\sigma}t}+e^{Ly-\kappa\theta L t})$.  The discounted and dividend adjusted price process $(e^{-(r-\delta)t}S_t)_t=(e^{X_t-\frac{\rho\kappa\theta}{\sigma}t})_t$  is a martingale (we refer to \cite{Kr} for an analysis of the martingale property in general affine stochastic volatility models), so we deduce that it is of class $\mathcal{D}$. On the other hand, we can prove that the process $(e^{LY_t-\kappa\theta t})_t$ is of class $\mathcal D$ following the same arguments used in Remark \ref{remarksemigroup}. Therefore, the process $(\Phi(t+s,X^{t,x,y}_s))_{s\in[t,T]}$ is of class $ \mathcal{D}$ for every $(t,x,y)\in[0,T]\times\R\times [0,\infty)$.
		 			 		  So we can pass to the limit in the left hand side of \eqref{relationmgloc} and  we get that $\lim_{n\rightarrow \infty }\E_{x,y}[u(\tau\wedge \tau_n,X_{\tau\wedge \tau_n}, Y_{\tau\wedge \tau_n})]= \E_{x,y}[u(\tau,X_{\tau}, Y_{\tau})]$. Therefore, passing to the limit as $n\rightarrow \infty$, we get
		 			 			\begin{equation*}
		 			 			\E_{x,y}[u(\tau,X_{\tau}, Y_{\tau})]=u(0,x,y)-\E_{x,y}\left[\int_0^{\tau} (  F(s,X_s,Y_s)-\lambda(1+Y_s)u(s,X_s,Y_s))ds\right], 
		 			 			\end{equation*}
		 			 			for every $\tau\in\mathcal{T}_{[0,T]}$.
		 			Recall that $F -\lambda(1+y)u\geq0$, so the process $ u(t,X_t,Y_t)$ is actually a supermartingale. Since $u\geq \psi$, we deduce directly from the definition of Snell envelope that $u(t,X_t,Y_t) \geq u^*(t,X_t,Y_t)  $ a.e. for $t\in [0,T]$. 
		 			
		 			In order to show the opposite inequality, we consider the so called continuation region 
		 			$$
		 			\mathcal{C}=\{ (t,x,y)\in [0,T)\times \R\times [0,\infty): u(t,x,y)>\psi(t,x,y)   \},
		 			$$ 
		 			its $t$-sections
		 			$$
		 			\mathcal{ C}_t=\{(x,y)\in 	\R\times [0,\infty) : (t,x,y)\in \mathcal C  \}, \qquad t\in [0,T),
		 			$$
		 			and the stopping time 
		 			\begin{align*}
		 			\tau_t=\inf\{ s \geq t : (s,X_s,Y_s)\notin \mathcal{C} \}=\inf\{ s \geq t : u(s,X_s,Y_s)=\psi(s,X_s,Y_s) \}.
		 		\end{align*}
		 			Note that $u(x,X_s,Y_s)>\psi(s,X_s,Y_s)$ for $t\leq s < \tau_t$. Moreover, recall  that $(F-\lambda(1+y)u,\psi-u)=0$ a.e., so $Leb\{ (x,y)\in \mathcal{C}_t :  F-\lambda(1+y)u \neq 0 \}= 0 \, dt$ a.e.. Since the two dimensional diffusion $(X,Y)$ has a density, we deduce that  $\E\left[ F(s,X_s,Y_s)-\lambda(1+Y_s)u(s,X_s,Y_s)\mathbf{1}_{\{(X_s,Y_s)\in \mathcal C_s \}}  \right]=0$, and so $F(s,X_s,Y_s)-\lambda(1+Y_s)u(s,X_s,Y_s) = 0 $ $ds,\ d\P- a.e. $ on $\{s<\tau_t\}$. Therefore,
		 			$$
		 			\E\left[    u(\tau_t,X_{\tau_t},Y_{\tau_t}) \right]= 	\E\left[    u(t,X_t,Y_t) \right],
		 			$$
		 			and, since $u(\tau_t,X_{\tau_t},Y_{\tau_t}) = \psi(\tau_t,X_{\tau_t},Y_{\tau_t})$ thanks to the continuity of $u$ and $\psi$,
		 			$$
		 			E\left[  u(t,X_t,Y_t) \right]= 	\E\left[    \psi(\tau_t,X_{\tau_t},Y_{\tau_t}) \right]\leq 	 	\E\left[    u^*(t,X_{t},Y_{t}) \right],
		 			$$
		 			so that $u(t,X_t,Y_t) = u^*(t,X_t,Y_t)  $  a.e..  With the same arguments  we can prove that $u(t,x,y)=u^*(t,x,y)$ and this concludes the proof.
		 		\end{proof}
		
		 		\subsubsection {Weaker assumptions on $\psi$}
		 		The last step is to establish the equality $u=u^*$ under weaker assumptions on $\psi$, so proving Theorem \ref{theorem2}. 
		 		\begin{proof}[Proof of Theorem \ref{theorem2}]
		 		First assume that there exists a sequence $(\psi_n)_ {n	\in \N}$ of continuous functions on $[0,T]	\times \R\times[0,\infty)$ which converges uniformly to $\psi$ and such that, for each $n\in \N$,  $\psi_n$ satisfies the assumptions of Proposition \ref{identification}.
		 			For every $n\in \N$,  we set  $u_n=u_n(t,x,y)$     the  unique solution of the variational inequality \eqref{variationalinequality} with final condition $u_n(T,x,y)=\psi_n(T,x,y)$ and $u_n^*(t,x,y)= \sup_{\tau\in\T_{t,T}}\E[\psi_n(\tau,X_\tau^{t,x,y},Y^{t,y}_\tau)]$. Then, thanks to Proposition \ref{identification}, for every $n\in\N$ we have
		 			$$
		 			u_n(t,x,y)=u_n^*(t,x,y) \qquad \mbox{ on } [0,T]\times \bar{\mathcal{O}}.
		 			$$
		 			Now, the left hand side  converges to $u(t,x,y)$ thanks to the Comparison Principle. As regards the right hand side, 
		 			$$
		 			\sup_{\tau \in \mathcal{T}_{t,T}}\E \left[   \psi_n(\tau,X_\tau^{t,x,y}, Y_\tau^{t,x,y})     \right]\rightarrow \sup_{\tau \in \mathcal{T}_{t,T}} \E \left[ e^{-r(\tau-t)}  \psi(\tau,X_\tau^{t,x,y}, Y_\tau^{t,x,y})     \right]
		 			$$
		 			thanks to the uniform convergence of $\psi_n$ to $\psi$.

		 			Therefore, it is enough to prove that, if $\psi$ satisfies Assumption  $\mathcal{H}^*$, then it is the uniform limit of  a sequence of functions $\psi_n$ which satisfy the assumptions of Proposition \ref{identification}. 
		 			This can be done following the very same arguments of \cite[Lemma 3.3]{JLL} so we omit  the technical details (see \cite{T}).
		\end{proof}

\end{document}